\documentclass{article}
\usepackage{arxiv}
\usepackage[utf8]{inputenc} % allow utf-8 input
\usepackage[T1]{fontenc}    % use 8-bit T1 fonts
\usepackage{url}            % simple URL typesetting
\usepackage{booktabs}       % professional-quality tables
\usepackage{amsfonts}       % blackboard math symbols
\usepackage{nicefrac}       % compact symbols for 1/2, etc.
\usepackage{microtype}      % microtypography
\usepackage{lipsum}
\usepackage{graphicx}
\usepackage{amsmath}
\usepackage{amssymb}
\usepackage{amsthm}
\theoremstyle{definition}
\newtheorem{theorem}{Theorem}[section]
\theoremstyle{definition}

\theoremstyle{definition}
\newtheorem{lemma}[theorem]{Lemma}
\theoremstyle{definition}
\newtheorem{proposition}[theorem]{Proposition}
\theoremstyle{definition}
\newtheorem{definition}{Definition}[section]
\theoremstyle{definition}
\newtheorem{assumption}{Assumption}[section]
\theoremstyle{remark}

\usepackage{arydshln}
\usepackage{enumerate}
\usepackage{enumitem}
\usepackage{mathtools}
\usepackage{xcolor}
\usepackage{hyperref}
\hypersetup{
    colorlinks = true,
    citecolor = black,
    linkcolor = black,
    urlcolor = {brown},
}
\usepackage{bm}
\usepackage{float}
\usepackage{algorithm}
\usepackage[noend]{algpseudocode}
\usepackage{soul,color}
\usepackage{caption}
\usepackage{subcaption}
\usepackage{multirow}
\usepackage{multicol}
\DeclareMathOperator*{\argmin}{\arg\!\min}
\DeclareMathOperator*{\argmax}{\arg\!\max}
\allowdisplaybreaks
\title{Distributed and Asynchronous Algorithms for N-block Convex Optimization with Coupling Constraints}
\author{
  Run~Chen\\
  School of Industrial Engineering\\
  Purdue University\\
  West Lafayette, IN 47906\\
  \texttt{chen885@purdue.edu}\\
  \And
  Andrew L.~Liu\\
  School of Industrial Engineering\\
  Purdue University\\
  West Lafayette, IN 47906\\
  \texttt{andrewliu@purdue.edu}\\
}
\begin{document}
\maketitle
%
%
%
%
%\begin{abstract}
%\end{abstract}
%
%
%
%
%%%%%%%%%%%%%%%%%%%%%%%%%%%%%%%%%%%%%%%%%%%%%%%%%%%%%%%%%%%%%%%%%%%%%%%%%%%%%%%%
\section{Introduction}
In this work, we focus on designing a distributed algorithm for solving block-separable convex optimization problems with both linear and nonlinear coupling constraints. More specifically, we consider the following problem:
\begin{equation}\label{eq: GP}
\begin{aligned}
\underset{\mathbf{x}_1, \dots, \mathbf{x}_N}{\text{minimize}} \quad & f(\mathbf{x}_1,\ldots, \mathbf{x}_N) = \sum_{i=1}^N f_i(\mathbf{x}_i) \\
\text{subject to} \quad &\mathbf{x}_i \in \mathcal{X}_i, \quad i = 1, \dots, N, \\
&\sum_{i=1}^N A_i \mathbf{x}_i = \mathbf{b}, \\
&g_j(\mathbf{x}_1, \dots, \mathbf{x}_N) = \sum_{i=1}^N g_{ji}(\mathbf{x}_i) \leq 0, \quad j = 1, \dots, M,
\end{aligned}
\end{equation}
where each block of decision variable $\mathbf{x}_i \in \mathbb{R}^{n_i}$ is constrained by a closed and convex set $\mathcal{X}_i \subset \mathbb{R}^{n_i}$ for all $i = 1, \dots, N$, and $\sum_{i=1}^N n_i = n$. The objective function $f: \mathbb{R}^n \to \mathbb{R}$ is block-separable, and each function $f_i: \mathbb{R}^{n_i} \to \mathbb{R}$ is assumed to be continuous and convex for all $i = 1, \dots, N$. All blocks $\mathbf{x}_i$'s are coupled in a linear equality constraint, where each $A_i \in \mathbb{R}^{m \times n_i}$ is a given matrix for all $i = 1, \dots, N$, and $\mathbf{b} \in \mathbb{R}^m$ is a given vector. All blocks $\mathbf{x}_i$'s are also coupled in a system of nonlinear inequality constraints, where each constraint function $g_j: \mathbb{R}^n \to \mathbb{R}$ is also block-separable and each function $g_{ji}: \mathbb{R}^{n_i} \to \mathbb{R}$ is assumed to be continuous and convex for all $i = 1, \dots, N$ and $j = 1, \dots, M$. 
A wide range of application problems can be mathematically formulated as optimization problems of the form \eqref{eq: GP}, arising from the areas including optimal control \cite{yu2017simple}, network optimization \cite{hallac2015network}, statistical learning \cite{boyd2011distributed} and etc.
\par The alternating direction method of multipliers (ADMM) \cite{boyd2011distributed}, as well as its variants \cite{deng2017parallel,wang2014parallel,wang2013solving}, is an efficient distributed algorithm for solving convex block-separable optimization problems with linear coupling constraints, but problems of \eqref{eq: GP} with nonlinear coupling constraints can not be directly handled by ADMM-typed algorithms.
%a distributed algorithm is proposed in \cite{yu2017simple}.
%
%contribution
\par To overcome the above-mentioned limitations of the ADMM-typed algorithms, we first extend the $2$-block Predictor Corrector Proximal Multiplier Method (PCPM) algorithm to solve an $N$-block convex optimization problem with both linear and nonlinear coupling constraints. We further extend the $N$-block PCPM algorithm to an asynchronous iterative scheme, where a maximum tolerable delay is allowed for each distributed unit, and apply it to solve an $N$-block convex optimization problem with general linear coupling constraints.
% outline.
\par The remainder of the chapter is organized as follows. In Section~\ref{sec: N-PCPM}, we present an extended $N$-block PCPM algorithm for solving general constrained $N$-block convex optimization problems. We first establish global convergence under mild assumptions, and then prove the linear convergence rate with slightly stronger assumptions. In Section~\ref{sec: AN-PCPM}, we further extend the N-block PCPM algorithm to an asynchronous scheme with the bounded delay assumption. We establish both convergence and global sub-linear convergence rate under the conditions of strong convexity. Section~\ref{sec: Num} presents the numerical results of applying the proposed algorithms to solve a graph optimization problem arising from an application of housing price prediction. Finally, Section \ref{sec: Conclusion} concludes this chapter with discussions of the limitations of the algorithms and possible future research directions. 
\section{Extending the PCPM Algorithm to Solving General Constrained N-block Convex Optimization Problems}\label{sec: N-PCPM}
\subsection{PCPM Algorithm}\label{sec: 2-PCPM}
To present our distributed algorithm, we first briefly describe the original PCPM algorithm \cite{chen1994proximal} to make this paper self-contained. For this purpose, it suffices to consider a 2-block linearly constrained convex optimization problem:  
\begin{equation}\label{eq: 2_block CP Problem Form}
\begin{aligned}
\underset{\mathbf{x}_1 \in \mathbb{R}^{n_1},\ \mathbf{x}_2 \in \mathbb{R}^{n_2}}{\text{minimize}} \quad &f_1(\mathbf{x}_1) + f_2(\mathbf{x}_2) \\
\text{subject to} \quad &A_1 \mathbf{x}_1 + A_2 \mathbf{x}_2 = \mathbf{b}, \quad (\bm{\lambda})
\end{aligned}
\end{equation}
where $f_1: \mathbb{R}^{n_1} \to (-\infty, +\infty]$ and $f_2: \mathbb{R}^{n_2} \to (-\infty, +\infty]$ are closed proper convex functions, $A_1 \in \mathbb{R}^{m \times n_1}$ and $A_2 \in \mathbb{R}^{m \times n_2}$ are full row-rank matrices, $\mathbf{b} \in \mathbb{R}^m$ is a given vector, and $\bm{\lambda} \in \mathbb{R}^m$ is the corresponding  Lagrangian multiplier associated with the linear equality constraint. The classic Lagrangian function $\mathcal{L}: \mathbb{R}^{n_1} \times \mathbb{R}^{n_2} \times \mathbb{R}^m \to \mathbb{R}$ is defined as:
\begin{equation}
\mathcal{L}(\mathbf{x}_1, \mathbf{x}_2, \bm{\lambda}) = f_1(\mathbf{x}_1) + f_2(\mathbf{x}_2) + \bm{\lambda}^T (A_1 \mathbf{x}_1 + A_2 \mathbf{x}_2 - \mathbf{b}).  
\end{equation}
It is well-known that for a convex problem of the specific form in \eqref{eq: 2_block CP Problem Form} (where the linear constraint qualification automatically holds), finding an optimal solution is equivalent to finding a saddle point $(\mathbf{x}_1^*, \mathbf{x}_2^*, \bm{\lambda}^*)$ such that $\mathcal{L}(\mathbf{x}_1^*, \mathbf{x}_2^*, \bm{\lambda}) \leq \mathcal{L}(\mathbf{x}_1^*, \mathbf{x}_2^*, \bm{\lambda}^*) \leq \mathcal{L}(\mathbf{x}_1, \mathbf{x}_2, \bm{\lambda}^*)$. To find such a saddle point, a simple dual decomposition algorithm can be applied to $\mathcal{L}(\mathbf{x}_1, \mathbf{x}_2, \bm{\lambda})$. More specifically, at each iteration $k$, given a fixed Lagrangian multiplier $\bm{\lambda}^k$, the primal decision variables $(\mathbf{x}_1^{k+1}, \mathbf{x}_2^{k+1})$ can be  obtained, in parallel, by minimizing $\mathcal{L}(\mathbf{x}_1, \mathbf{x}_2, \bm{\lambda}^k)$. Then a dual update $\bm{\lambda}^{k+1} = \bm{\lambda}^k + \rho (A_1 \mathbf{x}_1^{k+1} + A_2 \mathbf{x}_2^{k+1} - \mathbf{}b)$ is performed. 
\par While the above algorithmic idea is simple, it is well-known that convergence cannot be established without more restrictive assumptions, such as strict convexity of $f_1$ and $f_2$ (e.g., Theorem~26.3 in \cite{rockafellar2015convex}). One approach to overcome such difficulties is the proximal point algorithm, which obtains $(\mathbf{x}_1^{k+1},\mathbf{x}_2^{k+1})$  by minimizing the proximal augmented Lagrangian function defined as $\mathcal{L}_{\rho}(\mathbf{x}_1, \mathbf{x}_2, \bm{\lambda}^k) := \mathcal{L}(\mathbf{x}_1, \mathbf{x}_2, \bm{\lambda}^k) + \frac{\rho}{2} \lVert A_1 \mathbf{x}_1 + A_2 \mathbf{x}_2 - \mathbf{b} \rVert_2^2 + \frac{1}{2 \rho} \lVert \mathbf{x}_1 - \mathbf{x}_1^k \rVert_2^2 + \frac{1}{2 \rho} \lVert \mathbf{x}_2 - \mathbf{x}_2^k \rVert_2^2$. The parameter $\rho$ is given, which determines the step-size for updating both primal and dual variables in each iteration, and plays a key role in the convergence of the overall algorithm.  The primal minimization step now becomes:
\begingroup
\begin{align}
(\mathbf{x}_1^{k+1}, \mathbf{x}_2^{k+1}) = \underset{\mathbf{x}_1 \in \mathbb{R}^{n_1}, \mathbf{x}_2 \in \mathbb{R}^{n_2}}{\argmin} &f_1(\mathbf{x}_1) + f_2(\mathbf{x}_2) + (\bm{\lambda}^k)^T (A_1 \mathbf{x}_1 + A_2 \mathbf{x}_2 - \mathbf{b}) \nonumber \\
+ &\frac{\rho}{2} \lVert A_1 \mathbf{x}_1 + A_2 \mathbf{x}_2 - \mathbf{b} \rVert_2^2 \nonumber \\
+ &\frac{1}{2 \rho} \lVert \mathbf{x}_1 - \mathbf{x}_1^k \rVert_2^2 + \frac{1}{2 \rho} \lVert \mathbf{x}_2 - \mathbf{x}_2^k \rVert_2^2. \label{eq: proximal point primal minimization}
\end{align}
\endgroup
With \eqref{eq: proximal point primal minimization}, however, $\mathbf{x}_1^{k+1}$ and $\mathbf{x}_2^{k+1}$ can no longer be obtained in parallel due to the augmented term $\lVert A_1 \mathbf{x}_1 + A_2 \mathbf{x}_2 - \mathbf{b} \rVert_2^2$. To overcome this difficulty, the PCPM algorithm introduces a predictor variable $\bm{\mu}^{k+1}$:
\begin{equation}\label{eq:PCPM_Predictor}
\bm{\mu}^{k+1} \coloneqq \bm{\lambda}^k + \rho (A_1 \mathbf{x}_1^k + A_2 \mathbf{x}_2^k - \mathbf{b}).
\end{equation}
Using the predictor variable, the optimization in \eqref{eq: proximal point primal minimization} can be approximated as:
\begingroup
\begin{align}
(\mathbf{x}_1^{k+1}, \mathbf{x}_2^{k+1}) = \underset{\mathbf{x}_1 \in \mathbb{R}^{n_1}, \mathbf{x}_2 \in \mathbb{R}^{n_2}}{\argmin} &f_1(\mathbf{x}_1) + f_2(\mathbf{x}_2) + (\bm{\mu}^{k+1})^T (A_1 \mathbf{x}_1 + A_2 \mathbf{x}_2 - \mathbf{b}) \nonumber \\
+ &\frac{1}{2 \rho} \lVert \mathbf{x}_1 - \mathbf{x}_1^k \rVert_2^2 + \frac{1}{2 \rho} \lVert \mathbf{x}_2 - \mathbf{x}_2^k \rVert_2^2, \label{eq: PCPM primal minimization}
\end{align}
\endgroup
which allows $\mathbf{x}_1^{k+1}$ and $\mathbf{x}_2^{k+1}$ to be obtained in parallel again. After solving \eqref{eq: PCPM primal minimization}, the PCPM algorithm updates the dual variable as follows: 
\begin{equation}\label{eq:PCPM_Corrector}
\bm{\lambda}^{k+1} = \bm{\lambda}^k + \rho (A_1 \mathbf{x}_1^{k+1} + A_2 \mathbf{x}_2^{k+1} - \mathbf{b}),
\end{equation} 
which is referred to as a corrector update. 
\subsection{N-block PCPM Algorithm for General Constrained Convex Optimization Problems}
In the $2$-block PCPM algorithm presented in Section~\ref{sec: 2-PCPM}, we observe that the introduction of the predictor variable eliminates the quadratic term in the proximal augmented Lagrangian function and make the primal minimization step parallelizable again, which is a major difference from the ADMM algorithm. For an $N$-block convex optimization problem with additional nonlinear coupling constraints:
\begin{equation}\label{eq: GP again}
\begin{aligned}
\underset{\mathbf{x}_1, \dots, \mathbf{x}_N}{\text{minimize}} \quad &\sum_{i=1}^N f_i(\mathbf{x}_i) \\
\text{subject to} \quad &\mathbf{x}_i \in \mathcal{X}_i, \quad i = 1, \dots, N, \\
&\sum_{i=1}^N A_i \mathbf{x}_i = \mathbf{b}, \quad (\bm{\lambda}) \\
&\sum_{i=1}^N g_{ji}(\mathbf{x}_i) \leq 0, \quad j = 1, \dots, M, \quad (\mu_j)
\end{aligned}
\end{equation}
the potential coupling caused by the quadratic term $\lVert \sum_{i=1}^N g_{ji}(\mathbf{x}_i) \rVert_2^2$ could be even worse. Using the same technique of introducing the predictor variable, we extend the $2$-block PCPM algorithm for solving \eqref{eq: GP again}.
\par First, we make a blanket assumption on problem \eqref{eq: GP again} throughout this chapter that the Slater's constraint qualification (CQ) holds.
\begin{assumption}[Slater's CQ]\label{assp: N-PCPM Slater}
There exists a point $(\bar{\mathbf{x}}_1, \dots, \bar{\mathbf{x}}_N)$ such that
$$
\Bigg\{\bar{\mathbf{x}}_i \in relint(\mathcal{X}_i), \quad i = 1, \dots, N \Bigg| \begin{array}{ll}\sum_{i=1}^N A_i \bar{\mathbf{x}}_i = \mathbf{b} \\ \sum_{i=1}^N g_{ji}(\bar{\mathbf{x}}_i) < 0, \quad j = 1, \dots, M\end{array}\Bigg\},
$$
where $relint(\mathcal{X}_i)$ denotes the relative interior of the convex set $\mathcal{X}_i$ for all $i = 1, \dots, N$.
\end{assumption}
To apply the PCPM algorithm to \eqref{eq: GP again}, at each iteration $k$, with a given primal dual pair, $\big(\mathbf{x}_1^k, \dots, \mathbf{x}_N^k, \bm{\lambda}^k, \bm{\mu}^k \coloneqq (\mu_1^k, \dots, \mu_M^k)^T\big)$, we start with a predictor update:
\begin{equation}\label{eq: N-block PCPM predictor update}
\begin{aligned}
&\textbf{(predictor update)}: \\
&\bm{\gamma}^{k+1} = \bm{\lambda}^k +\rho\Big(\sum_{i=1}^N A_i \mathbf{x}_i^k - \mathbf{b}\Big), \\
&\nu_j^{k+1} = \Pi_{\mathbb{R}_+}\Big[\mu_j^k + \rho \sum_{i=1}^N g_{ji}(\mathbf{x}_i^k)\Big], \quad j = 1, \dots, M,
\end{aligned}
\end{equation}
where $\Pi_{\mathbb{Z}}(\mathbf{z})$ denotes the projection of a vector $\mathbf{z} \in \mathbb{R}^n$ onto a set $\mathbb{Z} \subset \mathbb{R}^n$, and $\mathbb{R}_+$ refers to the set of all non-negative real numbers.
\par After the predictor update, we update the primal variables $(\mathbf{x}_1^{k+1}, \dots, \mathbf{x}_N^{k+1})$ by minimizing the Lagrangian function $\mathcal{L}(\mathbf{x}_1, \dots, \mathbf{x}_N, \bm{\gamma}^{k+1}, \bm{\nu}^{k+1})$ evaluated at the predictor variable $\big(\bm{\gamma}^{k+1}, \bm{\nu}^{k+1} \coloneqq (\nu_1^{k+1}, \dots, \nu_M^{k+1})^T\big)$, plus the proximal terms. The primal minimization step can be decomposed as
\begin{equation}\label{eq: N-block PCPM primal minimization}
\begin{aligned}
&\textbf{(primal minimization)}: \\
&\mathbf{x}_i^{k+1} = \underset{\mathbf{x}_i \in \mathcal{X}_i}{\argmin} \enspace f_i(\mathbf{x}_i) + (\bm{\gamma}^{k+1})^T A_i \mathbf{x}_i + \sum_{j=1}^M \nu_j^{k+1} g_{ji}(\mathbf{x}_i) + \frac{1}{2 \rho} \lVert \mathbf{x}_i - \mathbf{x}_i^k \rVert_2^2, \\
&i = 1, \dots, N.
\end{aligned}
\end{equation}
A corrector update is then performed for each Lagrangian multiplier $(\bm{\lambda}^{k+1}, \bm{\mu}^{k+1})$: 
\begin{equation}\label{eq: N-block PCPM corrector update}
\begin{aligned}
&\textbf{(dual corrector)}: \\
&\bm{\lambda}^{k+1} = \bm{\lambda}^k +\rho\Big(\sum_{i=1}^N A_i \mathbf{x}_i^{k+1} - \mathbf{b}\Big), \\
&\mu_j^{k+1} = \Pi_{\mathbb{R}_+}\Big[\mu_j^k + \rho \sum_{i=1}^N g_{ji}(\mathbf{x}_i^{k+1})\Big], \quad j = 1, \dots, M.
\end{aligned}
\end{equation}
\par The overall structure of $N$-block PCPM algorithm is presented in Algorithm~\ref{alg: N-PCPM} below.
\begin{algorithm}[ht]
%\footnotesize
\caption{$N$-PCPM}\label{alg: N-PCPM}
\begin{algorithmic}[1]
\State \textbf{Initialization} choose an arbitrary starting point $(\mathbf{x}_1^0, \dots, \mathbf{x}_N^0, \bm{\lambda}^0, \bm{\nu}^0)$.
\State $k \gets 0$.
\While{termination conditions are not met}
\State (Predictor update) \par \textbf{update} $(\bm{\gamma}^{k+1}, \bm{\nu}^{k+1})$ according to \eqref{eq: N-block PCPM predictor update};
\State (Primal minimization) \par \textbf{update} $(\mathbf{x}_1^{k+1}, \dots, \mathbf{x}_N^{k+1})$ according to \eqref{eq: N-block PCPM primal minimization};
\State (Corrector update) \par \textbf{update} $(\bm{\lambda}^{k+1}, \bm{\mu}^{k+1})$ according to \eqref{eq: N-block PCPM corrector update};
\State $k \leftarrow k+1$
\EndWhile
\State \textbf{return} $(\mathbf{x}_1^k, \dots, \mathbf{x}_N^k, \bm{\lambda}^k, \bm{\nu}^k)$.
\end{algorithmic}
\end{algorithm}
\subsection{Convergence Analysis}\label{sec: N-PCPM convergence}
We make the following additional assumptions on the optimization problem \eqref{eq: GP again}.
\begin{assumption}[Lipschitz Continuity]
For all $j = 1 \dots M$ and $i = 1 \dots N$, each single-valued function $g_{ij}: \mathcal{X}_i \to \mathbb{R}$ is Lipschitz continuous with modulus of $L_{ji}$, i.e., $\lVert g_{ji}(\mathbf{x}_1) - g_{ji}(\mathbf{x}_2)\rVert_2 \leq L_{ji} \lVert \mathbf{x}_1 - \mathbf{x}_2 \rVert_2$ for any $\mathbf{x}_1, \mathbf{x}_2 \in \mathcal{X}_i$. 
\end{assumption}
\begin{assumption}[Existence of a Saddle Point]\label{assp: N-PCPM Saddle Point}
For the Lagrangian function of \eqref{eq: GP again}: 
\begin{equation}\label{eq: GP Lagrangian}
\mathcal{L}(\mathbf{x}_1, \dots, \mathbf{x}_N, \bm{\lambda}, \bm{\mu}) \coloneqq \sum_{i=1}^N f_i(\mathbf{x}_i) + \bm{\lambda}^T \Big(\sum_{i=1}^N A_i \mathbf{x}_i - \mathbf{b}\Big) + \sum_{j=1}^M \mu_j \sum_{i=1}^N g_{ji}(\mathbf{x}_i),
\end{equation}
we assume that a saddle point $(\mathbf{x}_1^*, \dots, \mathbf{x}_N^*, \bm{\lambda}^*, \bm{\mu}^*)$ exists; that is, for any $\mathbf{x}_i \in \mathcal{X}_i$, $i = 1, \dots, N$, $\bm{\lambda} \in \mathbb{R}^m$ and $\bm{\mu} \in \mathbb{R}_+^M$, 
\begin{equation}
\mathcal{L}(\mathbf{x}_1^*, \dots, \mathbf{x}_N^*, \bm{\lambda}, \bm{\mu}) \leq \mathcal{L}(\mathbf{x}_1^*, \dots, \mathbf{x}_N^*, \bm{\lambda}^*, \bm{\mu}^*) \leq \mathcal{L}(\mathbf{x}_1, \dots, \mathbf{x}_N, \bm{\lambda}^*, \bm{\mu}^*).
\end{equation}
\end{assumption}
Note that coupled with the blanket Assumption~\ref{assp: N-PCPM Slater} that Slater's CQ holds for the optimization problem \eqref{eq: GP again}, the above assumption is equivalent to say that an optimal solution of \eqref{eq: GP again} is assumed to exist (see Corollary 28.3.1 in \cite{rockafellar2015convex}).
\par Next, we derive some essential lemmas for constructing the main convergence proof. The following lemma is due to Proposition~$6$ in \cite{rockafellar1976monotone}, and for completeness, we provide the detailed statements below.
\begin{lemma}[Inequality of Proximal Minimization Point]\label{lem: Proximal Minimization}
Given a closed, convex set $\mathbb{Z} \subset \mathbb{R}^n$, and a continuous, convex function $F: \mathbb{Z} \to \mathbb{R}$. With a given  point $\bar{\mathbf{z}} \in \mathbb{Z}$ and a positive number $\rho > 0$, if $\widehat{\mathbf{z}}$ is a proximal minimization point; i.e. $\widehat{\mathbf{z}} \coloneqq \arg \underset{\mathbf{z} \in \mathbb{Z}}{\min}\ F(\mathbf{z}) + \frac{1}{2 \rho} \lVert \mathbf{z} - \bar{\mathbf{z}} \rVert_2^2$, then we have that 
\begin{equation}
2 \rho [F(\widehat{\mathbf{z}}) - F(\mathbf{z})] \leq \lVert \bar{\mathbf{z}} - \mathbf{z} \rVert_2^2 - \lVert \widehat{\mathbf{z}} - \mathbf{z} \rVert_2^2 - \lVert \widehat{\mathbf{z}} - \bar{\mathbf{z}} \rVert_2^2, \quad \forall \mathbf{z} \in \mathbb{Z}.
\end{equation}
\end{lemma}
\begin{proof}
Denote $\Phi(\mathbf{z}) = F(\mathbf{z}) + \frac{1}{2 \rho} \lVert \mathbf{z} - \bar{\mathbf{z}} \rVert_2^2$. By the definition of $\widehat{\mathbf{z}}$, we have $\mathbf{0} \in \partial_{\mathbf{z}} \Phi(\widehat{\mathbf{z}})$. Since $\Phi(\mathbf{z})$ is strongly convex with modulus $\frac{1}{\rho}$, it follows that $2 \rho \big[\Phi(\mathbf{z}) - \Phi(\widehat{\mathbf{z}})\big] \geq \lVert \widehat{\mathbf{z}} - \mathbf{z} \rVert_2^2$ for any $\mathbf{z} \in \mathbb{Z}$.
\end{proof}
\begin{lemma} \label{lem: Dual}
The update steps \eqref{eq: N-block PCPM predictor update} and \eqref{eq: N-block PCPM corrector update} are equivalent to obtaining proximal minimization points as follows:
\begin{equation}
\begin{aligned}
(\bm{\gamma}^{k+1}, \bm{\nu}^{k+1}) = \underset{\bm{\lambda} \in \mathbb{R}^m, \bm{\mu} \in \mathbb{R}_+^M}{\argmin} \enspace -&\mathcal{L}(\mathbf{x}_1^k, \dots, \mathbf{x}_N^k, \bm{\lambda}, \bm{\mu}) \\
+ &\frac{1}{2 \rho} \lVert \bm{\lambda} - \bm{\lambda}^k \rVert_2^2 + \frac{1}{2 \rho} \lVert \bm{\mu} - \bm{\mu}^k \rVert_2^2, \\
(\bm{\lambda}^{k+1}, \bm{\mu}^{k+1}) = \underset{\bm{\lambda} \in \mathbb{R}^m, \bm{\mu} \in \mathbb{R}_+^M}{\argmin} \enspace -&\mathcal{L}(\mathbf{x}_1^{k+1}, \dots, \mathbf{x}_N^{k+1}, \bm{\lambda}, \bm{\mu}) \\
+ &\frac{1}{2 \rho} \lVert \bm{\lambda} - \bm{\lambda}^k \rVert_2^2 + \frac{1}{2 \rho} \lVert \bm{\mu} - \bm{\mu}^k \rVert_2^2.
\end{aligned}
\end{equation}
\end{lemma}
Similar to the convergence analysis of the PCPM algorithm in \cite{chen1994proximal}, we now establish some fundamental estimates of the distance at each iteration $k$ between the solution point $(\mathbf{x}_1^{k+1}, \dots, \mathbf{x}_N^{k+1}, \bm{\lambda}^{k+1}, \bm{\mu}^{k+1})$ and the saddle point $(\mathbf{x}_1^*, \dots, \mathbf{x}_N^*, \bm{\lambda}^*, \bm{\mu}^*)$. 
\begin{proposition}\label{prop: distance}
Let $(\mathbf{x}_1^*, \dots, \mathbf{x}_N^*, \bm{\lambda}^*, \bm{\mu}^*)$ be a saddle point of the optimization problem \eqref{eq: GP again}. For all $k \geq 0$, we have that
\begin{equation}\label{eq: N-PCPM Primal Distance}
\begin{aligned}
\sum_{i=1}^N \lVert \mathbf{x}_i^{k+1} - \mathbf{x}_i^* \rVert_2^2 \leq &\sum_{i=1}^N \lVert \mathbf{x}_i^k - \mathbf{x}_i^* \rVert_2^2 - \sum_{i=1}^N \lVert \mathbf{x}_i^{k+1} - \mathbf{x}_i^k \rVert_2^2 \\
+ &2 \rho \Big[(\bm{\lambda}^* - \bm{\gamma}^{k+1})^T \sum_{i=1}^N A_i \mathbf{x}_i^{k+1} + \sum_{j=1}^M (\mu_j^* - \nu_j^{k+1}) \sum_{i=1}^N g_{ji}(\mathbf{x}_i^{k+1})\Big] \\
\end{aligned}
\end{equation}
and
\begin{equation}\label{eq: N-PCPM Dual Distance}
\begin{aligned}
&\lVert \bm{\lambda}^{k+1} - \bm{\lambda}^* \rVert_2^2 + \lVert \bm{\mu}^{k+1} - \bm{\mu}^* \rVert_2^2 \leq \lVert \bm{\lambda}^k - \bm{\lambda}^* \rVert_2^2 + \lVert \bm{\mu}^k - \bm{\mu}^* \rVert_2^2 \\
- &\lVert \bm{\gamma}^{k+1} - \bm{\lambda}^{k+1} \rVert_2^2 - \lVert \bm{\nu}^{k+1} - \bm{\mu}^{k+1} \rVert_2^2 - \lVert \bm{\gamma}^{k+1} - \bm{\lambda}^k \rVert_2^2 - \lVert \bm{\nu}^{k+1} - \bm{\mu}^k \rVert_2^2 \\
+ &2 \rho \Big[(\bm{\gamma}^{k+1} - \bm{\lambda}^{k+1})^T \sum_{i=1}^N A_i \mathbf{x}_i^k + \sum_{j=1}^M (\nu_j^{k+1} - \mu_j^{k+1}) \sum_{i=1}^N g_{ji}(\mathbf{x}_i^k) \\
&\hspace*{12pt} +(\bm{\lambda}^{k+1} - \bm{\lambda}^*)^T \sum_{i=1}^N A_i \mathbf{x}_i^{k+1} + \sum_{j=1}^M (\mu_j^{k+1} - \mu_j^*) \sum_{i=1}^N g_{ji}(\mathbf{x}_i^{k+1})\Big]
\end{aligned}
\end{equation}
\end{proposition}
\begin{proof}
The details of the proof are provided in Section~\ref{app: N-PCPM Proof_1}.
\end{proof}

\begin{theorem}[Global Convergence]\label{thm: N-PCPM global convergence}
Assume that Assumption~\ref{assp: N-PCPM Slater} to Assumption~\ref{assp: N-PCPM Saddle Point} hold. Given a scalar $0 < \epsilon < 1$, choose a step size $\rho$ satisfying
\begin{equation}\label{eq: step size}
0 < \rho \leq \min\Bigg\{\frac{1 - \epsilon}{A_{max} + M L_{max}}, \frac{1 - \epsilon}{N A_{max}}, \frac{1 - \epsilon}{N L_{max}}\Bigg\},
\end{equation}
where $A_{max} \coloneqq \max_{i=1}^N\{\lVert A_i \rVert_2\}$, and $L_{max} \coloneqq \max_{j=1}^M\big\{\max_{i=1}^N\{L_{ji}\}\big\}$. \\
Let $\{\mathbf{x}_1^k, \dots, \mathbf{x}_N^k, \bm{\lambda}^k, \bm{\mu}^k\}$ be the sequence generated by Algorithm~\ref{alg: N-PCPM}, with an arbitrary point $(\mathbf{x}_1^0, \dots, \mathbf{x}_N^0, \bm{\lambda}^0, \bm{\mu}^0)$; then the sequence converges globally to a saddle point $(\mathbf{x}_1^* \dots \mathbf{x}_N^*, \bm{\lambda}^*, \bm{\mu}^*)$ of the optimization problem \eqref{eq: GP again}.
\end{theorem}
\begin{proof}
Please see Section~\ref{app: N-PCPM Proof_2} for details.
\end{proof}
\par To establish convergence rate, we need to make an additional assumption on Problem \eqref{eq: GP again}, as follows.
\begin{assumption}[Lipschitz Inverse Mapping]\label{assp: lip_inverse_mapping}
Assume that there exists a unique saddle point $(\mathbf{x}_1^*, \dots, \mathbf{x}_N^*, \bm{\lambda}^*, \bm{\mu}^*)$ such that, given an inverse mapping $\mathcal{S}^{-1}: \mathbb{R}^n \times \mathbb{R}^m \times \mathbb{R}^M \to \mathbb{R}$:
\begin{equation}
\begin{aligned}
&\mathcal{S}^{-1}(\mathbf{u}_1, \dots, \mathbf{u}_N, \mathbf{v}) \\
= &\arg \underset{(\mathbf{x}_1, \dots, \mathbf{x}_N) \in \prod_{i=1}^N \mathcal{X}_i}{\min} \quad \underset{\bm{\lambda} \in \mathbb{R}^m, \bm{\mu} \in \mathbb{R}_+^M}{\max} \Bigg\{\mathcal{L}(\mathbf{x}_1, \dots, \mathbf{x}_N, \bm{\lambda}, \bm{\mu}) - \sum_{i=1}^N \mathbf{x}_i^T \mathbf{u}_i + \bm{\lambda}^T \mathbf{v} + \bm{\mu}^T \mathbf{w}\Bigg\},
\end{aligned}
\end{equation}
and a fixed positive real number $\tau > 0$, we have 
$$
\sum_{i=1}^N \lVert \mathbf{x}_i - \mathbf{x}_i^* \rVert_2^2 + \lVert \bm{\lambda} - \bm{\lambda}^* \rVert_2^2 + \lVert \bm{\mu} - \bm{\mu}^* \rVert_2^2 \leq a^2 \Big(\sum_{i=1}^N \lVert \mathbf{u}_i \rVert_2^2 + \lVert \mathbf{v} \rVert_2^2 + \lVert \mathbf{w} \rVert_2^2\Big)
$$
for some $a \geq 0$, whenever the point $(\mathbf{x}_1, \dots, \mathbf{x}_N, \bm{\lambda}, \bm{\mu}) \in \mathcal{S}^{-1}(\mathbf{u}_1, \dots, \mathbf{u}_N, \mathbf{v}, \mathbf{w})$ and 
$$
\sum_{i=1}^N \lVert \mathbf{u}_i \rVert_2^2 + \lVert \mathbf{v} \rVert_2^2 + \lVert \mathbf{w} \rVert_2^2 \leq \tau^2.
$$
\end{assumption}
The above assumption states that the inverse mapping $\mathcal{S}^{-1}$ is Lipschitz continuous at the origin with modulus $a$. By Proposition~2 in \cite{rockafellar1976augmented}, to obtain a plausible condition for the Lipschitz continuity of $\mathcal{S}^{-1}$ at the origin, we appeal to the \textit{strong second-order conditions for optimality} which are comprised of the following properties:
\begin{enumerate}[label=(\roman*)]
\item There is a saddle point $(\mathbf{x}_1^*, \dots, \mathbf{x}_N^*, \bm{\lambda}^*, \bm{\mu}^*)$ of the optimization problem \eqref{eq: GP again} such that $\mathbf{x}_i^* \in int(\mathcal{X}_i)$ for all $i = 1 \dots N$, where $int(\mathcal{X}_i)$ denotes the interior of the convex set $\mathcal{X}_i$. Moreover, for all $i = 1 \dots N$ and $j = 1, \dots, M$, function $g_{ji}(\mathbf{x}_i)$ is twice continuously differentiable on a neighborhood of $\mathbf{x}_i^*$.
\item Let $\mathcal{J}$ denote the set of active constraint indices at the point of $(\mathbf{x}_1^*, \dots, \mathbf{x}_N^*)$: $\mathcal{J} = \Big\{j = 1, \dots, M \Big| \sum_{i=1}^N g_{ji}(\mathbf{x}_i^*) = 0\Big\}$. Then $\mu_j^* > 0$ for all $j \in \mathcal{J}$, and $\Big\{\sum_{i=1}^N \nabla_{\mathbf{x}_i} g_{ji}(\mathbf{x}_i^*)\Big\}_{j \in \mathcal{J}} \cup \Big\{\sum_{i=1}^N A_i \mathbf{x}_i^*\Big\}$ forms a linearly independent set.
\item The Hessian matrix $H \coloneqq \nabla_{(\mathbf{x}_1, \dots, \mathbf{x}_N)}^2 \mathcal{L}(\mathbf{x}_1^*, \dots, \mathbf{x}_N^*, \bm{\lambda}^*, \bm{\mu}^*)$ satisfies $\mathbf{y}^T H \mathbf{y} > 0$ for any
$$
\mathbf{y} \coloneqq \left(\begin{array}{c}\mathbf{y}_1 \\[-6pt] \vdots \\[-8pt] \mathbf{y}_N\end{array}\right) \in \Bigg\{\mathbf{y} \not= \mathbf{0} \Bigg| \begin{array}{l} \sum_{i=1}^N A_i \mathbf{y}_i = \mathbf{0} \\ \sum_{i=1}^N \mathbf{y}_i^T \big[\nabla_{\mathbf{x}_i} g_{ji}(\mathbf{x}_i^*)\big] = 0, \quad \forall j \in \mathcal{J}\end{array}\Bigg\}.
$$
\end{enumerate}
We now establish the linear convergence rate of Algorithm~\ref{alg: N-PCPM}.
\begin{theorem}[Linear Convergence Rate]\label{thm: N-PCPM linear convergence}
Assume that Assumption~\ref{assp: N-PCPM Slater} to Assumption~\ref{assp: lip_inverse_mapping} hold. Let $\rho$ satisfy \eqref{eq: step size} and let $\{\mathbf{x}_1^k, \dots, \mathbf{x}_N^k, \bm{\lambda}^k, \bm{\mu}^k\}$ be a sequence generated by Algorithm~\ref{alg: N-PCPM}, with an arbitrary starting point $(\mathbf{x}_1^0, \dots, \mathbf{x}_N^0, \bm{\lambda}^0, \bm{\mu}^0)$; then the sequence converges linearly to the unique saddle point $(\mathbf{x}_1^*, \dots, \mathbf{x}_N^*, \bm{\lambda}^*, \bm{\mu}^*)$. More specifically, there exists an integer $\bar{k}$ such that, for all $k \geq \bar{k}$, we have:
\begin{equation}
\begin{aligned}
&\sum_{i=1}^N \lVert \mathbf{x}_i^{k+1} - \mathbf{x}_i^* \rVert_2^2 + \lVert \bm{\lambda}^{k+1} - \bm{\lambda}^* \rVert_2^2 + \lVert \bm{\mu}^{k+1} - \bm{\mu}^* \rVert_2^2 \\
\leq &\theta^2 \Big(\sum_{i=1}^N \lVert \mathbf{x}_i^k - \mathbf{x}_i^* \rVert_2^2 + \lVert \bm{\lambda}^k - \bm{\lambda}^* \rVert_2^2 + \lVert \bm{\mu}^k - \bm{\mu}^* \rVert_2^2\Big),
\end{aligned}
\end{equation}
where $0 < \theta < 1$.
\end{theorem}
\begin{proof}
Please see Section~\ref{app: N-PCPM Proof_3} for details.
\end{proof}
\subsection{Numerical Experiments}
Consider the following $20$-dimensioned, block-separable convex optimization problem with $17$ nonlinear coupling constraints, modeling a decentralized planning of an economic system and suggested by \cite{asaadi1973computational}:
\begingroup
\allowdisplaybreaks
\begin{align}
\underset{x_1, \dots, x_{20}}{\text{minimize}} \quad &x_1^2 + x_2^2 + x_1 x_2 - 14 x_1 - 16 x_2 + (x_3 - 10)^2 + 4 (x_4 - 5)^2 + (x_5 - 3)^2 \nonumber \\
+ &2 (x_6 - 1)^2 + 5 x_7^2 + 7 (x_8 - 11)^2 + 2 (x_9 - 10)^2 + (x_{10} - 7)^2 + (x_{11} - 9)^2 \nonumber \\
+ &10 (x_{12} - 1)^2 + 5 (x_{13} - 7)^2 + 4 (x_{14} - 14)^2 + 27 (x_{15} - 1)^2 + x_{16}^4 \nonumber \\
+ &(x_{17} - 2)^2 + 13 (x_{18} - 2)^2 + (x_{19} - 3)^2 + x_{20}^2 + 95 \nonumber \\
\text{subject to} \quad &3 (x_1 - 2)^2 + 4 (x_2 - 3)^2 + 2 x_3^2 - 7 x_4  - 120 \leq  0 \nonumber \\
&5 x_1^2 + 8 x_2 + (x_3 - 6)^2 - 2 x_4 - 40 \leq 0 \nonumber \\
&\frac{1}{2} (x_1 - 8)^2 + 2 (x_2 - 4)^2 + 3x_5^2 - x_6 - 30 \leq 0 \nonumber \\
&x_1^2 + 2 (x_2 - 2)^2 - 2 x_1 x_2 + 14 x_5 - 6 x_6 \leq 0 \nonumber \\
&4 x_1 + 5 x_2 - 3 x_7 + 9 x_8  - 105 \leq 0 \nonumber \\
&10 x_1 - 8 x_2 - 17 x_7 + 2 x_8 \leq 0 \nonumber \\
&3 x_1 + 6 x_2 + 12 (x_9 - 8)^2 - 7 x_{10} \leq 0 \nonumber \\
&-8 x_1 + 2 x_2 + 5 x_9 - 2 x_{10} - 12 \leq 0 \nonumber \\
&x_1 + x_2 + 4 x_{11} - 21 x_{12} \leq 0 \nonumber \\
&x_1^2 + 15 x_{11} - 8 x_{12} - 28 \leq 0 \nonumber \\
&4 x_1 + 9 x_2 + 5 x_{13}^2 - 9 x_{14} - 87 \leq 0 \nonumber \\
&3 x_1 + 4 x_2 + 3 (x_{13} - 6)^2 - 14 x_{14} - 10 \leq 0 \nonumber \\
&14 x_1^2 + 35 x_{15} - 79 x_{16} - 92 \leq 0 \nonumber \\
&15 x_2^2 + 11 x_{15} - 61 x_{16} - 54 \leq 0 \nonumber \\
&5 x_1^2 + 2 x_2 + 9 x_{17}^4 - x_{18} - 68 \leq 0 \nonumber \\
&x_1^2 - x_2 + 19 x_{19} - 20 x_{20} + 19 \leq 0 \nonumber \\
&7 x_1^2 + 5 x_2^2 + x_{19}^2 - 30 x_{20} \leq 0 \label{eq: mid-sized P}
\end{align}
\endgroup
A minimum function value of $133.723$ can be obtained at the point of ($2.18$, $2.35$, $8.77$, $5.07$, $0.99$, $1.43$, $1.33$, $9.84$, $8.29$, $8.37$, $2.28$, $1.36$, $6.08$, $14.17$, $1.00$, $0.66$, $1.47$, $2.00$, $1.05$, $2.06$).\footnote{The solution is obtained using the nonlinear constrained optimization solver \textbf{filter} of Neos Solver at \href{https://neos-server.org/neos/solvers/}{https://neos-server.org/neos/solvers/}.} Applying Algorithm~\ref{alg: N-PCPM} and decomposing the original problem into $19$ small sub-problems, we achieve the following convergence results, presented in Figure~\ref{fig: test 1}. 
\begin{figure}[ht]
\centering
\begin{subfigure}{.36\textwidth}
\centering
\includegraphics[width=1.0\linewidth]{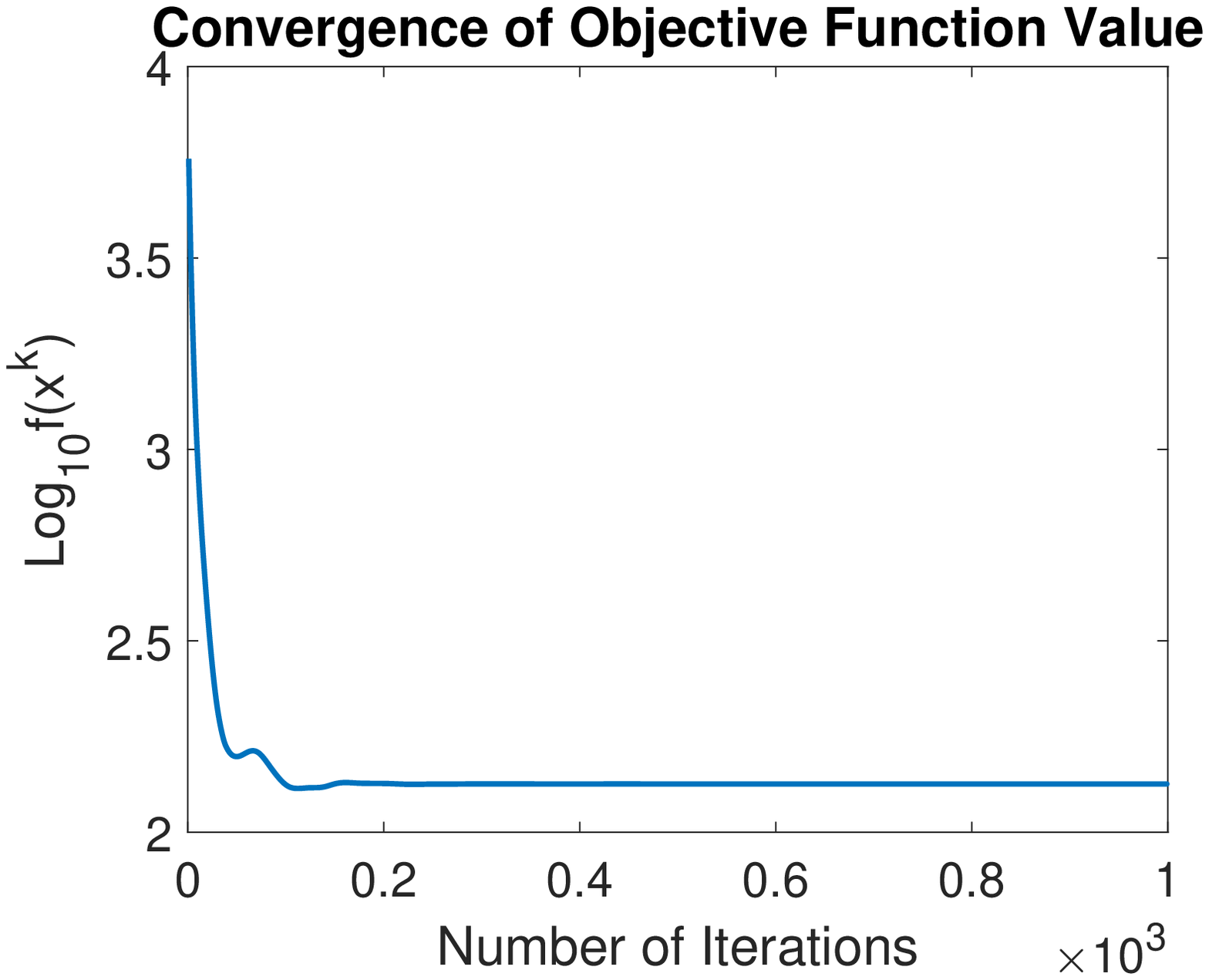}
\end{subfigure}%
\begin{subfigure}{.36\textwidth}
\centering
\includegraphics[width=1.0\linewidth]{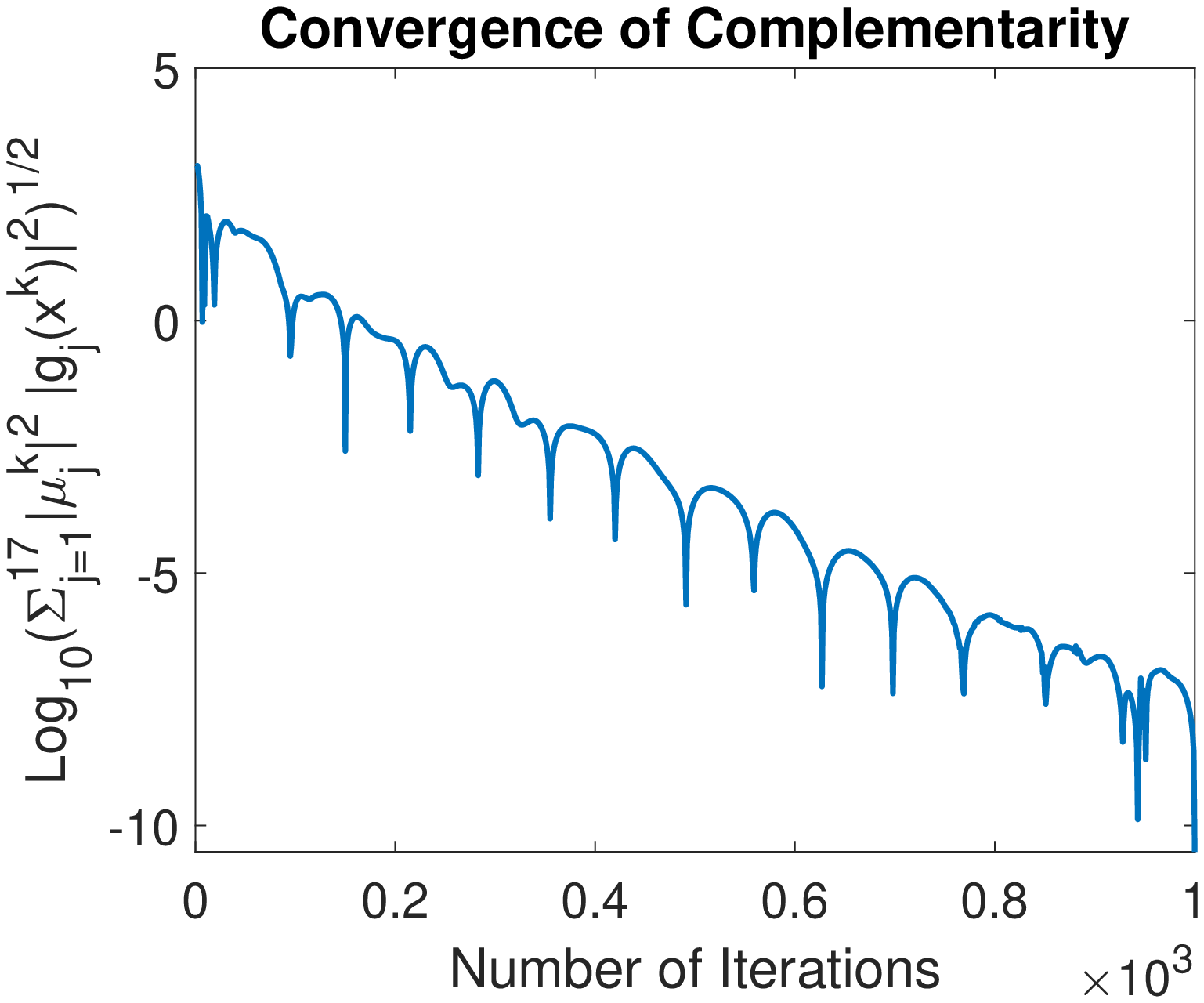}
\end{subfigure}
\begin{subfigure}{.36\textwidth}
\centering
\includegraphics[width=1.0\linewidth]{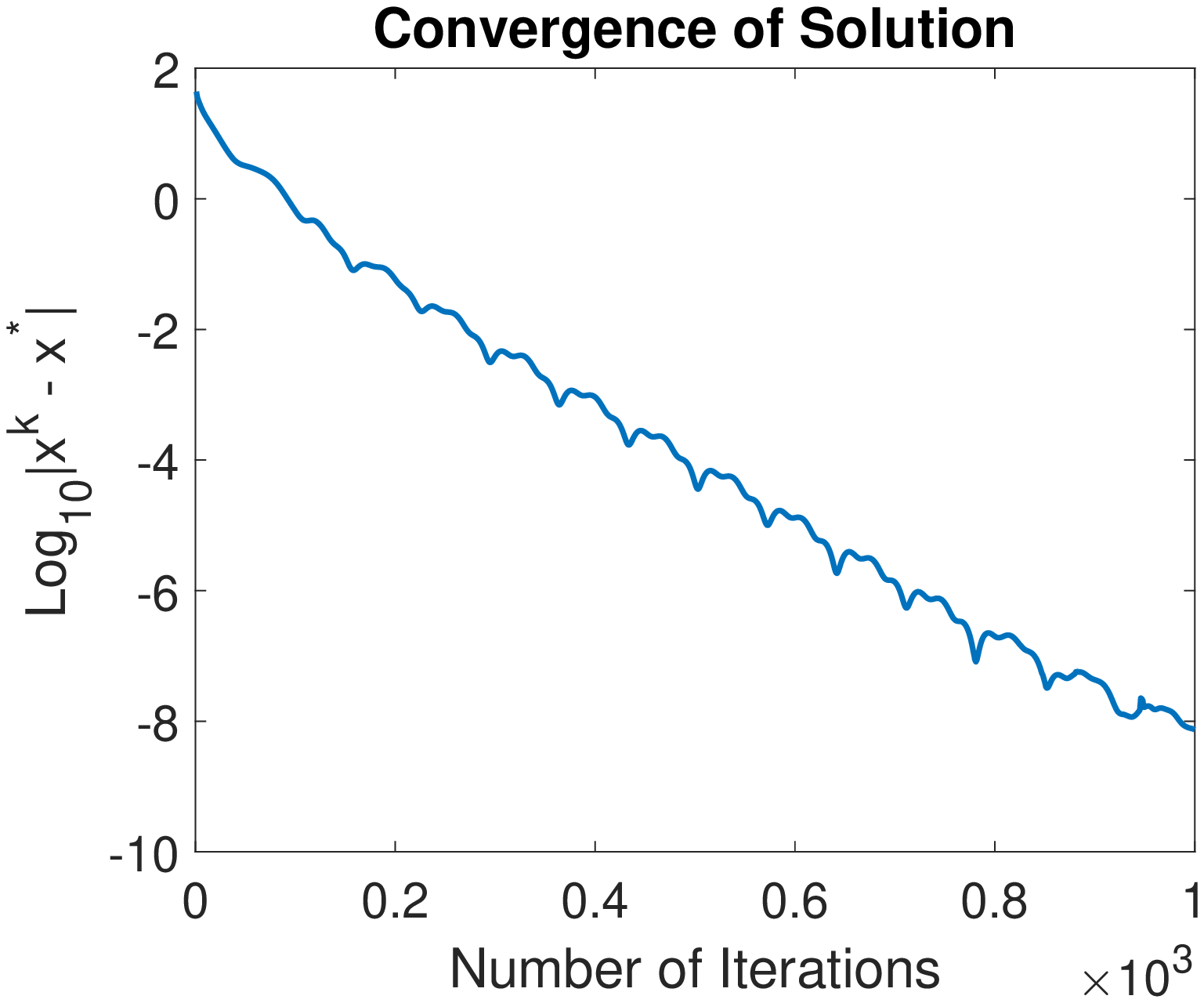}
\end{subfigure}%
\begin{subfigure}{.36\textwidth}
\centering
\includegraphics[width=1.0\linewidth]{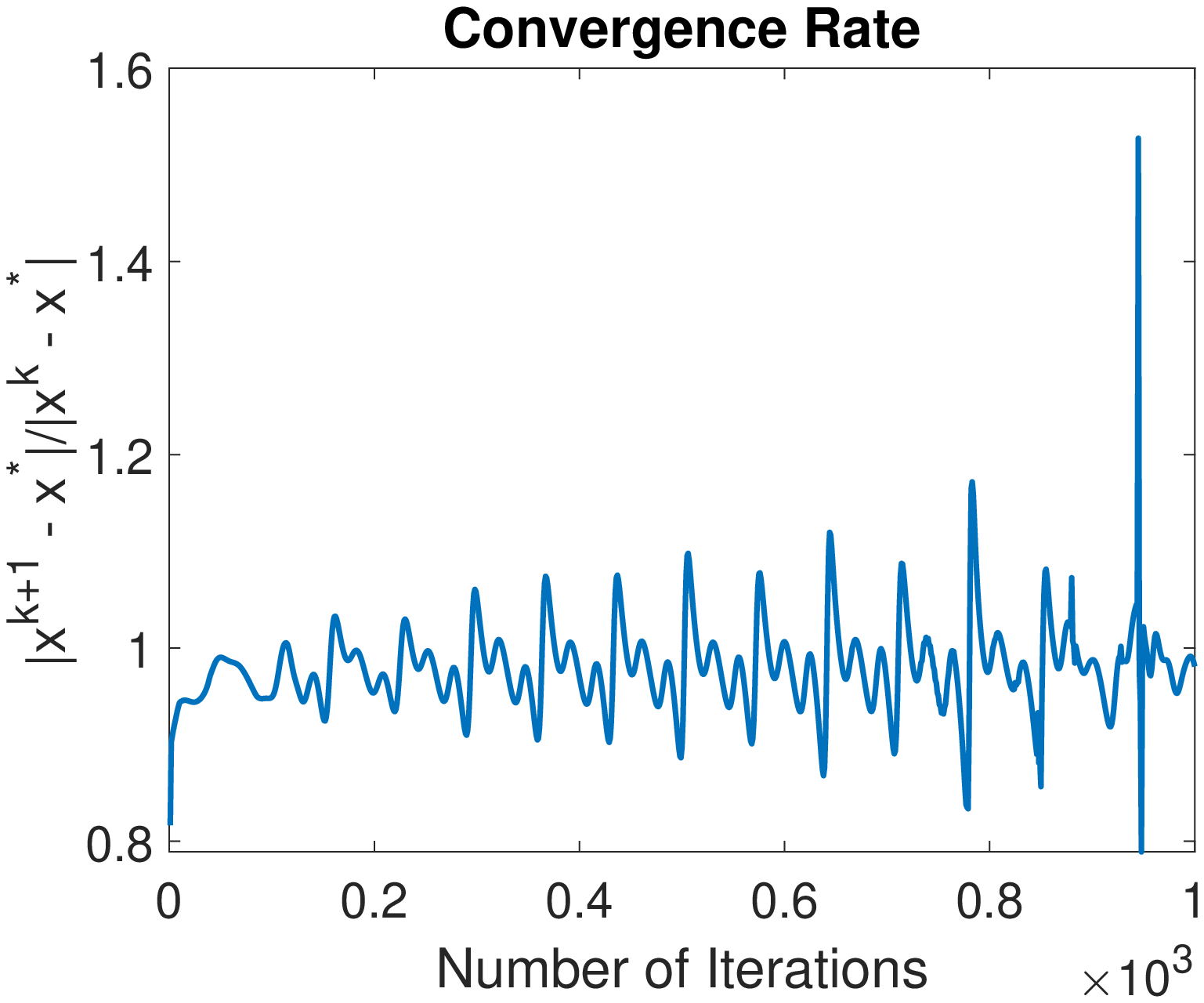}
\end{subfigure}
\caption{Convergence Results of Applying Algorithm~\ref{alg: N-PCPM} to Problem \eqref{eq: mid-sized P} with $\rho = 0.009$.}
\label{fig: test 1}%\vspace*{-10pt}
\end{figure}
\par Next, by replacing the term of $x_{16}^4$ in the objective function of \eqref{eq: mid-sized P} with $x_{16}^2$ and the term of $x_{17}^4$ in the $15$-th constraint with $x_{17}^2$, we make the modified problem satisfy Assumption~\ref{assp: lip_inverse_mapping}. A minimum function value of $133.687$ can be obtained at the point of ($2.18$, $2.34$, $8.76$, $5.07$, $0.99$, $1.43$, $1.34$, $9.84$, $8.30$, $8.36$, $2.27$, $1.36$, $6.08$, $14.17$, $1.00$, $0.64$, $2.00$, $2.00$, $1.04$, $2.06$). An additional linear convergence rate of applying Algorithm~\ref{alg: N-PCPM} is observed in the following convergence results, presented in Figure~\ref{fig: test 2}.
\begin{figure}[ht]
\centering
\begin{subfigure}{.36\textwidth}
\centering
\includegraphics[width=1.0\linewidth]{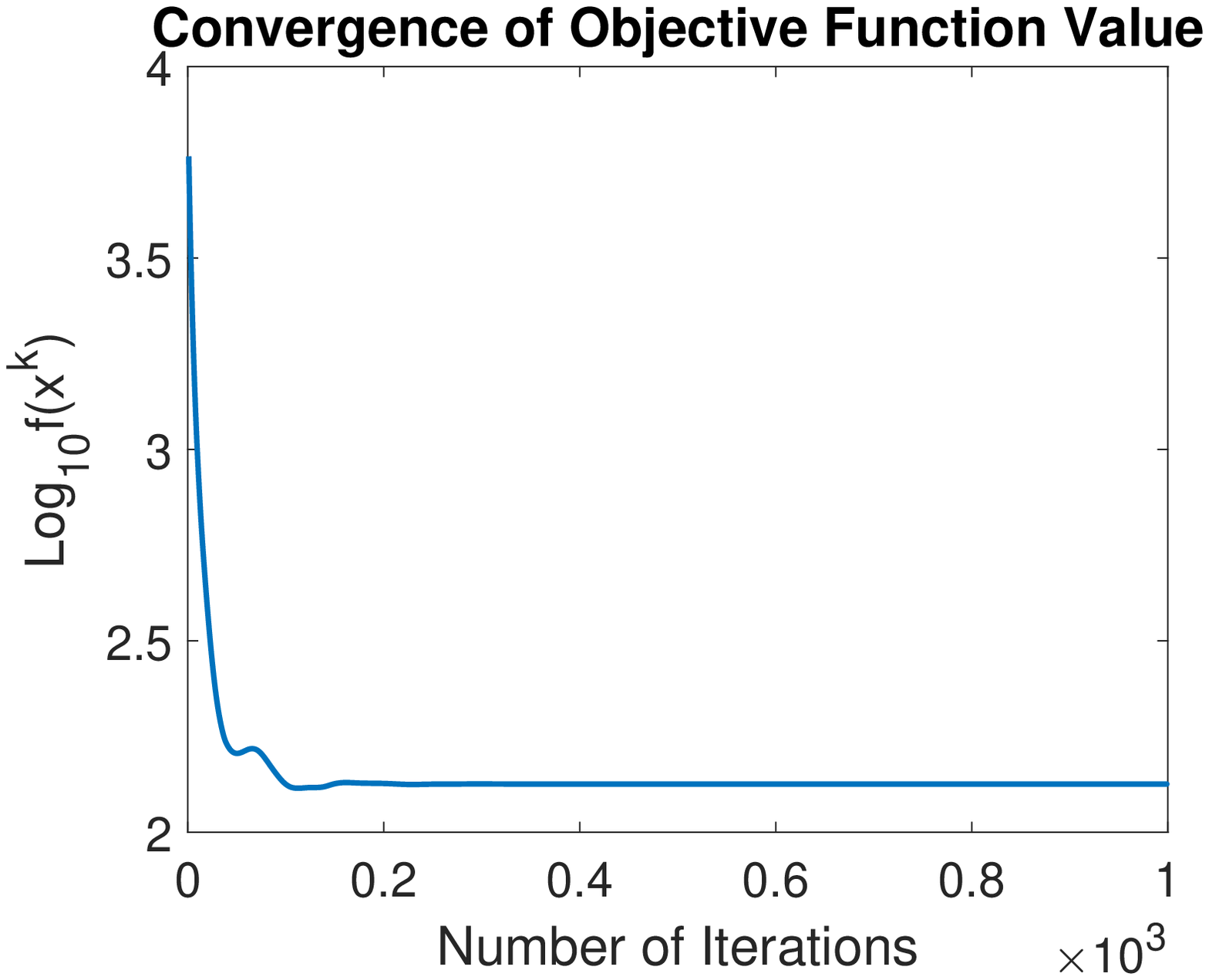}
\end{subfigure}%
\begin{subfigure}{.36\textwidth}
\centering
\includegraphics[width=1.0\linewidth]{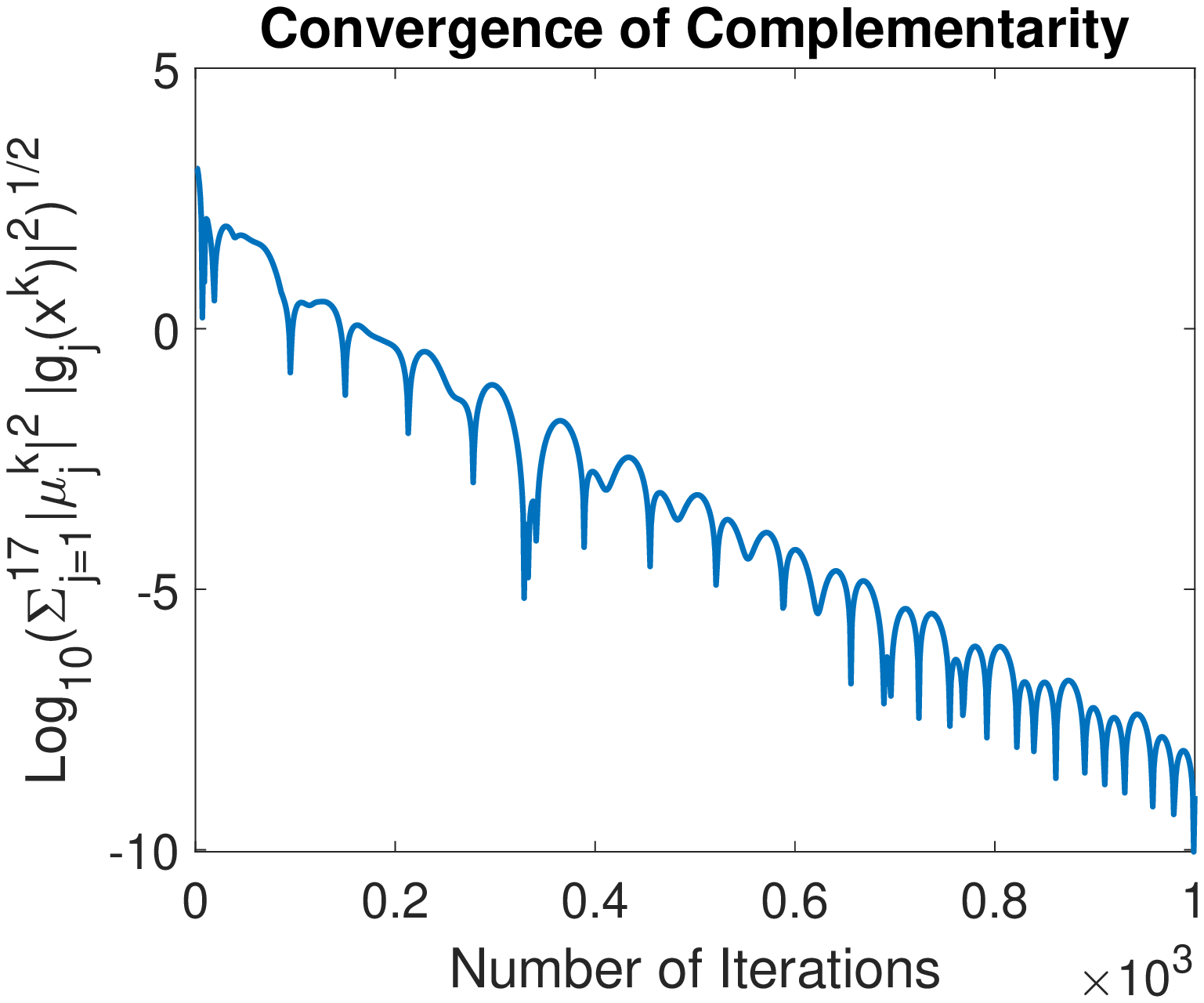}
\end{subfigure}
\begin{subfigure}{.36\textwidth}
\centering
\includegraphics[width=1.0\linewidth]{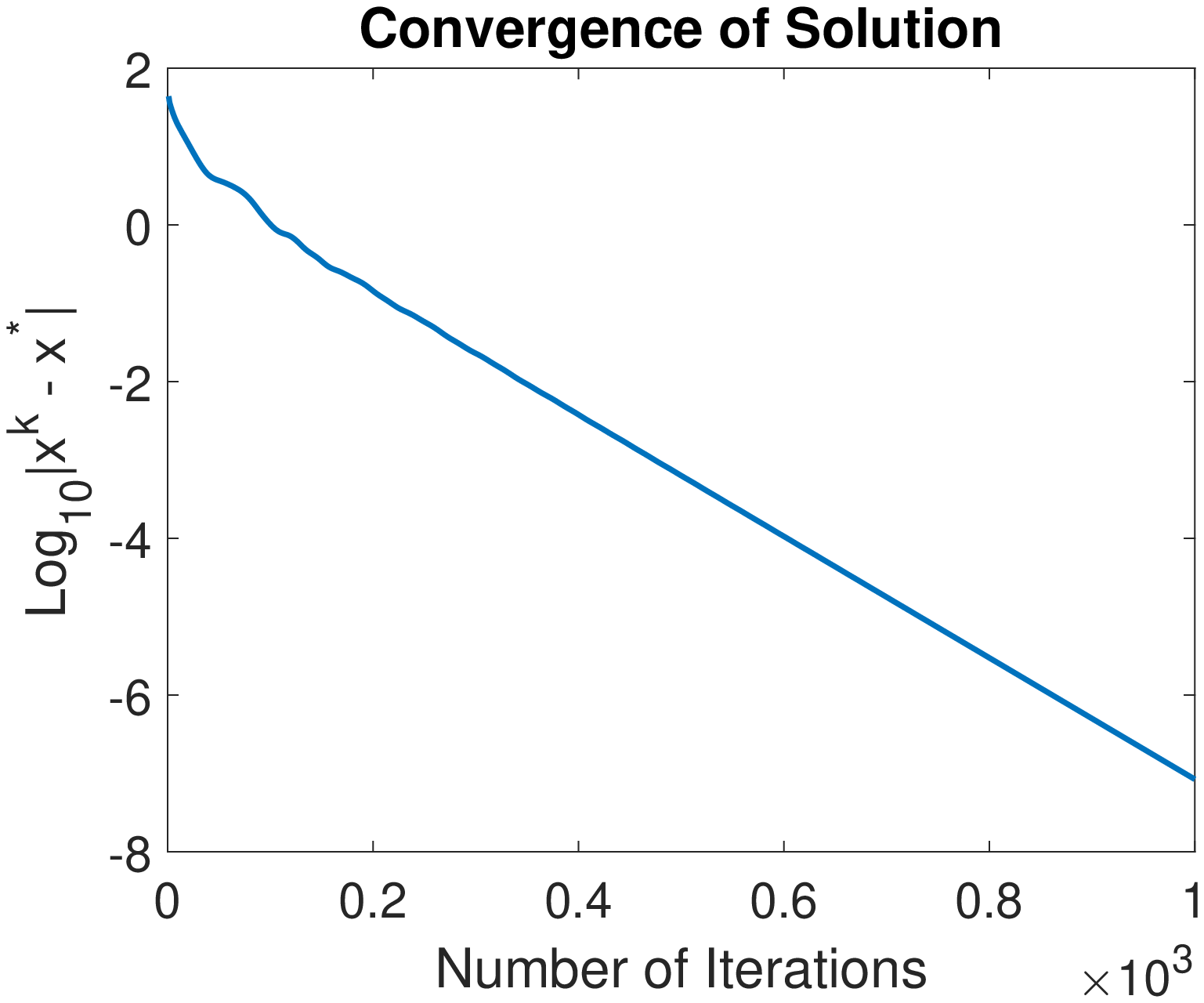}
\end{subfigure}%
\begin{subfigure}{.36\textwidth}
\centering
\includegraphics[width=1.0\linewidth]{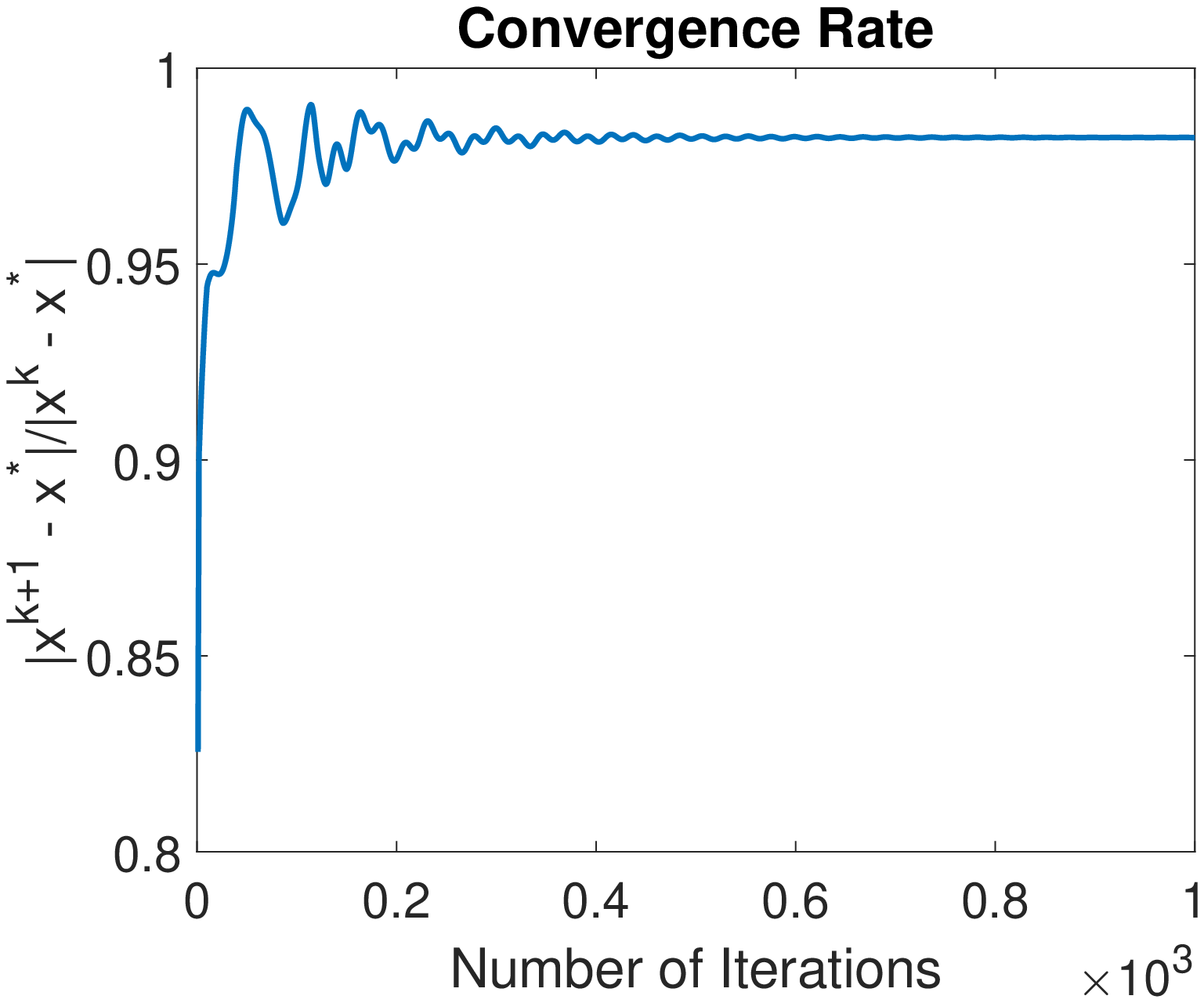}
\end{subfigure}
\caption{Convergence Results of Applying Algorithm~\ref{alg: N-PCPM} to the Modified Problem \eqref{eq: mid-sized P} with $\rho = 0.009$.}
\label{fig: test 2}%\vspace*{-10pt}
\end{figure}
\section{Extending the N-block PCPM Algorithm to an Asynchronous Scheme}\label{sec: AN-PCPM}
As a starting point, we first consider the following N-block convex optimization problem with only linear coupling constraints:
\begin{equation}\label{eq: N-LCP}
\begin{aligned}
\underset{\mathbf{x}_1, \dots, \mathbf{x}_N}{\text{minimize}} \quad &\sum_{i=1}^N f_i(\mathbf{x}_i) \\
\text{subject to} \quad &\mathbf{x}_i \in \mathcal{X}_i, \quad i = 1 \dots N, \\
&\sum_{i=1}^N A_i \mathbf{x}_i = \mathbf{b}. \\
\end{aligned}
\end{equation}
The decision variables, the objective function and the constraints are the same as in the optimization problem \eqref{eq: GP}. When applying Algorithm~\ref{alg: N-PCPM} to solving the above problem, each iteration can be interpreted as main-worker paradigm \cite{sahni1996master}, shown in Figure~\ref{fig: master_slave}.
\begin{figure}[ht]
\begin{center}\includegraphics[scale = 0.24]{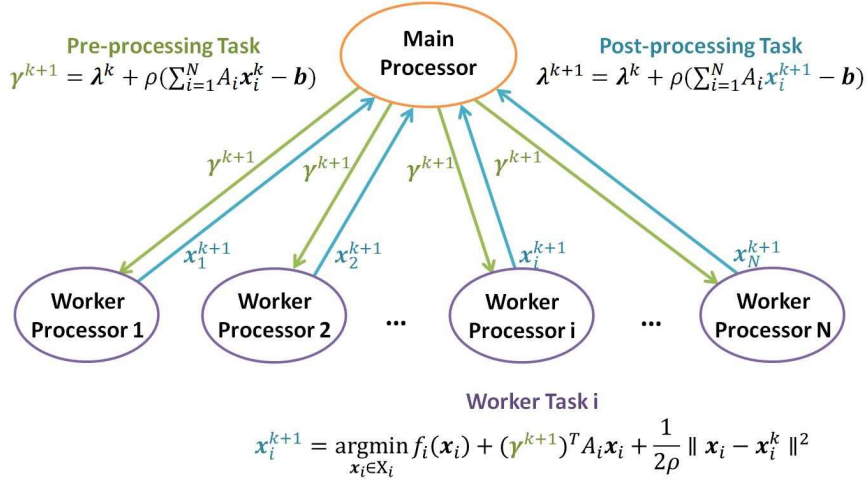}
\caption{Main-Worker Paradigm for Algorithm~\ref{alg: N-PCPM}.}
\label{fig: master_slave}
\end{center}
\end{figure}
At each iteration $k$, a predictor update of $\bm{\gamma}^{k+1}$ is first performed on a \textit{main processor} and is broadcast to each \textit{worker processor}, which is called a \textit{pre-processing task}. Upon receiving the updated predictor variable from the main processor, each worker processor solves the decomposed sub-problem in parallel and send its updated primal decision variable $\mathbf{x}_i^{k+1}$ back to the main processor, which is called a \textit{worker task}. After gathering all updated decision variables, a corrector update is then performed on the main processor, which is called a \textit{post-processing task}.
\par The speed of the algorithm is significantly limited by the slowest worker processor, since the post-processing task can not start until all worker tasks are finished and the results are sent back to the main processor. For large-scale problems, with the number of worker processors increasing, the issue of node synchronization can be a major concern for the performance of synchronous distributed algorithms. While in an asynchronous scheme, the main processor can proceed with only part of worker tasks finished. Figure~\ref{fig: syn_asyn}
\begin{figure}[ht]
\begin{center}
\includegraphics[scale = 0.3]{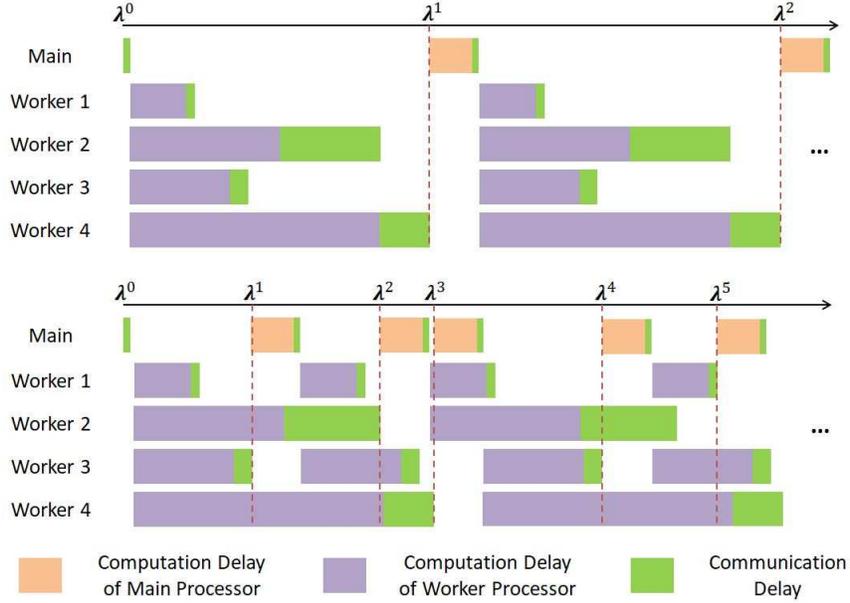}
\caption{Illustration of how asynchronous scheme iterates faster than synchronous scheme.}
\label{fig: syn_asyn}
\end{center}
\end{figure}
shows an example of $1$ main processor and $4$ worker processors with different lengths of computation and communication delays. In the asynchronous scheme, the main processor starts a new iteration whenever receives the results from at least $2$ worker processors, which leads to much faster iterations than the synchronous scheme. 
\par In this section, we extend the $N$-block PCPM algorithm to an asynchronous scheme to solve the linearly constrained $N$-block convex optimization problem \eqref{eq: N-LCP}. 
\subsection{Asynchronous N-block PCPM Algorithm for Convex Optimization Problems with Linear Coupling Constraints}
To achieve the convergence of the asynchronous $N$-block PCPM algorithm, similar to \cite{chang2016asynchronous1,chang2016asynchronous2}, we require that the asynchronous delay of each parallel worker processor is bounded. Let $k \geq 0$ denote the iteration index on the main processor. At each iteration $k$, let $\mathcal{A}_k \subseteq \{1, \dots, N\}$ denote the subset of worker processors from whom the main processor receives the updated decision variable $\widehat{\mathbf{x}}_i$, and let $\mathcal{A}_k^{\complement} \subseteq \{1, \dots, N\}$ denote the rest of the worker processors, whose information does not arrive.
\begin{definition}[Bounded Delay]\label{def: bounded delay}
Let an integer $\tau \geq 1$ denote the maximum tolerable delay. At any iteration $k \geq 0$, with a \textit{bounded delay}, it must holds that $i \in \mathcal{A}_k \cup \mathcal{A}_{k-1} \cup \cdots \cup \mathcal{A}_{k-\tau+1}$ for all $i = 1, \dots, N$. When $\tau = 1$, it's a synchronous scheme.
\end{definition}
At each each iteration $k$ on the main processor, all worker processors are divided into two sets $\mathcal{A}_k$ and $\mathcal{A}_k^{\complement}$, distinguished by whether their information arrives or not at the current moment. Let $d_i$ denote the number of iterations that each worker processor $i$ is delayed. If $d_i < \tau - 1$ for each worker processor $i \in \mathcal{A}_k^{\complement}$, the main processor uses the partially updated decision variables $(\mathbf{x}_1^{k+1}, \dots, \mathbf{x}_N^{k+1})$ to perform both corrector and predictor updates for the Lagrangian multiplier; otherwise, the main processor must wait until the worker processors with $d_i = \tau -1$ finish their tasks with information received by the master processor. Consequently, new divide of worker processors into $\mathcal{A}_k$ and $\mathcal{A}_k^{\complement}$ is generated, and the bounded delay condition is then satisfied. The overall structure of the Asynchronous $N$-Block PCPM algorithm for solving the linearly constrained convex optimization problem \eqref{eq: N-LCP} is presented in Algorithm \ref{alg:AN-PCPM-M} and Algorithm \ref{alg:AN-PCPM-W}.
\begin{algorithm}[ht]
%\footnotesize
\caption{AN-PCPM (Main Processor)}\label{alg:AN-PCPM-M}
\begin{algorithmic}[1]
\State \textbf{Initialization} choose an arbitrary starting point $\bm{\lambda}^0$.
\State $k\gets 0$, $d_1, d_2, \dots, d_N \gets 0$
\State \textbf{wait} until receiving $\{\hat{\mathbf{x}}_i\}_{i = 1 \dots N}$
\State {\textbf{update}:
\begin{equation}
\mathbf{x}_i^0 = \hat{\mathbf{x}}_i, \quad i = 1 \dots N
\end{equation}}
\State \textbf{broadcast} $\hat{\bm{\gamma}} = \bm{\lambda}^0 + \rho (\sum_{i=1}^N A_i \mathbf{x}_i^0 - \mathbf{b})$ to all worker processors
\While {termination conditions are not met}
\State {\textbf{wait} until receiving $\{\hat{\mathbf{x}}_i\}_{i \in \mathcal{A}_k}$ such that $d_i < \tau, \quad \forall i \in \mathcal{A}_k^{\complement}$}
\State {\textbf{update}:
\begin{equation}
\begin{aligned}
&\mathbf{x}_i^{k+1} = \left\{
\begin{aligned}
&\hat{\mathbf{x}}_i, &\forall i \in \mathcal{A}_k \\
&\mathbf{x}_i^k, &\forall i \in \mathcal{A}_k^{\complement}
\end{aligned}
\right. \\
&d_i = \left\{
\begin{aligned}
&0, &\forall i \in \mathcal{A}_k \\
&d_i + 1, &\forall i \in \mathcal{A}_k^{\complement}
\end{aligned}
\right.
\end{aligned}
\end{equation}}
\State {\textbf{update}:
\begin{equation}
\bm{\lambda}^{k+1} = \bm{\lambda}^k + \rho \big(\sum_{i = 1}^N A_i \mathbf{x}_i^{k+1} - \mathbf{b}\big)
\end{equation}}
\State \textbf{broadcast} $\hat{\bm{\gamma}} = \bm{\lambda}^{k+1} + \rho \sum_{i=1}^N A_i \mathbf{x}_i^{k+1}$ to the worker processors in $\mathcal{A}_k$
\State $k++$
\EndWhile
\State {\textbf{return} $(\mathbf{x}_1^k, \dots, \mathbf{x}_N^k, \bm{\lambda}^k)$}
\end{algorithmic}
\end{algorithm}
\begin{algorithm}[ht]
%\footnotesize
\caption{AN-PCPM (Woker Processor)}\label{alg:AN-PCPM-W}
\begin{algorithmic}[1]
\State \textbf{send} $\hat{\mathbf{x}}_i = \mathbf{x}_i^0$ to the main processor
\While {not receiving termination signal}
\State \textbf{wait} until receiving $\hat{\bm{\gamma}}$
\State {\textbf{calculate}:
\begin{equation}
\mathbf{y}_i = \underset{\mathbf{x}_i \in \mathcal{X}_i}{\argmin} f_i(\mathbf{x}_i) + \hat{\bm{\gamma}}^T A_i \mathbf{x}_i + \frac{1}{2 \rho} \lVert \mathbf{x}_i - \hat{\mathbf{x}}_i \rVert^2
\end{equation}}
\State \textbf{update}: $\hat{\mathbf{x}}_i = \mathbf{y}_i$
\State \textbf{send} $\hat{\mathbf{x}}_i$ to the main processor
\EndWhile
\end{algorithmic}
\end{algorithm}
\subsection{Convergence Analysis}\label{sec: AN-PCPM convergence}
Different from the synchronous $N$-block PCPM algorithm, the convexity of $f_i$ is not enough to achieve the global convergence due to the asynchronous delay in the system. We make the following additional assumption on problem \eqref{eq: N-LCP}. 
\begin{assumption}[Strong Convexity]\label{assp: strong convexity}
For all $i = 1 \dots N$, each $f_i: \mathcal{X}_i \to \mathbb{R}$ is a continuous, strongly convex function with modulus $\sigma_i > 0$. 
\end{assumption}
Accordingly, we extend Lemma~\ref{lem: Proximal Minimization} to functions with strong convexity.
\begin{lemma}[Inequality of Proximal Minimization Point with Strong Convexity]\label{lem: Proximal Minimization Strong Convexity}
Given a closed, convex set $\mathbb{Z} \subset \mathbb{R}^n$, and a continuous, strongly convex function $F: \mathbb{Z} \to \mathbb{R}$ with modulus $\sigma$. With a given  point $\bar{\mathbf{z}} \in \mathbb{Z}$ and a positive number $\rho > 0$, if $\widehat{\mathbf{z}}$ is a proximal minimization point; i.e. $\widehat{\mathbf{z}} \coloneqq \arg \underset{\mathbf{z} \in \mathbb{Z}}{\min}\ F(\mathbf{z}) + \frac{1}{2 \rho} \lVert \mathbf{z} - \bar{\mathbf{z}} \rVert_2^2$, then we have that
\begin{equation}
F(\widehat{\mathbf{z}}) - F(\mathbf{z}) \leq \frac{1}{2 \rho}\lVert \bar{\mathbf{z}} - \mathbf{z} \rVert_2^2 - (\frac{\sigma}{2} + \frac{1}{2 \rho})\lVert \widehat{\mathbf{z}} - \mathbf{z} \rVert_2^2 - \frac{1}{2 \rho}\lVert \widehat{\mathbf{z}} - \bar{\mathbf{z}} \rVert_2^2, \quad \forall \mathbf{z} \in \mathbb{Z}.
\end{equation}
\end{lemma}
\begin{proof}
Denote $\Phi(\mathbf{z}) = F(\mathbf{z}) + \frac{1}{2 \rho} \lVert \mathbf{z} - \bar{\mathbf{z}} \rVert_2^2$. By the definition of $\widehat{\mathbf{z}}$, we have $\partial_{\mathbf{z}} \Phi(\widehat{\mathbf{z}}) = \mathbf{0}$. Since $\Phi(\mathbf{z})$ is strongly convex with modulus $\sigma + \frac{1}{\rho}$, it follows that $\Phi(\mathbf{z}) - \Phi(\widehat{\mathbf{z}}) \geq (\frac{\sigma}{2} + \frac{1}{2 \rho})\lVert \widehat{\mathbf{z}} - \mathbf{z} \rVert_2^2$ for any $\mathbf{z} \in \mathbb{Z}$.
\end{proof}
Now, we present the main convergence result.
\begin{theorem}[Sub-linear Convergence Rate]\label{thm: sub-linear convergence}
Let Assumption~\ref{assp: N-PCPM Slater}, Assumption~\ref{assp: N-PCPM Saddle Point} and Assumption~\ref{assp: strong convexity} hold. Choose a step size $\rho$ satisfying:
\begin{equation}
\rho \leq \frac{\sigma_{min}}{25 N (\tau - 1)^2 A_{max}},
\end{equation}
where $\sigma_{min} \coloneqq \min_{i=1}^N\{\sigma_i\}$. Denote $\bar{\mathbf{x}}_i^k = \frac{1}{k} \sum_{k^{\prime}=1}^k \mathbf{x}_i^{k^{\prime}}$ for all $i = 1, \dots, N$, where $\{(\mathbf{x}_1^k, \dots, \mathbf{x}_N^k)\}$ is the sequence generated by Algorithm \ref{alg:AN-PCPM-M} and Algorithm \ref{alg:AN-PCPM-W}, then for all $k > 0$, it holds that:
\begin{equation}
\Big\lvert \sum_{i=1}^N f_i(\bar{\mathbf{x}}_i^k) - \sum_{i=1}^N f_i(\mathbf{x}_i^*) \Big\rvert \leq \frac{\delta_{\bm{\lambda}} C_1 + C_2}{k}, \quad \Big\lVert \sum_{i=1}^N A_i \bar{\mathbf{x}}^k - \mathbf{b}\Big\rVert_2 \leq \frac{C_1}{k},
\end{equation}
where $\delta_{\bm{\lambda}} = \lVert \bm{\lambda}^* \rVert_2$ and $C_1$, $C_2$ are some finite constants.
\end{theorem}
\begin{proof}
Please see Section~\ref{app: AN-PCPM Proof_1} for details.
\end{proof}
\section{Numerical Experiments}\label{sec: Num}
\subsection{An Optimization Problem on a Graph}
In this subsection, we consider an optimization problem on a graph, arising from the training process of regressors with spatial clustering, proposed by \cite{hallac2015network}. Traditional regressors obtains a parameter vector $\mathbf{x}$ via solving the following optimization problem on a training data set:
\begin{equation}
\underset{\mathbf{x} \in \mathcal{X}}{\text{minimize}} \quad \sum_{i = 1}^N f_i(\mathbf{x}) + r(\mathbf{x}),
\end{equation}
where $N$ is the number of data points, $\mathcal{X} \subset \mathbb{R}^n$ describes the constraints on the parameter vector, each function $f_i: \mathbb{R}^n \to \mathbb{R}$ denotes the loss function on each training data point for all $i = 1, \dots, N$, and $r: \mathbb{R}^n \to \mathbb{R}$ denotes some type of regularization function. 
\par When the spatial information is accessible, such as the latitude and longitude data, a map of data points then becomes available. Instead of using a global regressor with a common parameter vector $\mathbf{x}$ for the whole data set, a local regressor can be built at each data point $i = 1, \dots, N$ with a local parameter vector $\mathbf{x}_i \in \mathbb{R}^n$. Let $d_{ij}$ denote the distance between the data point $i$ and $j$ ($i \not= j$). Different from distributed learning, where a consensus constraint should be satisfied for all local variables, we require that the difference between two local parameter vectors $\lVert \mathbf{x}_i - \mathbf{x}_j \rVert_2^2$ decreases as the distance $d_{ij}$ decreases.
Let $\mathcal{N}_{\epsilon}(i)$ denote a set of data points within a neighborhood of point $i$, i.e., $\mathcal{N}_{\epsilon}(i) \coloneqq \{j = 1, \dots, N | j \not= i, d_{ij} \leq \epsilon \}$. If any data point is regarded as a vertex and any two data points within a neighborhood are connected through an edge, a graph is then constructed as $\mathcal{G} = (\mathcal{V}, \mathcal{E})$, where $\mathcal{V}$ denotes the set of vertices with $N = \lvert \mathcal{V} \rvert$ and $\mathcal{E}$ denotes the set of edges with $p = \lvert \mathcal{E} \rvert$. Consider the following optimization problem on the graph:
\begin{equation}\label{eq: P_graph}
\begin{aligned}
\underset{\mathbf{x}_1, \dots, \mathbf{x}_N}{\text{minimize}} \quad &\sum_{i \in \mathcal{V}} \big[f_i(\mathbf{x}_i) + r(\mathbf{x}_i)\big] + \omega \sum_{(j, k) \in \mathcal{E}} w_{jk} \lVert \mathbf{x}_j - \mathbf{x}_k \rVert_2^2 \\
\text{subject to} \quad &\mathbf{x}_i \in \mathcal{X}_i, \quad i = 1, \dots, N.
\end{aligned}
\end{equation}
Different from \cite{hallac2015network}, we use $\lVert \cdot \rVert_2^2$ instead of $\lVert \cdot \rVert_2$. The parameter $w_{jk}$ along each edge $(j,k) \in \mathcal{E}$, describing the weight of the penalty term of the difference between the two connected vertices, increases as $d_{jk}$ decreases. The global parameter $\omega$ describes the trade-off between minimizing the individual loss function on each data point and agreeing with neighbors. When $\omega = 0$, $\mathbf{x}_i^*$ is simply the solution to the optimization problem: $\underset{\mathbf{x}_i \in \mathbb{X}_i}{\text{minimize}} \quad f_i(\mathbf{x}_i) + r(\mathbf{x}_i)$, obtained locally at each vertex $i$. When $\omega \to +\infty$, the model reduces to a traditional regressor without spatial clustering.
\par Once the optimal solution $(\mathbf{x}_1^*, \dots, \mathbf{x}_N^*)$ is obtained, for any new node $i^{\prime}$, the local regressor can be evaluated with the local parameter vector $\mathbf{x}_{i^{\prime}}$, estimated through the interpolation of the solution:
\begin{equation}
\underset{\mathbf{x_{i^{\prime}}} \in \mathcal{X}_{i^{\prime}}}{\text{minimize}} \quad \sum_{j \in \mathcal{N}_{\epsilon}(i^{\prime})} w_{i^{\prime}j} \lVert \mathbf{x}_{i^{\prime}} - \mathbf{x}_j^* \rVert_2^2.
\end{equation}
\subsection{Two Problem Reformulations}
To apply $N$-block PCPM algorithm to solve the graph optimization problem \eqref{eq: P_graph}, we first need to reformulate it into the block-separable form as in \eqref{eq: N-LCP}.
\par Similar to \cite{hallac2015network}, for each pair of connected vertices $(\mathbf{x}_j, \mathbf{x}_k)$ along the edge $(j,k) \in \mathcal{E}$, introducing a copy $(\mathbf{z}_{jk}, \mathbf{z}_{kj})$, we can rewrite the graph optimization problem \eqref{eq: P_graph} as
\begin{itemize}
\item \textbf{Problem Reformulation 1}
\begin{equation}\label{eq: Convex P_graph}
\begin{aligned}
\underset{\mathbf{x}_1, \dots, \mathbf{x}_N}{\text{minimize}} \quad &\sum_{i \in \mathcal{V}} \big[f_i(\mathbf{x}_i) + r(\mathbf{x}_i)\big] + \omega \sum_{(j, k) \in \mathcal{E}} w_{jk} \lVert \mathbf{z}_{jk} - \mathbf{z}_{kj} \rVert_2^2 \\
\text{subject to} \quad &\mathbf{x}_i \in \mathcal{X}_i, \quad \forall i \in \mathcal{V}, \\
&\mathbf{x}_i - \mathbf{z}_{ij} = \mathbf{0}, \quad \forall j \in \mathcal{N}_{\epsilon}(i), \quad \forall i \in \mathcal{V}.
\end{aligned}
\end{equation}
\end{itemize}
The reformulated problem can be decomposed into $N$ sub-problems on vertices and $p$ sub-problems along edges, using $N$-block PCPM algorithm. One small issue is that the objective function of the edge sub-problem, $\lVert \mathbf{z}_{jk} - \mathbf{z}_{kj} \rVert_2^2$, is not strongly convex.
\par To overcome this limitation, we propose an alternative way of rewriting \eqref{eq: P_graph}. For each edge $(j,k) \in \mathcal{E}$, introducing a slack variable $\mathbf{z}_{jk} = \mathbf{x}_j - \mathbf{x}_k$, we can rewrite the graph optimization problem \eqref{eq: P_graph} as
\begin{itemize}
\item \textbf{Problem Reformulation 2}
\begin{equation}\label{eq: Strongly Convex P_graph}
\begin{aligned}
\underset{\mathbf{x}_1, \dots, \mathbf{x}_N}{\text{minimize}} \quad &\sum_{i \in \mathcal{V}} \big[f_i(\mathbf{x}_i) + r(\mathbf{x}_i)\big] + \omega \sum_{(j, k) \in \mathcal{E}} w_{jk} \lVert \mathbf{z}_{jk} \rVert_2^2 \\
\text{subject to} \quad &\mathbf{x}_i \in \mathcal{X}_i, \quad \forall i \in \mathcal{V}, \\
&\mathbf{x}_j - \mathbf{x}_k - \mathbf{z}_{jk} = \mathbf{0}, \quad \forall (j,k) \in \mathcal{E}.
\end{aligned}
\end{equation}
\end{itemize}
The reformulated problem can also be decomposed into $N$ sub-problems on vertices and $p$ sub-problems along edges. However, in this way of rewriting \eqref{eq: P_graph}, the objective function of each decomposed sub-problem enjoys the nice property of strong convexity. When applying $N$-block PCPM algorithm, a linear convergence rate is expected for synchronous iteration scheme, and a sub-linear convergence rate is expected for an asynchronous scheme.
\subsection{Housing Price Prediction}
In this subsection, we present an application example of the graph optimization problem, where the housing price is predicted based on a set of features, including the number of bedrooms, the number of bathrooms, the number of square feet, and the latitude and longitude of each house. We use the same data set as \cite{hallac2015network}, a list of 985 real estate transactions over a period of one week during May of 2008 in the Greater Sacramento area. All data is standardized with zero mean and unit variance, and all missing data is then set to zero. We randomly select a subset of 193 transactions as our test data set, and use the rest as our training data set.
\par The graph is constructed based on the latitude and longitude of each house. The rule of selecting neighbors is slightly different from \cite{hallac2015network}. For each house, we connect it with all the other houses within a distance of $1.0$ mile. If the number of connected houses is less then $5$, we connect more nearest houses until the number of neighbors reaches $5$. The resulting graph has $792$ vertices and $4303$ edges. Thus, the graph optimization problem can be decomposed into $792 + 4303 = 5095$ sub-problems, using $N$-block PCPM algorithm.
\par At each data point $i = 1 \dots 792$, the decision variable is $\mathbf{x}_i = (x_{i0}, x_{i1}, x_{i2}, x_{i3})$. The predicted price for each house is:
$$
 x_{i0} + x_{i1} \times \text{(num\_bed)}_i + x_{i2} \times \text{(num\_bath)}_i + x_{i3} \times \text{(num\_sq\_ft)}_i,
$$
where $\text{(num\_bed)}_i$, $\text{(num\_bath)}_i$ and $\text{(num\_sq\_ft)}_i$ are the number of bedrooms, the number of bathrooms and the number of square feets for each house respectively. At each vertex $i$, the objective function
$$
f_i(\mathbf{x}_i) = \lVert x_{i0} + x_{i1} \times \text{(num\_bed)}_i + x_{i2} \times \text{(num\_bath)}_i + x_{i3} \times \text{(num\_sq\_ft)}_i - \text{(price)}_i \rVert_2^2
$$
is strongly convex, as well as the regularization function
$$
r(\mathbf{x}_i) = \mu \big(\lVert x_{i1} \rVert_2^2 + \lVert x_{i2} \rVert_2^2 + \lVert x_{i3} \rVert_2^2\big),
$$
where $\text{(price)}_i$ is the actual sales price for each house, and $\mu$ is a constant regularization parameter, fixed as $\mu = 0.1$. 
\subsection{Numerical Results of Synchronous N-block PCPM Algorithm}
We first apply Algorithm~\ref{alg: N-PCPM} to solve the two reformulated problem \eqref{eq: Convex P_graph} and \eqref{eq: Strongly Convex P_graph}, and compare the performance. The convergence results are shown in Figure~\ref{fig: con_strong_con}.
\begin{figure}[htb]
\centering
\begin{subfigure}{.38\textwidth}
\centering
\includegraphics[width=1.0\linewidth]{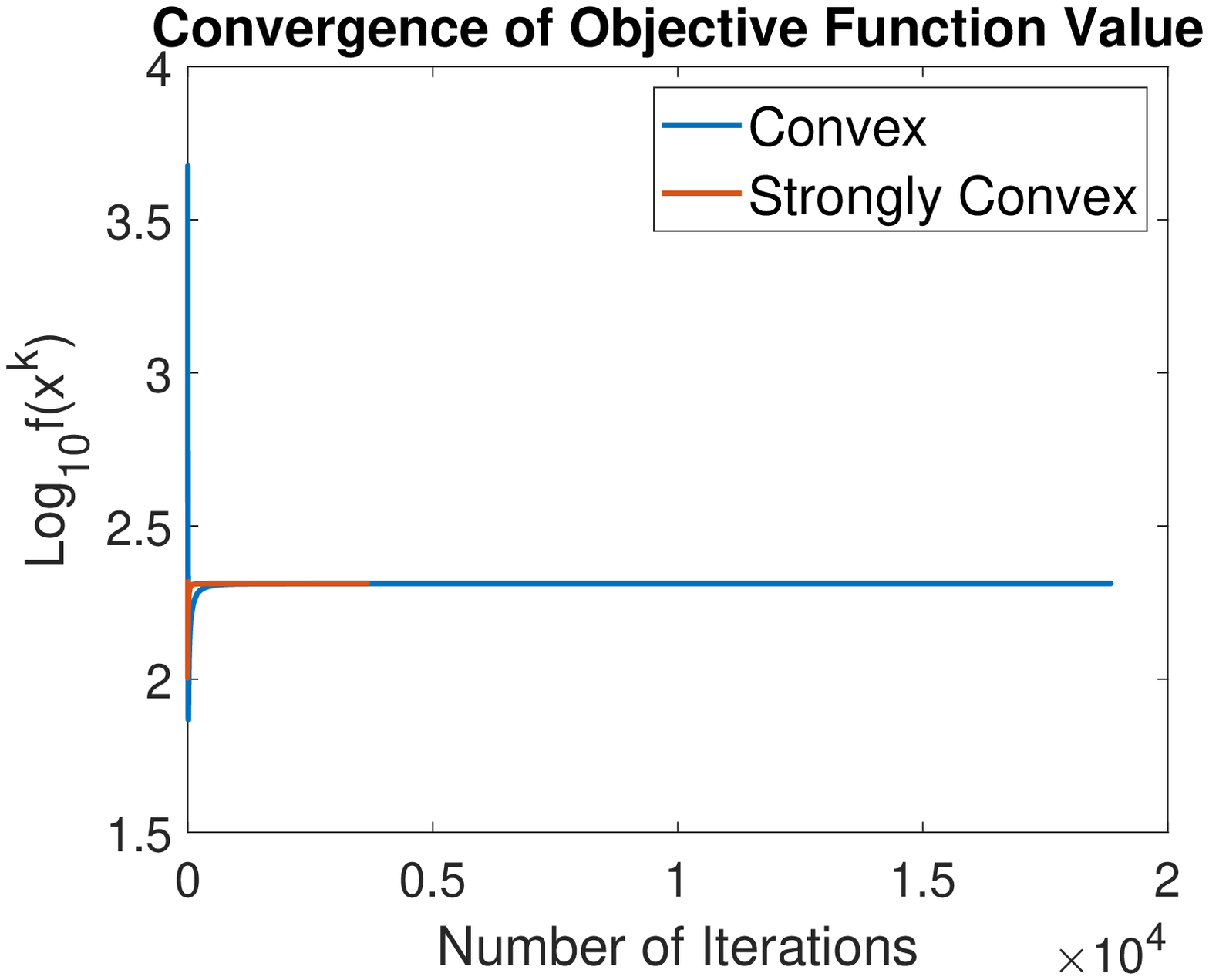}
\end{subfigure}%
\begin{subfigure}{.38\textwidth}
\centering
\includegraphics[width=1.0\linewidth]{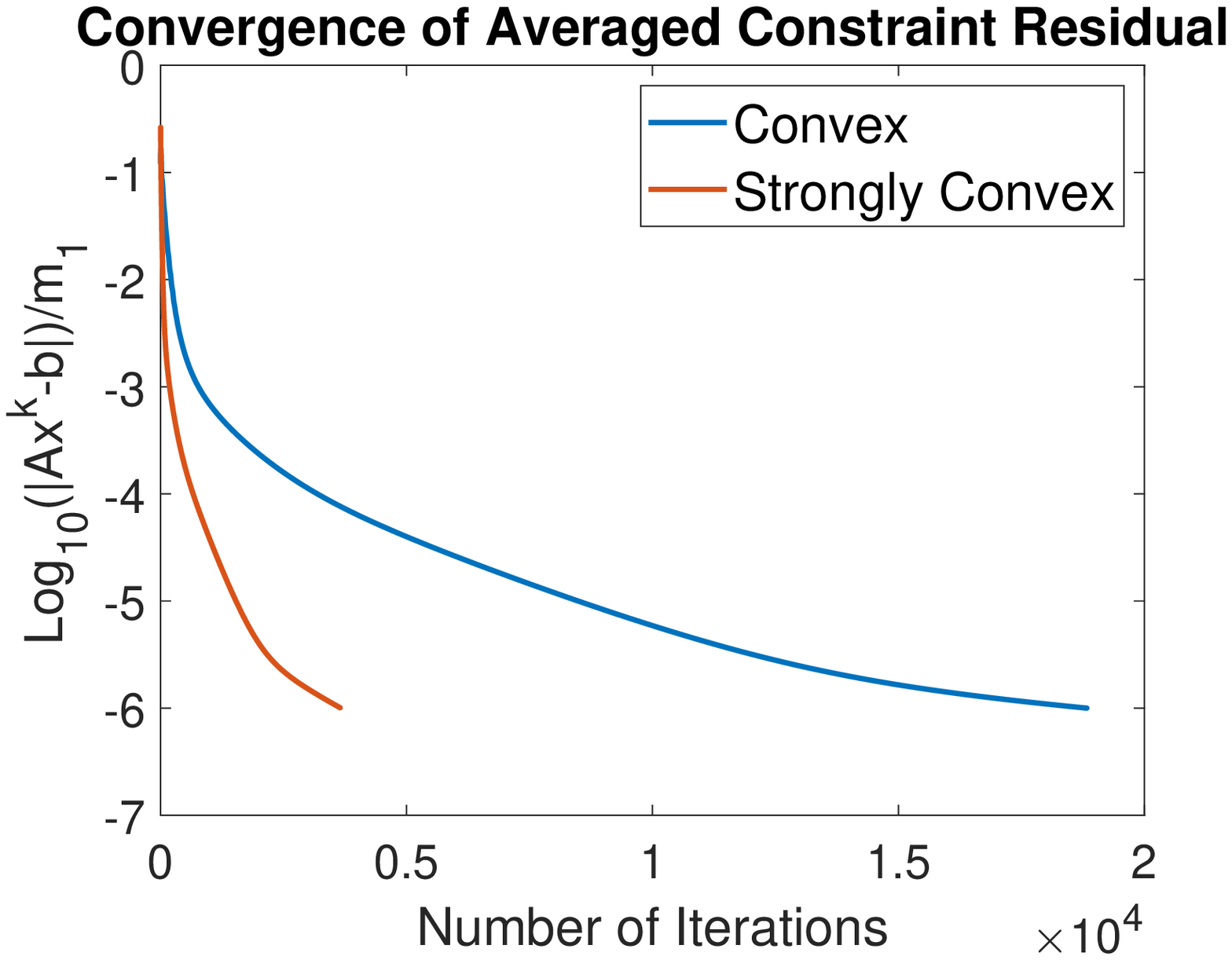}
\end{subfigure}
\begin{subfigure}{.38\textwidth}
\centering
\includegraphics[width=1.0\linewidth]{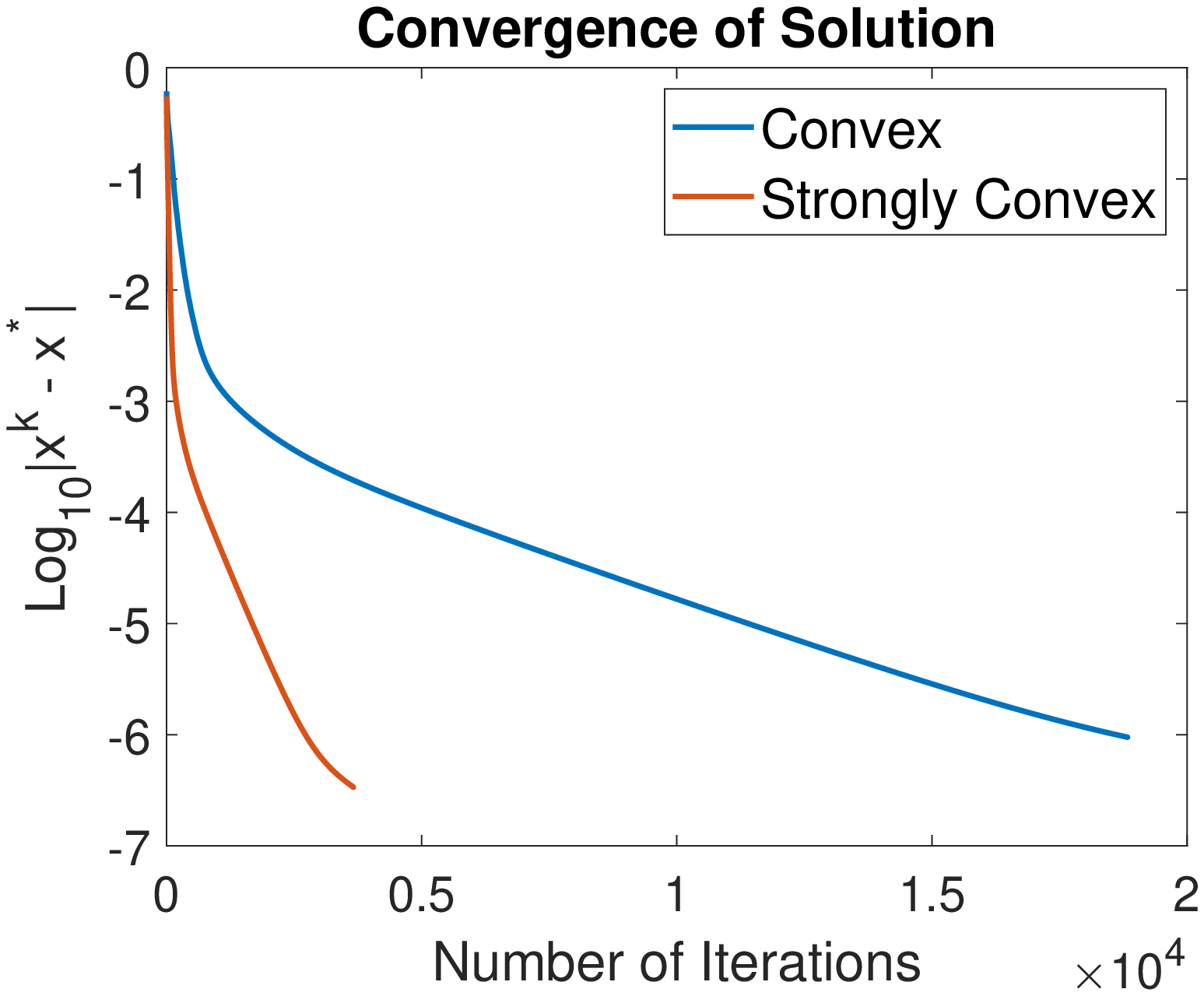}
\end{subfigure}%
\begin{subfigure}{.38\textwidth}
\centering
\includegraphics[width=1.0\linewidth]{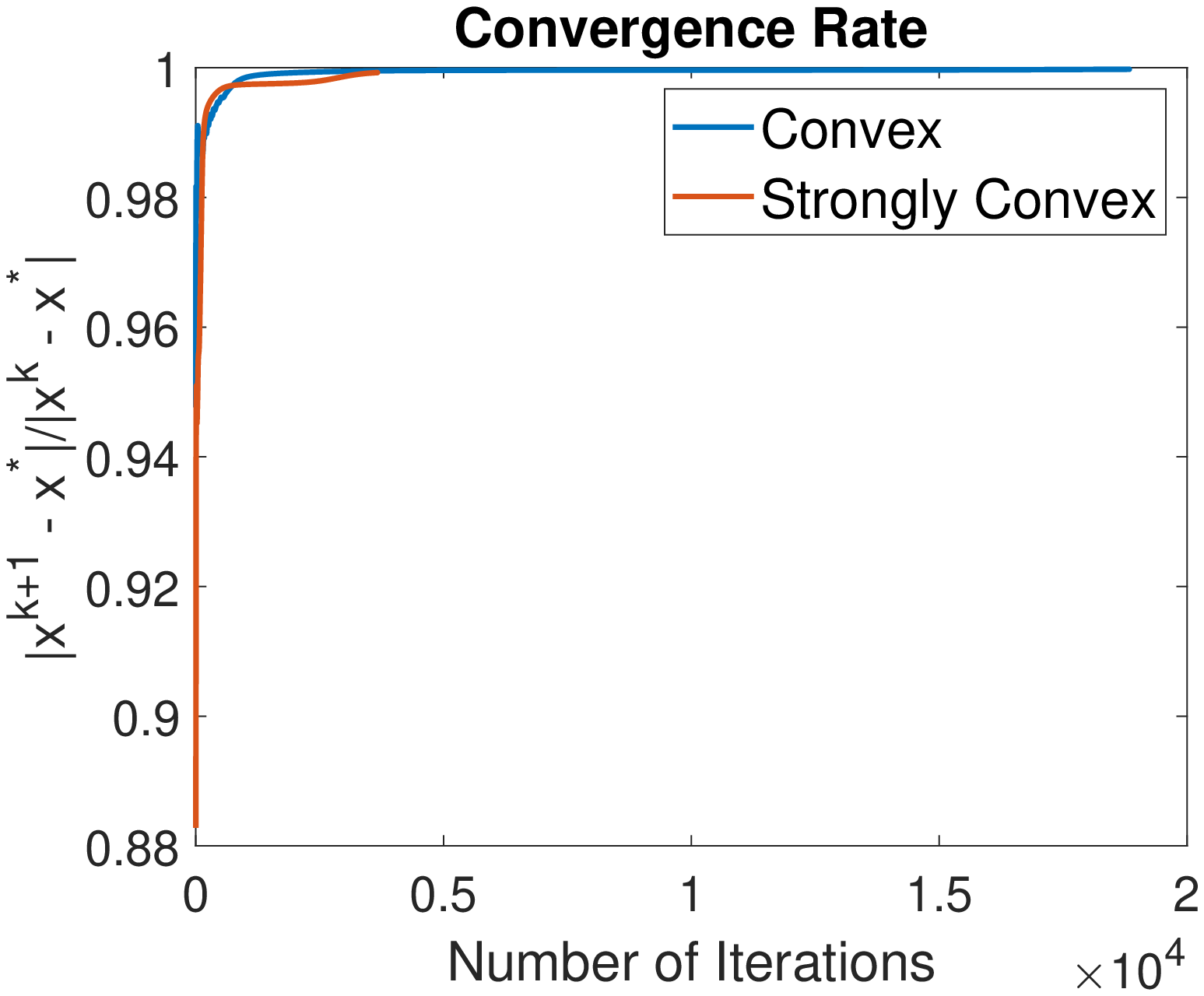}
\end{subfigure}
\caption{Convergence Results of Applying Algorithm~\ref{alg: N-PCPM} to solve Reformulated Problems \eqref{eq: Convex P_graph} and \eqref{eq: Strongly Convex P_graph} with $\omega = 1.0$ and $\rho = 0.06$.}
\label{fig: con_strong_con}%\vspace*{-10pt}
\end{figure}
Due to the strong convexity of the reformulated problem \eqref{eq: Strongly Convex P_graph}, the algorithm converges much faster than solving the reformulated problem \eqref{eq: Convex P_graph}.
\par We also plot the mean square error (MSE) on the testing data set using various values of $\omega$, shown in Figure \ref{fig: w_MSE}. 
\begin{figure}[ht]
\begin{center}
\includegraphics[scale = 0.38]{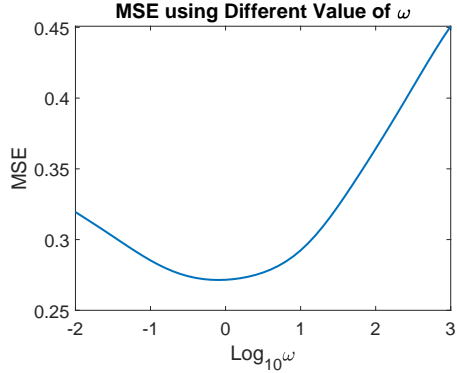}
\caption{MSE for Testing Data Set with $\omega$ Varying from $10^{-2}$ to $10^3$ and $\mu = 0.1$.}
\label{fig: w_MSE}
\end{center}
\end{figure}
When $\omega = 1.0$, a minimum MSE of $0.27$ can be obtained on the testing data set. We fixed $\omega = 1.0$ for all the numerical experiments.
\subsection{Numerical Results of Asynchronous N-block PCPM Algorithm}
We apply Algorithm~\ref{alg:AN-PCPM-M} and Algorithm~\ref{alg:AN-PCPM-W} to solve the reformulated problem \eqref{eq: Strongly Convex P_graph} with a maximum delay $\tau = 4$. The convergence results of are shown in Figure~\ref{fig: Asyn}.
\begin{figure}[htb]
\centering
\begin{subfigure}{.38\textwidth}
\centering
\includegraphics[width=1.0\linewidth]{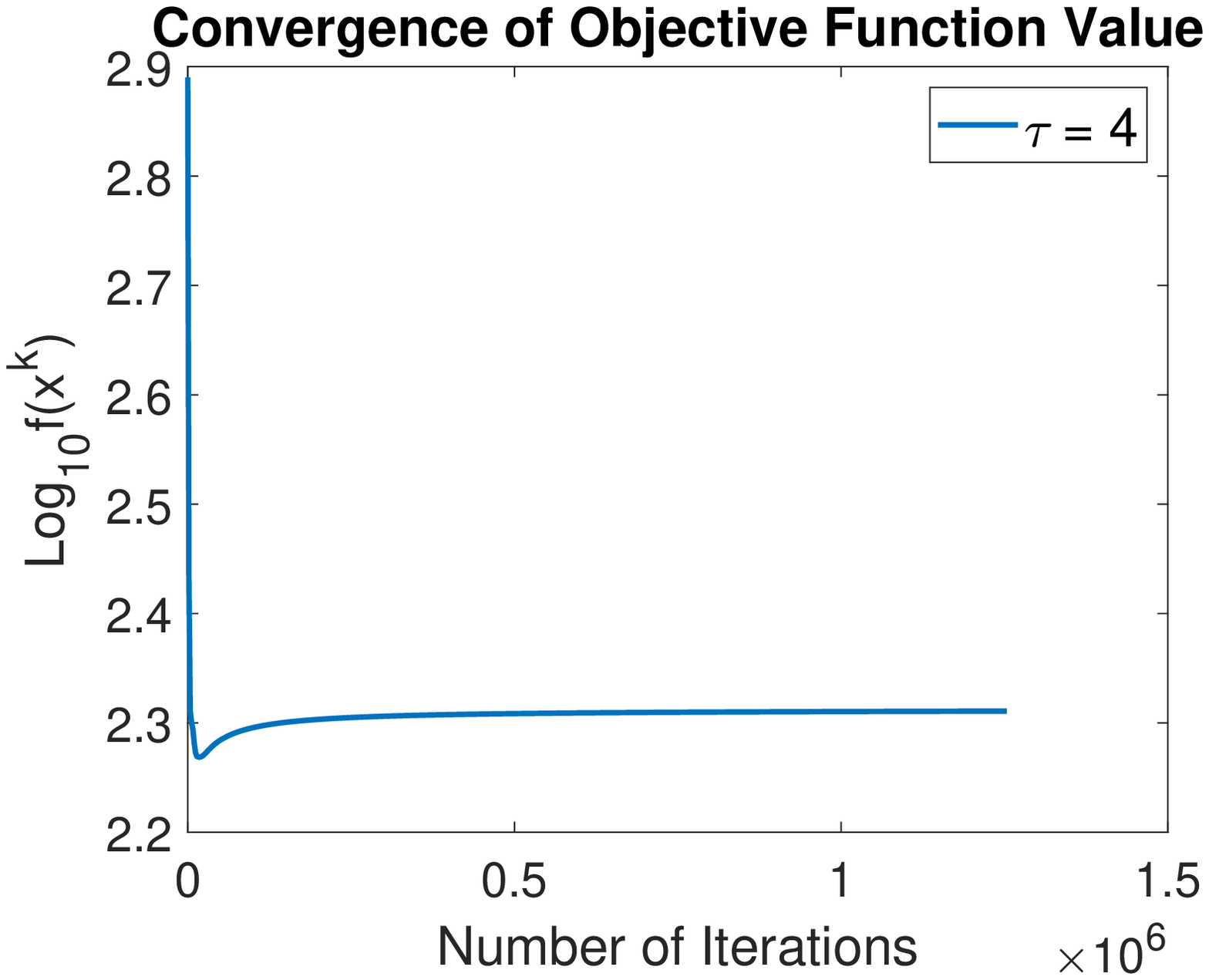}
\end{subfigure}%
\begin{subfigure}{.38\textwidth}
\centering
\includegraphics[width=1.0\linewidth]{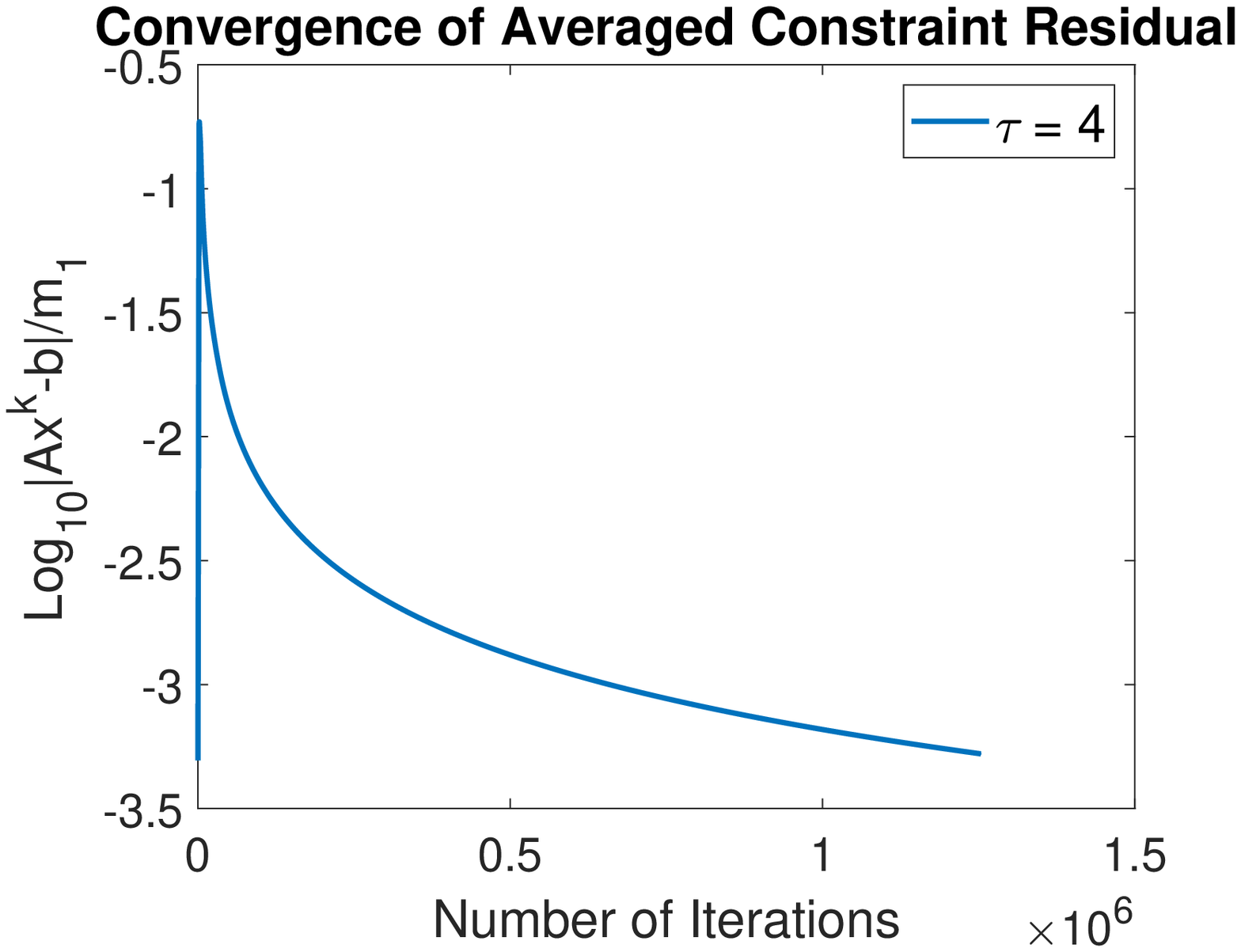}
\end{subfigure}
\begin{subfigure}{.38\textwidth}
\centering
\includegraphics[width=1.0\linewidth]{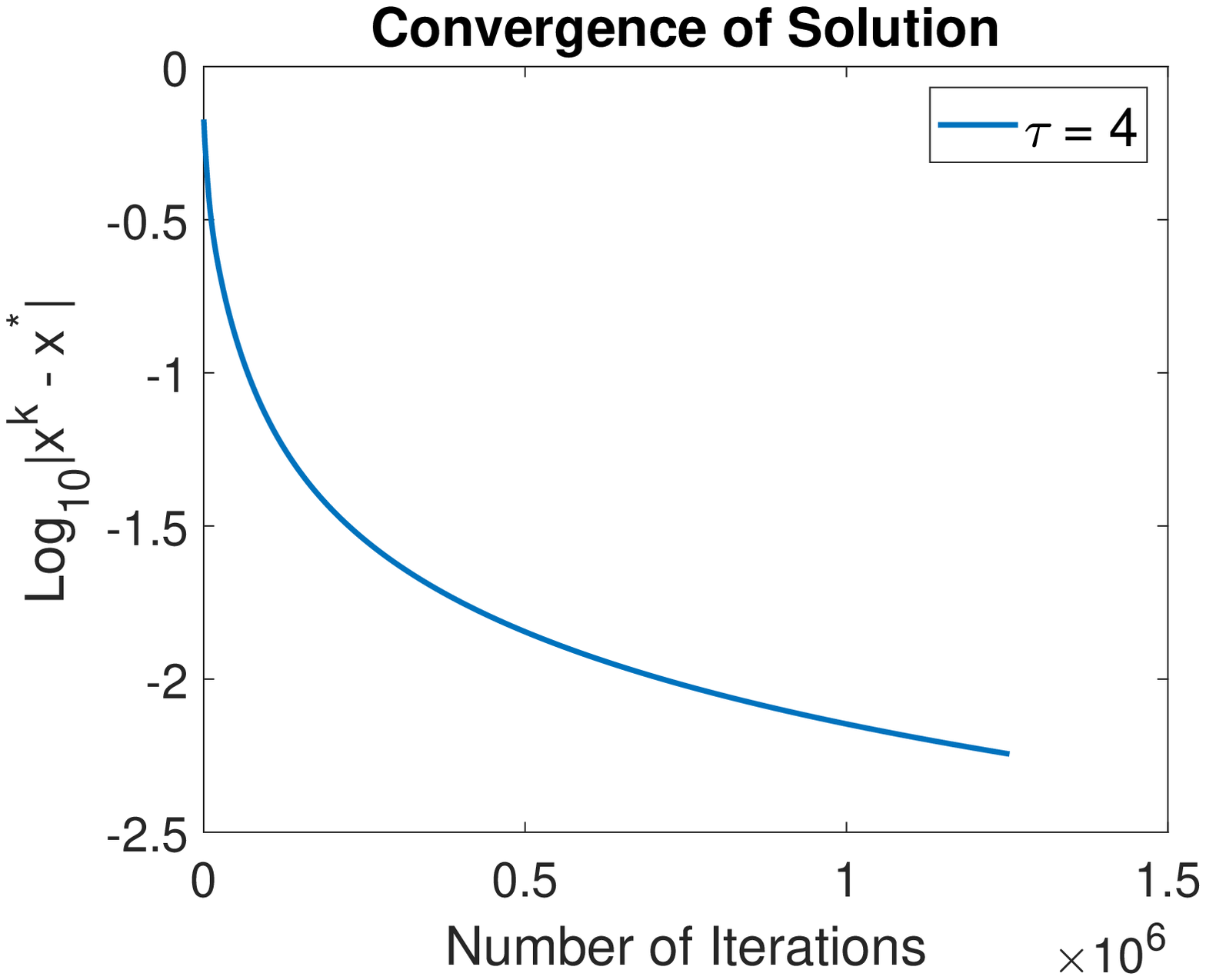}
\end{subfigure}%
\begin{subfigure}{.38\textwidth}
\centering
\includegraphics[width=1.0\linewidth]{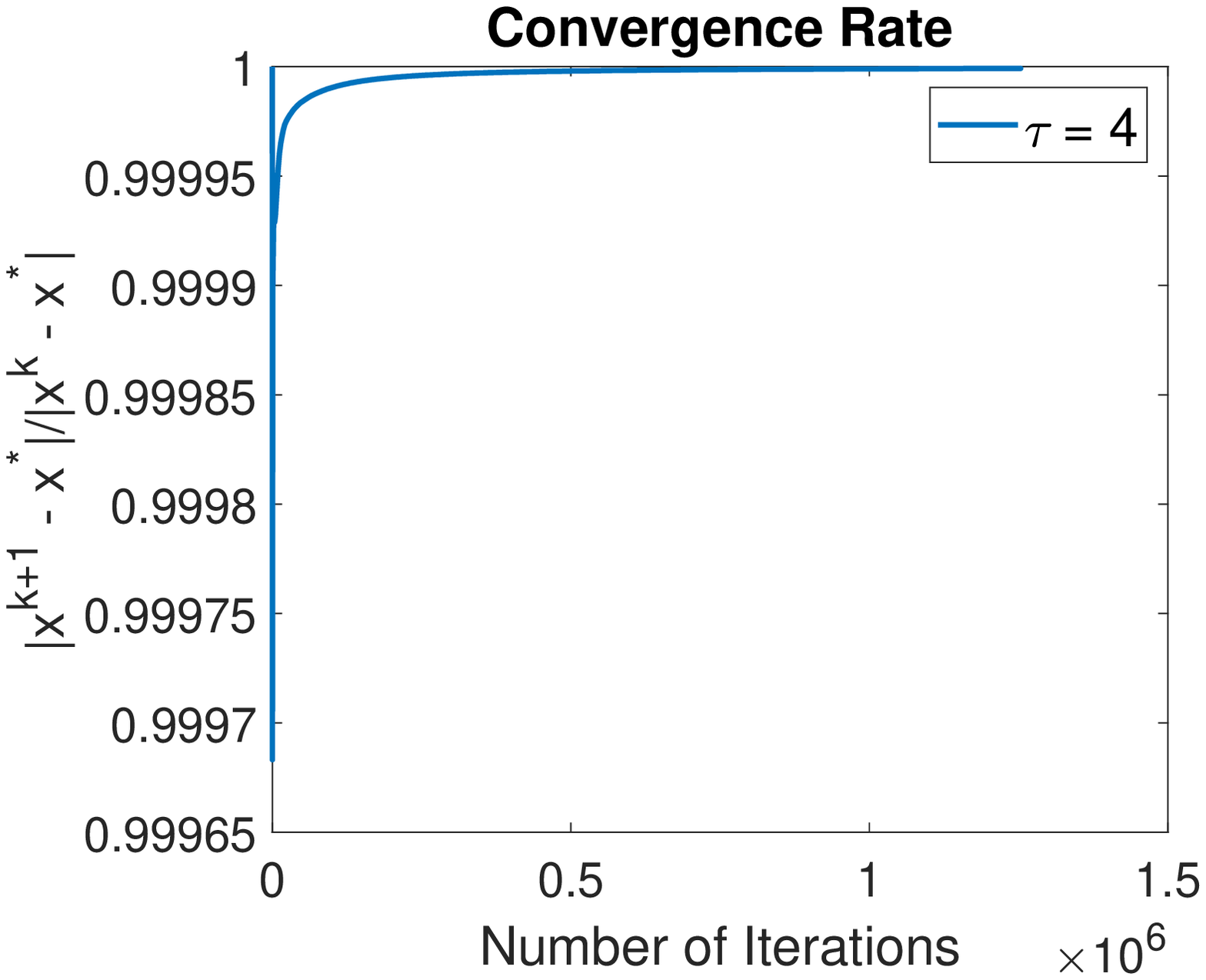}
\end{subfigure}
\caption{Convergence Results of Applying Algorithm~\ref{alg:AN-PCPM-M} and Algorithm~\ref{alg:AN-PCPM-W} to Solve the Reformulated Problem \eqref{eq: Strongly Convex P_graph} with $\tau = 4$ and $\rho = 0.0005$.}
\label{fig: Asyn}%\vspace*{-10pt}
\end{figure}
A sub-linear convergence rate is observed. 
\par While implementing the algorithm as a sequential code, we simulate the elapsed wall-clock time on the main processor. As illustrated in Figure~\ref{fig: syn_asyn}, the computation delay of the main processor is set as $1.0$ second, the computation delay of each worker processor for solving vertex sub-problem is set as $1.2$ second, the computation delay of each worker processor for solving edge sub-problem is set as $0.6$ second, and the communication delay of each worker processor is uniformly drawn from a range of $0.0$ to $1.0$ second. Under these settings, we simulate the elapsed wall-clock time on the main processor for the different values of maximum delay $\tau = 1, 2, 4, 7$, shown in Figure~\ref{fig: Asyn_Timeline}.
\begin{figure}[ht]
\begin{center}
\includegraphics[scale = 0.4]{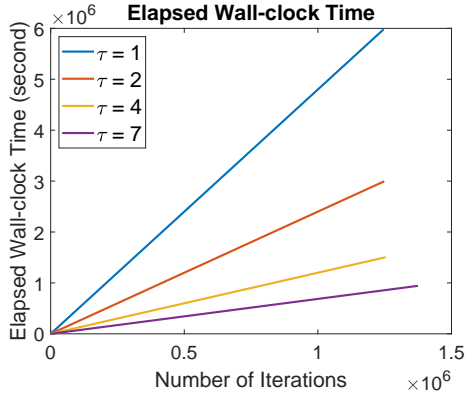}
\caption{Simulated Elapsed Wall-clock Time on the Main Processor with Various Maximum Delay $\tau$ and a Same $\rho = 0.0005$.}
\label{fig: Asyn_Timeline}
\end{center}
\end{figure}
We observe that, with a larger number of maximum delay, the number of iterations, used for the asynchronous algorithm to converge, increases but the simulated elapsed wall-clock time decreases, which implies a faster convergence with more short-time iterations.
\section{Conclusion and Future Works}\label{sec: Conclusion}
In this paper, we first proposed an $N$-block PCPM algorithm to solve $N$-block convex optimization problems with both linear and nonlinear constraints, with global convergence established. A linear convergence rate under the strong second-order conditions for optimality is observed in the numerical experiments. Next, for a starting point, we proposed an asynchronous $N$-block PCPM algorithm to solve linearly constrained $N$-block convex optimization problems. The numerical results demonstrate the sub-linear convergence rate under the bounded delay assumption, as well as the faster convergence with more short-time iterations than a synchronous iterative scheme.
\par However, the performance of real asynchronous implementation of $N$-block PCPM algorithm is unknown, and thus in the future, more experiments (probably with much larger problem sizes) will be conducted on a multi-node computer cluster using MPI functions without blocking communication, such as MPI\_Isend and MPI\_Irecv. Also, the extension of the asynchronous $N$-block PCPM algorithm to solve $N$-block convex optimization problems with both linear and nonlinear constraints is worth to be explored.
%
%
%
%
%%%%%%%%%%%%%%%%%%%%%%%%%%%%%%%%%%%%%%%%%%%%%%%%%%%%%%%%%%%%%%%%%%%%%%%%%%%%%%%%
%\section*{References}
%when submitting to arXiv: use the submit function to generate bbl file and then insert "\input{main.bbl}"
\bibliographystyle{plain}
\bibliography{main}

\begin{thebibliography}{10}

\bibitem{asaadi1973computational}
J.~Asaadi.
\newblock A computational comparison of some non-linear programs.
\newblock {\em Mathematical Programming}, 4(1):144--154, 1973.

\bibitem{boyd2011distributed}
Stephen Boyd, Neal Parikh, and Eric Chu.
\newblock Distributed optimization and statistical learning via the alternating
  direction method of multipliers.
\newblock {\em Foundations and Trends{\textregistered} in Machine learning},
  3(1):1--122, 2011.

\bibitem{chang2016asynchronous1}
Tsung-Hui Chang, Mingyi Hong, Wei-Cheng Liao, and Xiangfeng Wang.
\newblock Asynchronous distributed {ADMM} for large-scale optimization—part
  i: Algorithm and convergence analysis.
\newblock {\em IEEE Transactions on Signal Processing}, 64(12):3118--3130,
  2016.

\bibitem{chang2016asynchronous2}
Tsung-Hui Chang, Wei-Cheng Liao, Mingyi Hong, and Xiangfeng Wang.
\newblock Asynchronous distributed admm for large-scale optimization—part ii:
  Linear convergence analysis and numerical performance.
\newblock {\em IEEE Transactions on Signal Processing}, 64(12):3131--3144,
  2016.

\bibitem{chen1994proximal}
Gong Chen and Marc Teboulle.
\newblock A proximal-based decomposition method for convex minimization
  problems.
\newblock {\em Mathematical Programming}, 64(1-3):81--101, 1994.

\bibitem{deng2017parallel}
Wei Deng, Ming-Jun Lai, Zhimin Peng, and Wotao Yin.
\newblock Parallel multi-block {ADMM} with o (1/k) convergence.
\newblock {\em Journal of Scientific Computing}, 71(2):712--736, 2017.

\bibitem{hallac2015network}
David Hallac, Jure Leskovec, and Stephen Boyd.
\newblock Network lasso: clustering and optimization in large graphs.
\newblock In {\em Proceedings of the 21th ACM SIGKDD International Conference
  on Knowledge Discovery and Data Mining}, pages 387--396. ACM, 2015.

\bibitem{rockafellar1976augmented}
R~Tyrrell Rockafellar.
\newblock Augmented {L}agrangians and applications of the proximal point
  algorithm in convex programming.
\newblock {\em Mathematics of Operations Research}, 1(2):97--116, 1976.

\bibitem{rockafellar1976monotone}
R~Tyrrell Rockafellar.
\newblock Monotone operators and the proximal point algorithm.
\newblock {\em SIAM Journal on Control and Optimization}, 14(5):877--898, 1976.

\bibitem{rockafellar2015convex}
R.~Tyrrell Rockafellar.
\newblock {\em Convex Analysis}.
\newblock Princeton University Press, 2015.

\bibitem{sahni1996master}
Sartaj Sahni and George Vairaktarakis.
\newblock The master-slave paradigm in parallel computer and industrial
  settings.
\newblock {\em Journal of Global Optimization}, 9(3-4):357--377, 1996.

\bibitem{wang2014parallel}
Huahua Wang, Arindam Banerjee, and Zhi-Quan Luo.
\newblock Parallel direction method of multipliers.
\newblock In {\em Advances in Neural Information Processing Systems}, pages
  181--189, 2014.

\bibitem{wang2013solving}
Xiangfeng Wang, Mingyi Hong, Shiqian Ma, and Zhi-Quan Luo.
\newblock Solving multiple-block separable convex minimization problems using
  two-block alternating direction method of multipliers.
\newblock {\em arXiv preprint arXiv:1308.5294}, 2013.

\bibitem{yu2017simple}
Hao Yu and Michael~J Neely.
\newblock A simple parallel algorithm with an o(1/t) convergence rate for
  general convex programs.
\newblock {\em SIAM Journal on Optimization}, 27(2):759--783, 2017.

\end{thebibliography}
%
%
%
%
%%%%%%%%%%%%%%%%%%%%%%%%%%%%%%%%%%%%%%%%%%%%%%%%%%%%%%%%%%%%%%%%%%%%%%%%%%%%%%%%
\begin{appendix}
\section{Proofs in Section~\ref{sec: N-PCPM convergence}}
\subsection{Proof of Proposition~\ref{prop: distance}}\label{app: N-PCPM Proof_1}
We first prove the inequality \eqref{eq: N-PCPM Primal Distance}. From the primal minimization step \eqref{eq: N-block PCPM primal minimization}, we know that $(\mathbf{x}_1^{k+1}, \dots, \mathbf{x}_N^{k+1})$ is the unique proximal minimization point of the Lagrangian function evaluated at the predictor variable: $\mathcal{L}(\mathbf{x}_1, \dots, \mathbf{x}_N, \bm{\gamma}^{k+1}, \bm{\nu}^{k+1})$. Applying Lemma~\ref{lem: Proximal Minimization} with $\widehat{\mathbf{z}} = (\mathbf{x}_1^{k+1}, \dots, \mathbf{x}_N^{k+1})$, $\bar{\mathbf{z}} = (\mathbf{x}_1^k, \dots, \mathbf{x}_N^k)$ and $\mathbf{z} = (\mathbf{x}_1^*, \dots, \mathbf{x}_N^*)$, we have:
\begin{equation}
\begin{aligned}
&2 \rho\Big[\mathcal{L}(\mathbf{x}_1^{k+1}, \dots, \mathbf{x}_N^{k+1}, \bm{\gamma}^{k+1}, \bm{\nu}^{k+1}) - \mathcal{L}(\mathbf{x}_1^*, \dots, \mathbf{x}_N^*, \bm{\gamma}^{k+1}, \bm{\nu}^{k+1})\Big] \\
\leq &\sum_{i=1}^N \lVert \mathbf{x}_i^k - \mathbf{x}_i^* \rVert_2^2 - \sum_{i=1}^N \lVert \mathbf{x}_i^{k+1} - \mathbf{x}_i^* \rVert_2^2 - \sum_{i=1}^N \lVert \mathbf{x}_i^{k+1} - \mathbf{x}_i^k \rVert_2^2.
\end{aligned}
\end{equation}
Since $(\mathbf{x}_1^*, \dots, \mathbf{x}_N^*, \bm{\lambda}^*, \bm{\mu}^*)$ is a saddle point of the Lagrangian function $\mathcal{L}(\mathbf{x}_1, \dots, \mathbf{x}_N, \bm{\lambda}, \bm{\mu})$, i.e., $\mathcal{L}(\mathbf{x}_1^*, \dots, \mathbf{x}_N^*, \mathbf{p}^{k+1}) \leq \mathcal{L}(\mathbf{x}_1^*, \dots, \mathbf{x}_N^*, \bm{\lambda}^*) \leq \mathcal{L}(\mathbf{x}_1^{k+1}, \dots, \mathbf{x}_N^{k+1}, \bm{\lambda}^*)$, we have:
\begin{equation}
2 \rho\Big[\mathcal{L}(\mathbf{x}_1^*, \dots, \mathbf{x}_N^*, \bm{\gamma}^{k+1}, \bm{\nu}^{k+1}) - \mathcal{L}(\mathbf{x}_1^{k+1}, \dots, \mathbf{x}_N^{k+1}, \bm{\lambda}^*, \bm{\mu}^*)\Big] \leq 0.
\end{equation}
Adding the above two inequalities yields the inequality \eqref{eq: N-PCPM Primal Distance} in Proposition~\ref{prop: distance}.
\par To prove the second inequality, \eqref{eq: N-PCPM Dual Distance}, in Proposition~\ref{prop: distance}, we use a similar approach as above. By Lemma \ref{lem: Dual}, we know that $(\bm{\gamma}^{k+1}, \bm{\nu}^{k+1})$ is the unique proximal minimization point of the function $-\mathcal{L}(\mathbf{x}_1^k, \dots, \mathbf{x}_N^k, \bm{\lambda}, \bm{\mu})$. Applying Lemma~\ref{lem: Proximal Minimization} with $\widehat{\mathbf{z}} = (\bm{\gamma}^{k+1}, \bm{\nu}^{k+1})$, $\bar{\mathbf{z}} = (\bm{\lambda}^k, \bm{\mu}^k)$ and $\mathbf{z} = (\bm{\lambda}^{k+1}, \bm{\mu}^{k+1})$, we have:
\begingroup
\begin{align}
&2 \rho\Bigg\{\Big[-\mathcal{L}(\mathbf{x}_1^k, \dots, \mathbf{x}_N^k, \bm{\gamma}^{k+1}, \bm{\nu}^{k+1})\Big] - \Big[-\mathcal{L}(\mathbf{x}_1^k, \dots, \mathbf{x}_N^k, \bm{\lambda}^{k+1}, \bm{\mu}^{k+1})\Big]\Bigg\} \nonumber \\
\leq &\lVert \bm{\lambda}^k - \bm{\lambda}^{k+1} \rVert_2^2 + \lVert \bm{\mu}^k - \bm{\mu}^{k+1} \rVert_2^2 \nonumber \\
- &\lVert \bm{\gamma}^{k+1} - \bm{\lambda}^{k+1} \rVert_2^2 - \lVert \bm{\nu}^{k+1} - \bm{\mu}^{k+1} \rVert_2^2 - \lVert \bm{\gamma}^{k+1} - \bm{\lambda}^k \rVert_2^2 - \lVert \bm{\nu}^{k+1} - \bm{\mu}^k \rVert_2^2.
\end{align}
\endgroup
By Lemma \ref{lem: Dual}, we also know that $(\bm{\lambda}^{k+1}, \bm{\mu}^{k+1})$ is the unique proximal minimization point of the function $-\mathcal{L}(\mathbf{x}_1^{k+1}, \dots, \mathbf{x}_N^{k+1}, \bm{\lambda}, \bm{\mu})$. Applying Lemma~\ref{lem: Proximal Minimization} with $\widehat{\mathbf{z}} = (\bm{\lambda}^{k+1}, \bm{\mu}^{k+1})$, $\bar{\mathbf{z}} = (\bm{\lambda}^k, \bm{\mu}^k)$ and $\mathbf{z} = (\bm{\lambda}^*, \bm{\mu}^*)$, we have:
\begingroup
\begin{align}
&2 \rho\Bigg\{\Big[-\mathcal{L}(\mathbf{x}_1^{k+1}, \dots, \mathbf{x}_N^{k+1}, \bm{\lambda}^{k+1}, \bm{\mu}^{k+1})\Big] - \Big[-\mathcal{L}(\mathbf{x}_1^{k+1}, \dots, \mathbf{x}_N^{k+1}, \bm{\lambda}^*, \bm{\mu}^*)\Big]\Bigg\} \nonumber \\
\leq &\lVert \bm{\lambda}^k - \bm{\lambda}^* \rVert_2^2 + \lVert \bm{\mu}^k - \bm{\mu}^* \rVert_2^2 \nonumber \\
- &\lVert \bm{\lambda}^{k+1} - \bm{\lambda}^* \rVert_2^2 - \lVert \bm{\mu}^{k+1} - \bm{\mu}^* \rVert_2^2 - \lVert \bm{\lambda}^{k+1} - \bm{\lambda}^k \rVert_2^2 - \lVert \bm{\mu}^{k+1} - \bm{\mu}^k \rVert_2^2.
\end{align}
\endgroup
Adding the above two inequalities yields the inequality \eqref{eq: N-PCPM Dual Distance} in Proposition~\ref{prop: distance}.
\subsection{Proof of Theorem~\ref{thm: N-PCPM global convergence}}\label{app: N-PCPM Proof_2}
By adding the two inequalities \eqref{eq: N-PCPM Primal Distance} and \eqref{eq: N-PCPM Dual Distance} in Proposition~\ref{prop: distance}, we have:
\begingroup
\begin{align}
&\sum_{i=1}^N \lVert \mathbf{x}_i^{k+1} - \mathbf{x}_i^* \rVert_2^2 + \lVert \bm{\lambda}^{k+1} - \bm{\lambda}^* \rVert_2^2 + \lVert \bm{\mu}^{k+1} - \bm{\mu}^* \rVert_2^2 \nonumber \\
\leq &\sum_{i=1}^N \lVert \mathbf{x}_i^k - \mathbf{x}_i^* \rVert_2^2 + \lVert \bm{\lambda}^k - \bm{\lambda}^* \rVert_2^2 + \lVert \bm{\mu}^k - \bm{\mu}^* \rVert_2^2 \nonumber \\
- &\sum_{i=1}^N \lVert \mathbf{x}_i^{k+1} - \mathbf{x}_i^k \rVert_2^2 - \lVert \bm{\gamma}^{k+1} - \bm{\lambda}^{k+1} \rVert_2^2 - \lVert \bm{\nu}^{k+1} - \bm{\mu}^{k+1} \rVert_2^2 - \lVert \bm{\gamma}^{k+1} - \bm{\lambda}^k \rVert_2^2 - \lVert \bm{\nu}^{k+1} - \bm{\mu}^k \rVert_2^2 \nonumber \\
+ &\sum_{i=1}^N \underbrace{2 \rho (\bm{\lambda}^{k+1} - \bm{\gamma}^{k+1})^T A_i (\mathbf{x}_i^{k+1} - \mathbf{x}_i^k)}_{(a)_i} + \sum_{j=1}^M \sum_{i=1}^N \underbrace{2 \rho (\mu_j^{k+1} - \nu_j^{k+1})\Big[g_{ji}(\mathbf{x}_i^{k+1}) - g_{ji}(\mathbf{x}_i^k)\Big]}_{(b)_{ji}}. \label{eq: substitution}
\end{align}
\endgroup
Before we continue with the proof, we first show an extension of the Young's inequality\footnote{Young's inequality states that if $a$ and $b$ are two non-negative real numbers, and $p$ and $q$ are real numbers greater than 1 such that $\frac{1}{p} + \frac{1}{q} = 1$, then $ab < \frac{a^p}{p} + \frac{b^q}{q}$.} on vector products that will play a key role in the following proof.
\parindent=20pt Given any two vectors $\mathbf{z}_1, \mathbf{z}_2 \in \mathbb{R}^n$, we have that
$$
\mathbf{z}_1^T \mathbf{z}_2 = \sum_{j=1}^n z_{1j} z_{2j} = \sum_{j=1}^n \Big(\frac{1}{\delta} z_{1j}\Big) \Big(\delta z_{2j}\Big) \leq \sum_{j=1}^n \Big\lvert \frac{1}{\delta} z_{1j} \Big\rvert \Big\lvert \delta z_{2j} \Big\rvert,
$$
where $\delta$ is a non-zero real number. Applying Young's inequality on each summation term with $p = q = 2$, we obtain that 
\begin{equation}\label{eq: N-PCPM Young's Inequality}
\mathbf{z}_1^T \mathbf{z}_2 \leq \sum_{j=1}^n \left[\frac{1}{2} \Big(\frac{1}{\delta} z_{1j}\Big)^2 + \frac{1}{2} \Big(\delta z_{2j}\Big)^2\right] = \frac{1}{2 \delta^2} \lVert \mathbf{z}_1 \rVert_2^2 + \frac{\delta^2}{2} \lVert \mathbf{z}_2 \rVert_2^2.
\end{equation}
Applying \eqref{eq: N-PCPM Young's Inequality} on each term $(a)_i$ yields
\begin{equation}
\begin{aligned}
(a)_i \leq &2 \rho \Big[\frac{1}{2 \delta^2} \lVert \bm{\gamma}^{k+1} - \bm{\lambda}^{k+1} \rVert_2^2 + \frac{\delta^2}{2} \lVert A_i(\mathbf{x}_i^{k+1} - \mathbf{x}_i^k) \rVert_2^2\Big] \\
\leq &2 \rho \Big[\frac{1}{2 \delta^2} \lVert \bm{\gamma}^{k+1} - \bm{\lambda}^{k+1} \rVert_2^2 + \frac{\delta^2}{2} \lVert A_i \rVert_2^2 \lVert \mathbf{x}_i^{k+1} - \mathbf{x}_i^k \rVert_2^2\Big],
\end{aligned}
\end{equation}
and letting $\delta^2 = \frac{1}{\lVert A_i \rVert_2}$ yields
\begin{equation}\label{eq: (a)_i}
\begin{aligned}
(a)_i \leq &\rho \lVert A_i \rVert_2 \Big(\lVert \bm{\gamma}^{k+1} - \bm{\lambda}^{k+1} \rVert_2^2 + \lVert \mathbf{x}_i^{k+1} - \mathbf{x}_i^k \rVert_2^2\Big) \\
\leq &\rho A_{max} \Big(\lVert \bm{\gamma}^{k+1} - \bm{\lambda}^{k+1} \rVert_2^2 + \lVert \mathbf{x}_i^{k+1} - \mathbf{x}_i^k \rVert_2^2\Big).
\end{aligned}
\end{equation}
Applying \eqref{eq: N-PCPM Young's Inequality} on each term $(b)_{ji}$ yields
\begin{equation}
\begin{aligned}
(b)_{ji} \leq &2 \rho \Big[\frac{1}{2 \delta^2} \lVert \nu_j^{k+1} - \mu_j^{k+1} \rVert_2^2 + \frac{\delta^2}{2} \lVert g_{ji}(\mathbf{x}_i^{k+1}) - g_{ji}(\mathbf{x}_i^k) \rVert_2^2\Big] \\
\leq &2 \rho \Big[\frac{1}{2 \delta^2} \lVert \nu_j^{k+1} - \mu_j^{k+1} \rVert_2^2 + \frac{\delta^2}{2} L_{ji}^2 \lVert \mathbf{x}_i^{k+1} - \mathbf{x}_i^k \rVert_2^2\Big],
\end{aligned}
\end{equation}
and letting $\delta^2 = \frac{1}{L_{ji}}$ yields
\begin{equation}\label{eq: (b)_ji}
\begin{aligned}
(b)_{ji} \leq &\rho L_{ji} \Big(\lVert \nu_j^{k+1} - \mu_j^{k+1} \rVert_2^2 + \lVert \mathbf{x}_i^{k+1} - \mathbf{x}_i^k \rVert_2^2\Big) \\
\leq &\rho L_{max} \Big(\lVert \nu_j^{k+1} - \mu_j^{k+1} \rVert_2^2 + \lVert \mathbf{x}_i^{k+1} - \mathbf{x}_i^k \rVert_2^2\Big).
\end{aligned}
\end{equation}
Substituting \eqref{eq: (a)_i} and \eqref{eq: (b)_ji} into \eqref{eq: substitution} yields
\begingroup
\begin{align}
&\sum_{i=1}^N \lVert \mathbf{x}_i^{k+1} - \mathbf{x}_i^* \rVert_2^2 + \lVert \bm{\lambda}^{k+1} - \bm{\lambda}^* \rVert_2^2 + \lVert \bm{\mu}^{k+1} - \bm{\mu}^* \rVert_2^2 \nonumber \\
\leq &\sum_{i=1}^N \lVert \mathbf{x}_i^k - \mathbf{x}_i^* \rVert_2^2 + \lVert \bm{\lambda}^k - \bm{\lambda}^* \rVert_2^2 + \lVert \bm{\mu}^k - \bm{\mu}^* \rVert_2^2 \nonumber \\
- &(1 - \rho A_{max} - \rho M L_{max})\sum_{i=1}^N \lVert \mathbf{x}_i^{k+1} - \mathbf{x}_i^k \rVert_2^2 - (1 - \rho N A_{max})\lVert \bm{\gamma}^{k+1} - \bm{\lambda}^{k+1} \rVert_2^2 \nonumber \\
- &(1 - \rho N L_{max})\lVert \bm{\nu}^{k+1} - \bm{\mu}^{k+1} \rVert_2^2 - \lVert \bm{\gamma}^{k+1} - \bm{\lambda}^k \rVert_2^2 - \lVert \bm{\nu}^{k+1} - \bm{\mu}^k \rVert_2^2
\end{align}
\endgroup
Since $0 < \rho \leq \min\Big\{\frac{1 - \epsilon}{A_{max} + M L_{max}}, \frac{1 - \epsilon}{N A_{max}}, \frac{1 - \epsilon}{N L_{max}}\Big\}$, we have:
\begingroup
\begin{align}
&\sum_{i=1}^N \lVert \mathbf{x}_i^{k+1} - \mathbf{x}_i^* \rVert_2^2 + \lVert \bm{\lambda}^{k+1} - \bm{\lambda}^* \rVert_2^2 + \lVert \bm{\mu}^{k+1} - \bm{\mu}^* \rVert_2^2 \nonumber \\
\leq &\sum_{i=1}^N \lVert \mathbf{x}_i^k - \mathbf{x}_i^* \rVert_2^2 + \lVert \bm{\lambda}^k - \bm{\lambda}^* \rVert_2^2 + \lVert \bm{\mu}^k - \bm{\mu}^* \rVert_2^2 \nonumber \\
- &\epsilon \Big(\sum_{i=1}^N \lVert \mathbf{x}_i^{k+1} - \mathbf{x}_i^k \rVert_2^2 + \lVert \bm{\gamma}^{k+1} - \bm{\lambda}^{k+1} \rVert_2^2 + \lVert \bm{\nu}^{k+1} - \bm{\mu}^{k+1} \rVert_2^2 \nonumber \\
&\hspace*{15pt} + \lVert \bm{\gamma}^{k+1} - \bm{\lambda}^k \rVert_2^2 + \lVert \bm{\nu}^{k+1} - \bm{\mu}^k \rVert_2^2\Big). \label{eq: N-PCPM sequence inequality}
\end{align}
\endgroup
It implies that for all $k \geq 0$:
\begingroup
\begin{align}
0 \leq &\sum_{i=1}^N \lVert \mathbf{x}_i^{k+1} - \mathbf{x}_i^* \rVert_2^2 + \lVert \bm{\lambda}^{k+1} - \bm{\lambda}^* \rVert_2^2 + \lVert \bm{\mu}^{k+1} - \bm{\mu}^* \rVert_2^2 \nonumber \\
\leq &\sum_{i=1}^N \lVert \mathbf{x}_i^k - \mathbf{x}_i^* \rVert_2^2 + \lVert \bm{\lambda}^k - \bm{\lambda}^* \rVert_2^2 + \lVert \bm{\mu}^k - \bm{\mu}^* \rVert_2^2 \nonumber \\
\leq &\sum_{i=1}^N \lVert \mathbf{x}_i^{k-1} - \mathbf{x}_i^* \rVert_2^2 + \lVert \bm{\lambda}^{k-1} - \bm{\lambda}^* \rVert_2^2 + \lVert \bm{\mu}^{k-1} - \bm{\mu}^* \rVert_2^2 \nonumber \\
\leq &\cdots \leq \sum_{i=1}^N \lVert \mathbf{x}_i^0 - \mathbf{x}_i^* \rVert_2^2 + \lVert \bm{\lambda}^0 - \bm{\lambda}^* \rVert_2^2 + \lVert \bm{\mu}^0 - \bm{\mu}^* \rVert_2^2.
\end{align}
\endgroup
It further implies that the sequence $\Big\{\sum_{i=1}^N \lVert \mathbf{x}_i^k - \mathbf{x}_i^* \rVert_2^2 + \lVert \bm{\lambda}^k - \bm{\lambda}^* \rVert_2^2 + \lVert \bm{\mu}^k - \bm{\mu}^* \rVert_2^2\Big\}$ is monotonically decreasing and bounded below by $0$; hence the sequence must be convergent to a limit, denoted by $\xi$:
\begin{equation}\label{eq: N-PCPM bounded below sequence convergence}
\lim_{k \to +\infty} \sum_{i=1}^N \lVert \mathbf{x}_i^k - \mathbf{x}_i^* \rVert_2^2 + \lVert \bm{\lambda}^k - \bm{\lambda}^* \rVert_2^2 + \lVert \bm{\mu}^k - \bm{\mu}^* \rVert_2^2 = \xi.
\end{equation}
Taking the limit on both sides of \eqref{eq: N-PCPM sequence inequality} yields:
\begin{equation}\label{eq: 5 sequence convergence}
\begin{aligned}
&\lim_{k \to +\infty} \sum_{i=1}^N \lVert \mathbf{x}_i^{k+1} - \mathbf{x}_i^k \rVert_2^2 = 0, \\
&\lim_{k \to +\infty} \lVert \bm{\gamma}^{k+1} - \bm{\lambda}^{k+1} \rVert_2^2 = 0, \quad &\lim_{k \to +\infty} \lVert \bm{\nu}^{k+1} - \bm{\mu}^{k+1} \rVert_2^2 = 0, \\
&\lim_{k \to +\infty} \lVert \bm{\gamma}^{k+1} - \bm{\lambda}^k \rVert_2^2 = 0, \quad &\lim_{k \to +\infty} \lVert \bm{\nu}^{k+1} - \bm{\mu}^k \rVert_2^2 = 0.
\end{aligned}
\end{equation}
Additionally, \eqref{eq: N-PCPM bounded below sequence convergence} also implies that $\{(\mathbf{x}_1^k, \dots, \mathbf{x}_N^k, \bm{\lambda}^k, \bm{\mu}^k)\}$ is a bounded sequence, and thus there exists a sub-sequence $\{(\mathbf{x}_1^{k_j}, \dots, \mathbf{x}_N^{k_j}, \bm{\lambda}^{k_j}, \bm{\mu}^{k_j})\}$ that converges to a limit point $(\mathbf{x}_1^{\infty}, \dots, \mathbf{x}_N^{\infty}, \bm{\lambda}^{\infty}, \bm{\mu}^{\infty})$. We next show that the limit point is indeed a saddle point and is also the unique limit point of $\{(\mathbf{x}_1^k, \dots, \mathbf{x}_N^k, \bm{\lambda}^k, \bm{\mu}^k)\}$. Applying Lemma \ref{lem: Proximal Minimization} with $\widehat{\mathbf{z}} = (\mathbf{x}_1^{k+1}, \dots, \mathbf{x}_N^{k+1})$, $\bar{\mathbf{z}} = (\mathbf{x}_1^k, \dots, \mathbf{x}_N^k)$ and any $\mathbf{z} = (\mathbf{x}_1, \dots, \mathbf{x}_N) \in \prod_{i=1}^N \mathcal{X}_i$, we have:
\begingroup
\begin{align}
&2 \rho\Big[\mathcal{L}(\mathbf{x}_1^{k+1}, \dots, \mathbf{x}_N^{k+1}, \bm{\gamma}^{k+1}, \bm{\nu}^{k+1}) - \mathcal{L}(\mathbf{x}_1, \dots, \mathbf{x}_N, \bm{\gamma}^{k+1}, \bm{\nu}^{k+1})\Big] \nonumber \\
\leq &\sum_{i=1}^N \lVert \mathbf{x}_i^k - \mathbf{x}_i \rVert_2^2 - \sum_{i=1}^N \lVert \mathbf{x}_i^{k+1} - \mathbf{x}_i \rVert_2^2 - \sum_{i=1}^N \lVert \mathbf{x}_i^{k+1} - \mathbf{x}_i^k \rVert_2^2 \nonumber \\
\leq &\sum_{i=1}^N \big(\lVert \mathbf{x}_i^k - \mathbf{x}_i^{k+1} \rVert^2 + \lVert \mathbf{x}_i^{k+1} - \mathbf{x}_i \rVert^2\big) - \sum_{i=1}^N \lVert \mathbf{x}_i^{k+1} - \mathbf{x}_i \rVert^2 - \sum_{i=1}^N \lVert \mathbf{x}_i^{k+1} - \mathbf{x}_i^k \rVert^2 = 0 \nonumber \\
&\forall (\mathbf{x}_1, \dots, \mathbf{x}_N) \in \prod_{i=1}^N \mathcal{X}_i.
\end{align}
\endgroup
Taking the limits over an appropriate sub-sequence $\{k_j\}$ on both sides and using \eqref{eq: 5 sequence convergence}, we have:
\begin{equation}
\mathcal{L}(\mathbf{x}_1^{\infty}, \dots, \mathbf{x}_N^{\infty}, \bm{\lambda}^{\infty}, \bm{\mu}^{\infty}) \leq \mathcal{L}(\mathbf{x}_1, \dots, \mathbf{x}_N, \bm{\lambda}^{\infty}, \bm{\mu}^{\infty}), \quad \forall (\mathbf{x}_1, \dots, \mathbf{x}_N) \in \prod_{i=1}^N \mathcal{X}_i.
\end{equation}
Similarly, applying Lemma~\ref{lem: Proximal Minimization} with $\widehat{\mathbf{z}} = (\bm{\lambda}^{k+1}, \bm{\mu}^{k+1})$, $\bar{\mathbf{z}} = (\bm{\lambda}^k, \bm{\mu}^k)$ and any $\mathbf{z} = (\bm{\lambda}, \bm{\mu} \in \mathbb{R}_+^{m_2})$, we have:
\begingroup
\begin{align}
&2 \rho \left[\mathcal{L}(\mathbf{x}_1^{k+1}, \dots, \mathbf{x}_N^{k+1}, \bm{\lambda}, \bm{\mu}) - \mathcal{L}(\mathbf{x}_1^{k+1}, \dots, \mathbf{x}_N^{k+1}, \bm{\lambda}^{k+1}, \bm{\mu}^{k+1})\right] \nonumber \\
\leq &\lVert \bm{\lambda}^k - \bm{\lambda} \rVert^2 - \lVert \bm{\lambda}^{k+1} - \bm{\lambda} \rVert^2 - \lVert \bm{\lambda}^{k+1} - \bm{\lambda}^k \rVert^2 \nonumber \\
+ &\lVert \bm{\mu}^k - \bm{\mu} \rVert^2 - \lVert \bm{\mu}^{k+1} - \bm{\mu} \rVert^2 - \lVert \bm{\mu}^{k+1} - \bm{\mu}^k \rVert^2 \nonumber \\
\leq &\big(\lVert \bm{\lambda}^k - \bm{\lambda}^{k+1} \rVert^2 + \lVert \bm{\lambda}^{k+1} - \bm{\lambda} \rVert^2\big) - \lVert \bm{\lambda}^{k+1} - \bm{\lambda} \rVert^2 - \lVert \bm{\lambda}^{k+1} - \bm{\lambda}^k \rVert^2 \nonumber \\
+ &\big(\lVert \bm{\mu}^k - \bm{\mu}^{k+1} \rVert^2 + \lVert \bm{\mu}^{k+1} - \bm{\mu} \rVert^2\big) - \lVert \bm{\mu}^{k+1} - \bm{\mu} \rVert^2 - \lVert \bm{\mu}^{k+1} - \bm{\mu}^k \rVert^2 = 0, \quad \forall \bm{\mu} \in \mathbb{R}_+^{m_2}.
\end{align}
\endgroup
Taking the limits over an appropriate sub-sequence $\{k_j\}$ on both sides and using \eqref{eq: 5 sequence convergence}, we have:
\begin{equation}
\mathcal{L}(\mathbf{x}_1^{\infty}, \dots, \mathbf{x}_N^{\infty}, \bm{\lambda}, \bm{\mu}) \leq \mathcal{L}(\mathbf{x}_1^{\infty}, \dots, \mathbf{x}_N^{\infty}, \bm{\lambda}^{\infty}, \bm{\mu}^{\infty}), \quad \forall \bm{\mu} \in \mathbb{R}_+^{m_2}.
\end{equation}
Therefore, we show that $(\mathbf{x}_1^{\infty}, \dots, \mathbf{x}_N^{\infty}, \bm{\lambda}^{\infty}, \bm{\mu}^{\infty})$ is indeed a saddle point of the Lagrangian function $\mathcal{L}(\mathbf{x}_1, \dots, \mathbf{x}_N, \bm{\lambda}, \bm{\mu})$. Then \eqref{eq: N-PCPM bounded below sequence convergence} implies that
\begin{equation}
\lim_{k \to +\infty} \sum_{i=1}^N \lVert \mathbf{x}_i^k - \mathbf{x}_i^{\infty} \rVert_2^2 + \lVert \bm{\lambda}^k - \bm{\lambda}^{\infty} \rVert_2^2 + \lVert \bm{\mu}^k - \bm{\mu}^{\infty} \rVert_2^2 = \xi.
\end{equation}
Since we have already argued (after Eq. \eqref{eq: 5 sequence convergence}) that there exists a bounded sequence of $\{(\mathbf{x}_1^k, \dots, \mathbf{x}_N^k, \bm{\lambda}^k, \bm{\mu}^k)\}$ that converges to 0; that is, there exists $\{k_j\}$ such that 
$$
\lim_{k_j \to +\infty} \sum_{i=1}^N \lVert \mathbf{x}_i^{k_j} - \mathbf{x}_i^{\infty} \rVert_2^2 + \lVert \bm{\lambda}^{k_j} - \bm{\lambda}^{\infty} \rVert_2^2 + \lVert \bm{\mu}^{k_j} - \bm{\mu}^{\infty} \rVert_2^2 = 0,
$$
which then implies that $\xi = 0$. Therefore, we show that $\{(\mathbf{x}_1^k, \dots, \mathbf{x}_N^k, \bm{\lambda}^k, \bm{\mu}^k)\}$ converges globally to a saddle point $(\mathbf{x}_1^{\infty}, \dots, \mathbf{x}_N^{\infty}, \bm{\lambda}^{\infty}, \bm{\mu}^{\infty})$. 
\subsection{Proof of Theorem~\ref{thm: N-PCPM linear convergence}}\label{app: N-PCPM Proof_3}
Letting
\begin{equation}
\begin{aligned}
&\mathbf{u}_i^k = A_i^T(\bm{\lambda}^{k+1} - \bm{\gamma}^{k+1}) + \sum_{j=1}^{m_2}(\mu_j^{k+1} - \nu_j^{k+1}) \nabla_{\mathbf{x}_i} g_{ji}(\mathbf{x}_i^{k+1}) - \frac{1}{\rho} (\mathbf{x}_i^{k+1} - \mathbf{x}_i^k), \quad \forall i = 1 \dots N, \\
&\mathbf{v}^k = -\frac{1}{\rho} (\bm{\lambda}^{k+1} - \bm{\lambda}^k), \\
&\mathbf{w}^k = -\frac{1}{\rho} (\bm{\mu}^{k+1} - \bm{\mu}^k),
\end{aligned}
\end{equation}
we first show that $(\mathbf{x}_1^{k+1}, \dots, \mathbf{x}_N^{k+1}, \bm{\lambda}^{k+1}, \bm{\mu}^{k+1}) \in S^{-1}(\mathbf{u}_1^k, \dots, \mathbf{u}_N^k, \mathbf{v}^k, \mathbf{w}^k)$. By the primal minimization step \eqref{eq: N-block PCPM primal minimization}, we have, for all $i = 1 \dots N$:
\begin{equation}
-\nabla_{\mathbf{x}_i} f_i(\mathbf{x}_i^{k+1}) - \underbrace{\Big[A_i^T\bm{\gamma}^{k+1} + \sum_{j=1}^M \nu_j^{k+1} \nabla_{\mathbf{x}_i} g_{ji}(\mathbf{x}_i^{k+1}) + \frac{1}{\rho} (\mathbf{x}_i^{k+1} - \mathbf{x}_i^k)\Big]}_{\Delta_{u_i}} \in \mathcal{N}_{\mathcal{X}_i}(\mathbf{x}_i^{k+1}),
\end{equation}
where $\mathcal{N}_{\mathcal{X}_i}(\mathbf{x}^{k+1}) \coloneqq \{\mathbf{y} \in \mathbb{R}^{n_1} | \mathbf{y}^T (\mathbf{x} - \mathbf{x}^{k+1}) \leq \mathbf{0}, \forall \mathbf{x} \in \mathcal{X}_i\}$ denotes the normal cone to the set $\mathcal{X}_i$ at the point $\mathbf{x}_i^{k+1}$ for all $i = 1, \dots, N$. Plugging 
$$
\Delta_{u_i} = A_i^T \bm{\lambda}^{k+1} + \sum_{j=1}^M \mu_j^{k+1} \nabla_{\mathbf{x}_i}g_{ji}(\mathbf{x}_i^{k+1}) - \mathbf{u}_i^k
$$
into the above expression, we have that, for all $i = 1, \dots, N$:
\begin{equation}
-\nabla_{\mathbf{x}_i} f_i(\mathbf{x}_i^{k+1}) - A_i^T\bm{\lambda}^{k+1} - \sum_{j=1}^M \mu_j^{k+1} \nabla_{\mathbf{x}_i} g_{ji}(\mathbf{x}_i^{k+1}) + \mathbf{u}_i^k \in \mathcal{N}_{\mathcal{X}_i}(\mathbf{x}_i^{k+1}),
\end{equation}
which implies 
\begin{equation}\label{eq: first-order optimality 1}
\begin{aligned}
&(\mathbf{x}_1^{k+1}, \dots, \mathbf{x}_N^{k+1}) \\
\in &\underset{(\mathbf{x}_1, \dots, \mathbf{x}_N) \in \prod_{i=1}^N \mathcal{X}_i}{\argmin} \mathcal{L}(\mathbf{x}_1, \dots, \mathbf{x}_N, \bm{\lambda}^{k+1}, \bm{\mu}^{k+1}) - \sum_{i=1}^N \mathbf{x}_i^T \mathbf{u}_i^k + (\bm{\lambda}^{k+1})^T \mathbf{v}^k + (\bm{\mu}^{k+1})^T \mathbf{w}^k.
\end{aligned}
\end{equation}
Similarly, by the interpretation of $(\bm{\lambda}^{k+1}, \bm{\mu}^{k+1})$ in Lemma~\ref{lem: Dual}, we have:
\begin{equation}
\begin{aligned}
&\nabla_{\bm{\lambda}} \mathcal{L}(\mathbf{x}_1^{k+1}, \dots, \mathbf{x}_N^{k+1}, \bm{\lambda}^{k+1}, \bm{\mu}^{k+1}) + \underbrace{\Big[-\frac{1}{\rho} (\bm{\lambda}^{k+1} - \bm{\lambda}^k)\Big]}_{\mathbf{v}^k} = \mathbf{0}, \\[-6pt]
&\nabla_{\bm{\mu}} \mathcal{L}(\mathbf{x}_1^{k+1}, \dots, \mathbf{x}_N^{k+1}, \bm{\lambda}^{k+1}, \bm{\mu}^{k+1}) + \underbrace{\Big[-\frac{1}{\rho} (\bm{\mu}^{k+1} - \bm{\mu}^k)\Big]}_{\mathbf{w}^k} \in \mathcal{N}_{\mathbb{R}_+^{m_2}}(\bm{\mu}^{k+1}),
\end{aligned}
\end{equation}
which imply
\begin{equation}\label{eq: first-order optimality 2}
\begin{aligned}
&(\bm{\lambda}^{k+1}, \bm{\mu}^{k+1}) \in \underset{\bm{\lambda} \in \mathbb{R}^m, \bm{\mu} \in \mathbb{R}^M}{\argmax} \mathcal{L}(\mathbf{x}_1^{k+1}, \dots, \mathbf{x}_N^{k+1}, \bm{\lambda}, \bm{\mu}) - \sum_{i=1}^N (\mathbf{x}_i^{k+1})^T \mathbf{u}_i^k + \bm{\lambda}^T \mathbf{v}^k + \bm{\mu}^T \mathbf{w}^k.
\end{aligned}
\end{equation}
The first-order optimality conditions \eqref{eq: first-order optimality 1} and \eqref{eq: first-order optimality 2} together imply that 
$$
(\mathbf{x}_1^{k+1}, \dots, \mathbf{x}_N^{k+1}, \bm{\lambda}^{k+1}, \bm{\mu}^{k+1}) \in S^{-1}(\mathbf{u}_1^k, \dots, \mathbf{u}_N^k, \mathbf{v}^k, \mathbf{w}^k).
$$ \\
By \eqref{eq: 5 sequence convergence}, we have $\lim_{k \to \infty} (\mathbf{u}_1^k, \dots, \mathbf{u}_N^k, \mathbf{v}^k, \mathbf{w}^k) \to \mathbf{0}$. Choose in integer $\bar{k}$ such that, for all $k \geq \bar{k}$, $\lVert (\mathbf{u}_1^k, \dots, \mathbf{u}_N^k, \mathbf{v}^k, \mathbf{w}^k) \rVert_2 \leq \tau$, then by Assumption~\ref{assp: lip_inverse_mapping}, we have:
\begingroup
\begin{align}
&\sum_{i=1}^N \lVert \mathbf{x}_i^{k+1} - \mathbf{x}_i^* \rVert_2^2 + \lVert \bm{\lambda}^{k+1} - \bm{\lambda}^* \rVert_2^2 + \lVert \bm{\mu}^{k+1} - \bm{\mu}^* \rVert_2^2 \nonumber \\
\leq &a^2 \Big(\sum_{i=1}^N \lVert \mathbf{u}_i^k \rVert_2^2 + \lVert \mathbf{v}^k \rVert_2^2 + \lVert \mathbf{w}^k \rVert_2^2\Big) \nonumber \\
\leq &a^2 \Big(N A_{max}^2 \lVert \bm{\gamma}^{k+1} - \bm{\lambda}^{k+1} \rVert_2^2 + N L_{max}^2 \lVert \bm{\nu}^{k+1} - \bm{\mu}^{k+1} \rVert_2^2 + \frac{1}{\rho^2} \sum_{i=1}^N \lVert \mathbf{x}_i^{k+1} - \mathbf{x}_i^k \rVert_2^2 \nonumber \\
&\hspace*{15 pt}+ \frac{1}{\rho^2} \lVert \bm{\lambda}^{k+1} - \bm{\lambda}^k \rVert_2^2 + \frac{1}{\rho^2} \lVert \bm{\mu}^{k+1} - \bm{\mu}^k \rVert_2^2\Big) \nonumber \\
\leq &a^2 \Bigg[\frac{1}{\rho^2} \sum_{i=1}^N \lVert \mathbf{x}_i^{k+1} - \mathbf{x}_i^k \rVert_2^2 + N A_{max}^2 \lVert \bm{\gamma}^{k+1} - \bm{\lambda}^{k+1} \rVert_2^2 + N L_{max}^2 \lVert \bm{\nu}^{k+1} - \bm{\mu}^{k+1} \rVert_2^2 \nonumber \\
&\hspace*{15 pt}+ \frac{1}{\rho^2} \Big(\lVert \bm{\lambda}^{k+1} - \bm{\gamma}^{k+1} \rVert_2^2 + \lVert \bm{\gamma}^{k+1} - \bm{\lambda}^k \rVert_2^2\Big) \nonumber \\
&\hspace*{15 pt}+ \frac{1}{\rho^2} \Big(\lVert \bm{\mu}^{k+1} - \bm{\nu}^{k+1} \rVert_2^2 + \lVert \bm{\nu}^{k+1} - \bm{\mu}^k \rVert_2^2\Big)\Bigg] \nonumber \\
\leq &a^2 \Bigg[\frac{1}{\rho^2} \sum_{i=1}^N \lVert \mathbf{x}_i^{k+1} - \mathbf{x}_i^k \rVert_2^2 \nonumber \\
&\hspace*{15 pt}+ (N A_{max}^2 + \frac{1}{\rho^2}) \lVert \bm{\gamma}^{k+1} - \bm{\lambda}^{k+1} \rVert_2^2 + (N L_{max}^2 + \frac{1}{\rho^2}) \lVert \bm{\nu}^{k+1} - \bm{\mu}^{k+1} \rVert_2^2 \nonumber \\
&\hspace*{15 pt}+ \frac{1}{\rho^2} \lVert \bm{\gamma}^{k+1} - \bm{\lambda}^k \rVert_2^2 + \frac{1}{\rho^2} \lVert \bm{\nu}^{k+1} - \bm{\mu}^k \rVert_2^2 \Bigg] \nonumber \\
\leq &a^2(N \alpha^2 + \frac{1}{\rho^2}) \Big(\sum_{i=1}^N \lVert \mathbf{x}_i^{k+1} - \mathbf{x}_i^k \rVert_2^2 + \lVert \bm{\gamma}^{k+1} - \bm{\lambda}^{k+1} \rVert_2^2 + \lVert \bm{\nu}^{k+1} - \bm{\mu}^{k+1} \rVert_2^2 \nonumber \\
&\hspace*{80pt}+ \lVert \bm{\gamma}^{k+1} - \bm{\lambda}^k \rVert_2^2 + \lVert \bm{\nu}^{k+1} - \bm{\mu}^k \rVert_2^2\Big) \nonumber \\
\leq &\frac{a^2(N \alpha^2 + \frac{1}{\rho^2})}{\epsilon}\Bigg[\Big(\sum_{i=1}^N \lVert \mathbf{x}_i^k - \mathbf{x}_i^* \rVert_2^2 + \lVert \bm{\lambda}^k - \bm{\lambda}^* \rVert_2^2 + \lVert \bm{\mu}^k - \bm{\mu}^* \rVert_2^2\Big) \nonumber \\
&\hspace*{65pt}- \Big(\sum_{i=1}^N \lVert \mathbf{x}_i^{k+1} - \mathbf{x}_i^* \rVert_2^2 + \lVert \bm{\lambda}^{k+1} - \bm{\lambda}^* \rVert_2^2 + \lVert \bm{\mu}^{k+1} - \bm{\mu}^* \rVert_2^2\Big)\Bigg].
\end{align}
\endgroup
The last inequality is due to \eqref{eq: N-PCPM sequence inequality}, and $\alpha \coloneqq \max\{A_{max}, L_{max}\}$. We further derive
\begin{equation}
\begin{aligned}
&\sum_{i=1}^N \lVert \mathbf{x}_i^{k+1} - \mathbf{x}_i^* \rVert_2^2 + \lVert \bm{\lambda}^{k+1} - \bm{\lambda}^* \rVert_2^2 + \lVert \bm{\mu}^{k+1} - \bm{\mu}^* \rVert_2^2 \\
\leq &\theta^2 \Big(\sum_{i=1}^N \lVert \mathbf{x}_i^k - \mathbf{x}_i^* \rVert_2^2 + \lVert \bm{\lambda}^k - \bm{\lambda}^* \rVert_2^2 + \lVert \bm{\mu}^k - \bm{\mu}^* \rVert_2^2\Big),
\end{aligned}
\end{equation}
where $\theta = \sqrt{\frac{1}{1 + \beta}} < 1$ and $\beta = \frac{\epsilon}{a^2 (N \alpha^2 + \frac{1}{\rho^2})} > 0$.
\section{Proofs in Section~\ref{sec: AN-PCPM convergence}}
\subsection{Proof of Theorem~\ref{thm: sub-linear convergence}}\label{app: AN-PCPM Proof_1}
For all $k \geq 0$, we can equivalently write the update steps in Algorithm \ref{alg:AN-PCPM-M} and Algorithm \ref{alg:AN-PCPM-W} as
\begin{equation}
\mathbf{x}_i^{k+1} = 
\left\{
\begin{aligned}
&\underset{\mathbf{x}_i \in \mathcal{X}_i}{\argmin} f_i(\mathbf{x}_i) + (2\bm{\lambda}^{\hat{k}_i+1} - \bm{\lambda}^{\hat{k}_i})^T A_i \mathbf{x}_i + \frac{1}{2 \rho} \lVert \mathbf{x}_i - \mathbf{x}_i^{\hat{k}_i+1} \rVert^2, &\forall i \in \mathcal{A}_k \\
&\mathbf{x}_i^k, &\forall i \in \mathcal{A}_k^{\complement}
\end{aligned}
\right.,
\end{equation}
\begin{equation}
\bm{\lambda}^{k+1} = \bm{\lambda}^k + \rho \big(\sum_{i=1}^N A_i \mathbf{x}_i^{k+1} - \mathbf{b}\big),
\end{equation}
\begin{equation}
\hat{\bm{\gamma}} = \bm{\lambda}^{k+1} + \rho \big(\sum_{i=1}^N A_i \mathbf{x}_i^{k+1} - \mathbf{b}\big) = 2 \bm{\lambda}^{k+1} - \bm{\lambda}^k,
\end{equation}
where $\hat{k}_i$ is the last iteration when the main processor receives $\widehat{\mathbf{x}}_i$ from the worker processor $i \in \mathcal{A}_k$ before iteration $k$. For each worker processor $i \in \mathcal{A}_k^{\complement}$, we denote $\bar{k}_i \in (k-\tau, k)$ as the last iteration when the main processor receives $\hat{\mathbf{x}}_i$ from the worker processor $i$ before iteration $k$, and further denote $\bar{\bar{k}}_i \in [\bar{k}_i-\tau, \bar{k}_i)$ as the last iteration when the main processor receives $\widehat{\mathbf{x}}_i$ from the worker processor $i$ before iteration $\bar{k}_i$. We can rewrite the primal minimization step as
\begin{equation}
\mathbf{x}_i^{k+1} = 
\left\{
\begin{aligned}
&\arg \underset{\mathbf{x}_i}{\min} f_i(\mathbf{x}_i) + (2\bm{\lambda}^{\hat{k}_i+1} - \bm{\lambda}^{\hat{k}_i})^T A_i \mathbf{x}_i + \frac{1}{2 \rho} \lVert \mathbf{x}_i - \mathbf{x}_i^{\hat{k}_i+1} \rVert^2, &\forall i \in \mathcal{A}_k \\
\mathbf{x}_i^{\bar{k}_i+1} = &\arg \underset{\mathbf{x}_i}{\min} f_i(\mathbf{x}_i) + (2\bm{\lambda}^{\bar{\bar{k}}_i+1} - \bm{\lambda}^{\bar{\bar{k}}_i})^T A_i \mathbf{x}_i + \frac{1}{2 \rho} \lVert \mathbf{x}_i - \mathbf{x}_i^{\bar{\bar{k}}_i+1} \rVert^2, &\forall i \in \mathcal{A}_k^{\complement}
\end{aligned}
\right..
\end{equation}
At each iteration $k \geq 0$, for any $i \in \mathcal{A}_k$, applying Lemma~\ref{lem: Proximal Minimization Strong Convexity} with $\widehat{\mathbf{z}} = \mathbf{x}_i^{k+1}$, $\bar{\mathbf{z}} = \mathbf{x}_i^{\hat{k}_i+1}$, and $\mathbf{z} = \mathbf{x}_i^*$, we have:
\begingroup
\begin{align}
&f_i(\mathbf{x}_i^{k+1}) - f_i(\mathbf{x}_i^*) + \bm{\lambda}^T \big(A_i \mathbf{x}_i^{k+1} - A_i \mathbf{x}_i^*\big) + (\frac{\sigma_i}{2} + \frac{1}{2 \rho}) \lVert \mathbf{x}_i^{k+1} - \mathbf{x}_i^* \rVert_2^2 \nonumber \\
+ &\frac{1}{2 \rho} \lVert \mathbf{x}_i^{k+1} - \mathbf{x}_i^{\hat{k}_i+1} \rVert_2^2 - \frac{1}{2 \rho} \lVert \mathbf{x}_i^{\hat{k}_i+1} - \mathbf{x}_i^* \rVert_2^2 \nonumber \\
+ &(\bm{\lambda}^{k+1} - \bm{\lambda})^T \big(A_i \mathbf{x}_i^{k+1} - A_i \mathbf{x}_i^*\big) + (2\bm{\lambda}^{\hat{k}_i+1} - \bm{\lambda}^{\hat{k}_i} - \bm{\lambda}^{k+1})^T \big(A_i \mathbf{x}_i^{k+1} - A_i \mathbf{x}_i^*\big) \leq 0. \label{eq: inequality arrive}
\end{align}
\endgroup
At each iteration $k \geq 0$, for any $i \in \mathcal{A}_k^{\complement}$, applying Lemma~\ref{lem: Proximal Minimization Strong Convexity} with $\widehat{\mathbf{z}} = \mathbf{x}_i^{k+1}$, $\bar{\mathbf{z}} = \mathbf{x}_i^{\bar{\bar{k}}_i+1}$, and $\mathbf{z} = \mathbf{x}_i^*$, we have:
\begingroup
\begin{align}
&f_i(\mathbf{x}_i^{k+1}) - f_i(\mathbf{x}_i^*) + \bm{\lambda}^T \big(A_i \mathbf{x}_i^{k+1} - A_i \mathbf{x}_i^*\big) + (\frac{\sigma_i}{2} + \frac{1}{2 \rho}) \lVert \mathbf{x}_i^{k+1} - \mathbf{x}_i^* \rVert_2^2 \nonumber \\
+ &\frac{1}{2 \rho} \lVert \mathbf{x}_i^{k+1} - \mathbf{x}_i^{\bar{\bar{k}}_i+1} \rVert_2^2  - \frac{1}{2 \rho} \lVert \mathbf{x}_i^{\bar{\bar{k}}_i+1} - \mathbf{x}_i^* \rVert_2^2 \nonumber \\
+ &(\bm{\lambda}^{k+1} - \bm{\lambda})^T \big(A_i \mathbf{x}_i^{k+1} - A_i \mathbf{x}_i^*\big) + (2\bm{\lambda}^{\bar{\bar{k}}_i+1} - \bm{\lambda}^{\bar{\bar{k}}_i} - \bm{\lambda}^{k+1})^T \big(A_i \mathbf{x}_i^{k+1} - A_i \mathbf{x}_i^*\big) \leq 0. \label{eq: inequality not arrive}
\end{align}
\endgroup
Summing \eqref{eq: inequality arrive} over all $i \in \mathcal{A}_k$ and \eqref{eq: inequality not arrive} over all $i \in \mathcal{A}_k^{\complement}$ yields
\begingroup
\allowdisplaybreaks
\begin{align}
&\sum_{i=1}^N f_i(\mathbf{x}_i^{k+1}) - \sum_{i=1}^N f_i(\mathbf{x}_i^*) + \underbrace{\bm{\lambda}^T \sum_{i=1}^N \big(A_i \mathbf{x}_i^{k+1} - A_i \mathbf{x}_i^*\big)}_\text{(a)} + \frac{\sigma_{min}}{2} \sum_{i=1}^N \lVert \mathbf{x}_i^{k+1} - \mathbf{x}_i^* \rVert_2^2 \nonumber \\
+ &\underbrace{\frac{1}{2 \rho} \sum_{i=1}^N \lVert \mathbf{x}_i^{k+1} - \mathbf{x}_i^* \rVert_2^2}_\text{(b)} \nonumber \\
+ &\underbrace{\frac{1}{2 \rho} \Big[\sum_{i \in \mathcal{A}_k} \big(\lVert \mathbf{x}_i^{k+1} - \mathbf{x}_i^{\hat{k}_i+1} \rVert_2^2 - \lVert \mathbf{x}_i^{\hat{k}_i+1} - \mathbf{x}_i^* \rVert_2^2\big) + \sum_{i \in \mathcal{A}_k^{\complement}} \big(\lVert \mathbf{x}_i^{k+1} - \mathbf{x}_i^{\bar{\bar{k}}_i+1} \rVert_2^2 - \lVert \mathbf{x}_i^{\bar{\bar{k}}_i+1} - \mathbf{x}_i^* \rVert_2^2\big)\Big]}_\text{(c)} \nonumber \\
+ &\underbrace{(\bm{\lambda}^{k+1} - \bm{\lambda})^T \sum_{i=1}^N \big(A_i \mathbf{x}_i^{k+1} - A_i \mathbf{x}_i^*\big)}_\text{(d)} \nonumber \\
+ &\sum_{i \in \mathcal{A}_k} (2\bm{\lambda}^{\hat{k}_i+1} - \bm{\lambda}^{\hat{k}_i} - \bm{\lambda}^{k+1})^T \big(A_i \mathbf{x}_i^{k+1} - A_i \mathbf{x}_i^*\big) \nonumber \\
+ &\sum_{i \in \mathcal{A}_k^{\complement}} (2\bm{\lambda}^{\bar{\bar{k}}_i+1} - \bm{\lambda}^{\bar{\bar{k}}_i} - \bm{\lambda}^{k+1})^T \big(A_i \mathbf{x}_i^{k+1} - A_i \mathbf{x}_i^*\big) \nonumber \\
&\leq 0. \label{eq: AN-PCPM Big Inequality}
\end{align}
\endgroup
The term (a) can be rewritten as:
\begin{equation}\label{eq: AN-PCPM (a)}
(a) = \bm{\lambda}^T \big(\sum_{i=1}^N A_i \mathbf{x}_i^{k+1} - \sum_{i=1}^N A_i \mathbf{x}_i^*\big) = \bm{\lambda}^T \Big(\sum_{i=1}^N A_i \mathbf{x}_i^{k+1} - \mathbf{b}\Big).
\end{equation}
The term (b) + (c) can be rewritten as:
\begingroup
\begin{align}
(b) + (c) = &\frac{1}{2 \rho} \sum_{i \in \mathcal{A}_k} \big(\lVert \mathbf{x}_i^{k+1} - \mathbf{x}_i^* \rVert_2^2 + \lVert \mathbf{x}_i^{k+1} - \mathbf{x}_i^{\hat{k}_i+1} \rVert_2^2 - \lVert \mathbf{x}_i^{\hat{k}_i+1} - \mathbf{x}_i^* \rVert_2^2\big) \nonumber \\
+ &\frac{1}{2 \rho} \sum_{i \in \mathcal{A}_k^{\complement}} \big(\lVert \mathbf{x}_i^{k+1} - \mathbf{x}_i^* \rVert_2^2 + \lVert \mathbf{x}_i^{k+1} - \mathbf{x}_i^{\bar{\bar{k}}_i+1} \rVert_2^2 - \lVert \mathbf{x}_i^{\bar{\bar{k}}_i+1} - \mathbf{x}_i^* \rVert_2^2\big) \nonumber \\
& \geq 0. \label{eq: AN-PCPM (b)+(c)}
\end{align}
\endgroup
The term (d) can be rewritten as:
\begin{equation}\label{eq: AN-PCPM (d)}
(d) = (\bm{\lambda}^{k+1} - \bm{\lambda})^T \big(\sum_{i=1}^N A_i \mathbf{x}_i^{k+1} - \sum_{i=1}^N A_i \mathbf{x}_i^*\big) = \frac{1}{\rho} (\bm{\lambda}^{k+1} - \bm{\lambda})^T (\bm{\lambda}^{k+1} - \bm{\lambda}^k).
\end{equation}
We substitute \eqref{eq: AN-PCPM (a)}, \eqref{eq: AN-PCPM (b)+(c)} and \eqref{eq: AN-PCPM (d)} into \eqref{eq: AN-PCPM Big Inequality}, and sum it over $k = 0 \dots K-1$. Taking the average yields
\begingroup
\begin{align}
&\frac{1}{K} \sum_{k=0}^{K-1} \sum_{i=1}^N f_i(\mathbf{x}_i^{k+1}) - \sum_{i=1}^N f_i(\mathbf{x}_i^*) + \frac{1}{K} \bm{\lambda}^T \Big(\sum_{k=0}^{K-1} \sum_{i=1}^N A_i \mathbf{x}_i^{k+1} - \mathbf{b}\Big) \nonumber \\
\leq - &\frac{\sigma_{min}}{2 K} \sum_{k=0}^{K-1} \sum_{i=1}^N \lVert \mathbf{x}_i^{k+1} - \mathbf{x}_i^* \rVert_2^2 \nonumber \\
- &\frac{1}{\rho K} \underbrace{\sum_{k=0}^{K-1} (\bm{\lambda}^{k+1} - \bm{\lambda})^T (\bm{\lambda}^{k+1} - \bm{\lambda}^k)}_\text{(e)} \nonumber \\
+ &\frac{1}{K} \underbrace{\sum_{k=0}^{K-1} \sum_{i \in \mathcal{A}_k} (\bm{\lambda}^{\hat{k}_i} + \bm{\lambda}^{k+1} - 2\bm{\lambda}^{\hat{k}_i+1})^T \big(A_i \mathbf{x}_i^{k+1} - A_i \mathbf{x}_i^*\big)}_\text{(f)} \nonumber \\
+ &\frac{1}{K} \underbrace{\sum_{k=0}^{K-1} \sum_{i \in \mathcal{A}_k^{\complement}} (\bm{\lambda}^{\bar{\bar{k}}_i} + \bm{\lambda}^{k+1} - 2\bm{\lambda}^{\bar{\bar{k}}_i+1})^T \big(A_i \mathbf{x}_i^{k+1} - A_i \mathbf{x}_i^*\big)}_\text{(g)}. \label{eq: AN-PCPM Average Inequality}
\end{align}
\endgroup
The term (e) in \eqref{eq: AN-PCPM Average Inequality} can be rewritten as:
\begin{equation}\label{eq: AN-PCPM (e)}
\begin{aligned}
(e) = &\frac{1}{2} \sum_{k=0}^{k-1} \big(\lVert \bm{\lambda}^{k+1} - \bm{\lambda} \rVert_2^2 - \lVert \bm{\lambda} - \bm{\lambda}^k \rVert_2^2 + \lVert \bm{\lambda}^{k+1} - \bm{\lambda}^k \rVert_2^2\big) \\
= &\frac{1}{2} \lVert \bm{\lambda}^K - \bm{\lambda} \rVert_2^2 - \frac{1}{2} \lVert \bm{\lambda}^0 - \bm{\lambda} \rVert_2^2 + \frac{1}{2} \sum_{k=0}^{K-1} \lVert \bm{\lambda}^{k+1} - \bm{\lambda}^k \rVert_2^2.
\end{aligned}
\end{equation}
The term (f) in \eqref{eq: AN-PCPM Average Inequality} can be bounded as:
\begingroup
\begin{align}
&\sum_{k=0}^{K-1} \sum_{i \in \mathcal{A}_k} (\bm{\lambda}^{\hat{k}_i} + \bm{\lambda}^{k+1} - 2\bm{\lambda}^{\hat{k}_i+1})^T \big(A_i \mathbf{x}_i^{k+1} - A_i \mathbf{x}_i^*\big) \nonumber \\
=&\sum_{k=0}^{K-1} \sum_{i \in \mathcal{A}_k} (\bm{\lambda}^{\hat{k}_i} - \bm{\lambda}^{\hat{k}_i+1})^T \big(A_i \mathbf{x}_i^{k+1} - A_i \mathbf{x}_i^*\big) + \sum_{k=0}^{K-1} \sum_{i \in \mathcal{A}_k} (\bm{\lambda}^{k+1} - \bm{\lambda}^{\hat{k}_i+1})^T \big(A_i \mathbf{x}_i^{k+1} - A_i \mathbf{x}_i^*\big) \nonumber \\
= &\sum_{k=0}^{K-1} \sum_{i \in \mathcal{A}_k} \sum_{l=\hat{k}_i}^{\hat{k}_i} (\bm{\lambda}^l - \bm{\lambda}^{l+1})^T \big(A_i \mathbf{x}_i^{k+1} - A_i \mathbf{x}_i^*\big) \nonumber \\
+ &\sum_{k=0}^{K-1} \sum_{i \in \mathcal{A}_k} \sum_{l=\hat{k}_i+1}^k (\bm{\lambda}^{l+1} - \bm{\lambda}^l)^T \big(A_i \mathbf{x}_i^{k+1} - A_i \mathbf{x}_i^*\big) \nonumber \\
\leq &\sum_{i \in \mathcal{A}_k} \sum_{k=0}^{K-1} \sum_{l=\hat{k}_i}^{\hat{k}_i} \Big(\frac{1}{2 \delta^2}\lVert \bm{\lambda}^l - \bm{\lambda}^{l+1} \rVert_2^2 + \frac{\delta^2 \lVert A_i \rVert_2^2}{2}\lVert \mathbf{x}_i^{k+1} - \mathbf{x}_i^* \rVert_2^2\Big) \nonumber \\
+ &\sum_{i \in \mathcal{A}_k} \sum_{k=0}^{K-1} \sum_{l=\hat{k}_i+1}^k \Big(\frac{1}{2 \delta^2}\lVert \bm{\lambda}^{l+1} - \bm{\lambda}^l \rVert_2^2 + \frac{\delta^2 \lVert A_i \rVert_2^2}{2}\lVert \mathbf{x}_i^{k+1} - \mathbf{x}_i^* \rVert_2^2\Big) \nonumber \\
\leq &\sum_{i=1}^N \sum_{k=0}^{K-1} (\tau - 1) \Big(\frac{1}{2 \delta^2}\lVert \bm{\lambda}^k - \bm{\lambda}^{k+1} \rVert_2^2 + \frac{\delta^2 \lVert A_i \rVert_2^2}{2}\lVert \mathbf{x}_i^{k+1} - \mathbf{x}_i^* \rVert_2^2\Big) \nonumber \\
+ &\sum_{i=1}^N \sum_{k=0}^{K-1} (\tau - 1) \Big(\frac{1}{2 \delta^2}\lVert \bm{\lambda}^{k+1} - \bm{\lambda}^k \rVert_2^2 + \frac{\delta^2 \lVert A_i \rVert_2^2}{2}\lVert \mathbf{x}_i^{k+1} - \mathbf{x}_i^* \rVert_2^2\Big) \nonumber \\
\leq &\frac{(\tau - 1)N}{\delta^2} \sum_{k=0}^{K-1} \lVert \bm{\lambda}^{k+1} - \bm{\lambda}^k \rVert_2^2 + (\tau - 1) \delta^2 A_{\max}^2 \sum_{k=0}^{K-1} \sum_{i=1}^N \lVert \mathbf{x}_i^{k+1} - \mathbf{x}_i^* \rVert_2^2, \label{eq: AN-PCPM (f)}
\end{align}
\endgroup
where the first inequality is obtained by \eqref{eq: N-PCPM Young's Inequality}, and the second inequality is due to the fact that the term $\lVert \bm{\lambda}^{k+1} - \bm{\lambda}^k \rVert_2^2$ does not appear more than $\tau - 1$ times for each iteration $k$.
\par Similarly, the term (g) in \eqref{eq: AN-PCPM Average Inequality} can be bounded as:
\begingroup
\begin{align}
&\sum_{k=0}^{K-1} \sum_{i \in \mathcal{A}_k^{\complement}} (\bm{\lambda}^{\bar{\bar{k}}_i} + \bm{\lambda}^{k+1} - 2\bm{\lambda}^{\bar{\bar{k}}_i+1})^T \big(A_i \mathbf{x}_i^{k+1} - A_i \mathbf{x}_i^*\big) \nonumber \\
= &\sum_{k=0}^{K-1} \sum_{i \in \mathcal{A}_k^{\complement}} (\bm{\lambda}^{\bar{\bar{k}}_i} - \bm{\lambda}^{\bar{\bar{k}}_i+1})^T \big(A_i \mathbf{x}_i^{k+1} - A_i \mathbf{x}_i^*\big) + \sum_{k=0}^{K-1} \sum_{i \in \mathcal{A}_k^{\complement}} (\bm{\lambda}^{\bar{k}_i+1} - \bm{\lambda}^{\bar{\bar{k}}_i+1})^T \big(A_i \mathbf{x}_i^{k+1} - A_i \mathbf{x}_i^*\big) \nonumber \\
+ &\sum_{k=0}^{K-1} \sum_{i \in \mathcal{A}_k^{\complement}} (\bm{\lambda}^{k+1} - \bm{\lambda}^{\bar{k}_i+1})^T \big(A_i \mathbf{x}_i^{k+1} - A_i \mathbf{x}_i^*\big) \nonumber \\
= &\sum_{k=0}^{K-1} \sum_{i \in \mathcal{A}_k^{\complement}} \sum_{l=\bar{\bar{k}}_i}^{\bar{\bar{k}}_i} (\bm{\lambda}^l - \bm{\lambda}^{l+1})^T \big(A_i \mathbf{x}_i^{k+1} - A_i \mathbf{x}_i^*\big) \nonumber \\
+ &\sum_{k=0}^{K-1} \sum_{i \in \mathcal{A}_k^{\complement}} \sum_{l=\bar{\bar{k}}_i+1}^{\bar{k}_i} (\bm{\lambda}^{l+1} - \bm{\lambda}^l)^T \big(A_i \mathbf{x}_i^{k+1} - A_i \mathbf{x}_i^*\big) \nonumber \\
+ &\sum_{k=0}^{K-1} \sum_{i \in \mathcal{A}_k^{\complement}} \sum_{l=\bar{k}_i+1}^k (\bm{\lambda}^{l+1} - \bm{\lambda}^l)^T \big(A_i \mathbf{x}_i^{k+1} - A_i \mathbf{x}_i^*\big) \nonumber \\
\leq &\sum_{i \in \mathcal{A}_k^{\complement}} \sum_{k=0}^{K-1} \sum_{l=\bar{\bar{k}}_i}^{\bar{\bar{k}}_i} \Big(\frac{1}{2 \delta^2}\lVert \bm{\lambda}^l - \bm{\lambda}^{l+1} \rVert_2^2 + \frac{\delta^2 \lVert A_i \rVert_2^2}{2}\lVert \mathbf{x}_i^{k+1} - \mathbf{x}_i^* \rVert_2^2\Big) \nonumber \\
+ &\sum_{i \in \mathcal{A}_k^{\complement}} \sum_{k=0}^{K-1} \sum_{l=\bar{\bar{k}}_i+1}^{\bar{k}_i} \Big(\frac{1}{2 \delta^2}\lVert \bm{\lambda}^{l+1} - \bm{\lambda}^l \rVert_2^2 + \frac{\delta^2 \lVert A_i \rVert_2^2}{2}\lVert \mathbf{x}_i^{k+1} - \mathbf{x}_i^* \rVert_2^2\Big) \nonumber \\
+ &\sum_{i \in \mathcal{A}_k^{\complement}} \sum_{k=0}^{K-1} \sum_{l=\bar{k}_i+1}^k \Big(\frac{1}{2 \delta^2}\lVert \bm{\lambda}^{l+1} - \bm{\lambda}^l \rVert_2^2 + \frac{\delta^2 \lVert A_i \rVert_2^2}{2}\lVert \mathbf{x}_i^{k+1} - \mathbf{x}_i^* \rVert_2^2\Big) \nonumber \\
\leq &\sum_{i=1}^N \sum_{k=0}^{K-1} (\tau - 1) \Big(\frac{1}{2 \delta^2}\lVert \bm{\lambda}^k - \bm{\lambda}^{k+1} \rVert_2^2 + \frac{\delta^2 \lVert A_i \rVert_2^2}{2}\lVert \mathbf{x}_i^{k+1} - \mathbf{x}_i^* \rVert_2^2\Big) \nonumber \\
+ &\sum_{i=1}^N \sum_{k=0}^{K-1} (\tau - 1) \Big(\frac{1}{2 \delta^2}\lVert \bm{\lambda}^{k+1} - \bm{\lambda}^k \rVert_2^2 + \frac{\delta^2 \lVert A_i \rVert_2^2}{2}\lVert \mathbf{x}_i^{k+1} - \mathbf{x}_i^* \rVert_2^2\Big) \nonumber \\
+ &\sum_{i=1}^N \sum_{k=0}^{K-1} (\tau - 1) \Big(\frac{1}{2 \delta^2}\lVert \bm{\lambda}^{k+1} - \bm{\lambda}^k \rVert_2^2 + \frac{\delta^2 \lVert A_i \rVert_2^2}{2}\lVert \mathbf{x}_i^{k+1} - \mathbf{x}_i^* \rVert_2^2\Big) \nonumber \\
\leq &\frac{3 (\tau - 1)N}{2 \delta^2} \sum_{k=0}^{K-1} \lVert \bm{\lambda}^{k+1} - \bm{\lambda}^k \rVert_2^2 + \frac{3 (\tau - 1) \delta^2}{2} A_{\max}^2 \sum_{k=0}^{K-1} \sum_{i=1}^N \lVert \mathbf{x}_i^{k+1} - \mathbf{x}_i^* \rVert_2^2. \label{eq: AN-PCPM (g)}
\end{align}
\endgroup
By substituting \eqref{eq: AN-PCPM (e)}, \eqref{eq: AN-PCPM (f)} and \eqref{eq: AN-PCPM (g)} into \eqref{eq: AN-PCPM Average Inequality} and denoting
$$
\bar{\mathbf{x}}_i^K = \frac{1}{K} \sum_{k=0}^{K-1} \mathbf{x}_i^{k+1},
$$
for all $i = 1 \dots N$, we have:
\begingroup
\begin{align}
&\sum_{i=1}^N f_i(\bar{\mathbf{x}}_i^K) - \sum_{i=1}^N f_i(\mathbf{x}_i^*) + \bm{\lambda}^T \Big(\sum_{i=1}^N A_i \bar{\mathbf{x}}_i^K - \mathbf{b}\Big) \nonumber \\
\leq &\frac{1}{K} \sum_{k=0}^{K-1} \sum_{i=1}^N f_i(\mathbf{x}_i^{k+1}) - \sum_{i=1}^N f_i(\mathbf{x}_i^*) + \frac{1}{K} \bm{\lambda}^T \sum_{k=0}^{K-1} \Big(\sum_{i=1}^N A_i \mathbf{x}_i^{k+1} - \mathbf{b}\Big) \nonumber \\
\leq &\frac{1}{2 \rho K} \lVert \bm{\lambda}^0 - \bm{\lambda} \rVert_2^2 - \frac{1}{2 \rho K} \lVert \bm{\lambda}^K - \bm{\lambda} \rVert_2^2 \nonumber \\
+ &(-\frac{1}{2 \rho K} + \frac{5 (\tau - 1) N}{2 \delta^2 K}) \sum_{k=0}^{K-1} \lVert \bm{\lambda}^{k+1} - \bm{\lambda}^k \rVert_2^2 \nonumber \\
+ &(-\frac{\sigma_{min}}{2 K} + \frac{5 (\tau - 1) \delta^2 A_{\max}^2}{2 K}) \sum_{k=0}^{K-1} \sum_{i=1}^N \lVert \mathbf{x}_i^{k+1} - \mathbf{x}_i^* \rVert_2^2,
\end{align}
\endgroup
where the first inequality is due to the convexity of $f_i$ for all $i = 1 \dots N$. By choosing $\delta^2 \leq \frac{\sigma_{min}}{5 (\tau - 1) A_{\max}^2}$ and $\rho \leq \frac{\delta^2}{5 (\tau - 1) N}$, which implies
$$
\rho \leq \frac{\sigma_{min}}{25 N (\tau - 1)^2 A_{\max}^2},
$$
we derive:
\begin{equation}
\sum_{i=1}^N f_i(\bar{\mathbf{x}}_i^K) - \sum_{i=1}^N f_i(\mathbf{x}_i^*) + \bm{\lambda}^T \Big(\sum_{i=1}^N A_i \bar{\mathbf{x}}_i^K - \mathbf{b}\Big) \leq \frac{1}{2 \rho K} \lVert \bm{\lambda}^0 - \bm{\lambda} \rVert_2^2.
\end{equation}
Let $\bm{\lambda} = \bm{\lambda}^* + \frac{\sum_{i=1}^N A_i \bar{\mathbf{x}}^K - \mathbf{b}}{\lVert \sum_{i=1}^N A_i \bar{\mathbf{x}}^K - \mathbf{b} \rVert_2}$, and note that by the duality theory, we have:
\begin{equation}
\sum_{i=1}^N f_i(\bar{\mathbf{x}}_i^K) - \sum_{i=1}^N f_i(\mathbf{x}_i^*) + (\bm{\lambda}^*)^T \Big(\sum_{i=1}^N A_i \bar{\mathbf{x}}_i^K - \mathbf{b}\Big) \geq 0.
\end{equation}
Then, we further derive:
\begingroup
\begin{align}
&\big\lVert \sum_{i=1}^N A_i \bar{\mathbf{x}}^K - \mathbf{b}\big\rVert_2 \nonumber \\
\leq &\sum_{i=1}^N f_i(\bar{\mathbf{x}}_i^K) - \sum_{i=1}^N f_i(\mathbf{x}_i^*) + (\bm{\lambda}^*)^T \Big(\sum_{i=1}^N A_i \bar{\mathbf{x}}_i^K - \mathbf{b}\Big) + \big\lVert \sum_{i=1}^N A_i \bar{\mathbf{x}}^K - \mathbf{b}\big\rVert_2 \nonumber \\
\leq &\frac{1}{2 \rho K} \Bigg\lVert \bm{\lambda}^0 - \Big(\bm{\lambda}^* + \frac{\sum_{i=1}^N A_i \bar{\mathbf{x}}^K - \mathbf{b}}{\lVert \sum_{i=1}^N A_i \bar{\mathbf{x}}^K - \mathbf{b}\rVert_2}\Big) \Bigg\rVert_2^2,
\end{align}
\endgroup
which implies
$$
\big\lVert \sum_{i=1}^N A_i \bar{\mathbf{x}}^K - \mathbf{b}\big\rVert_2 \leq \frac{1}{K} \Big[\frac{1}{2 \rho} \underset{\lVert \bm{\gamma} \rVert_2 \leq 1}{\max} \lVert \bm{\lambda}^0 - \bm{\lambda}^* - \bm{\gamma} \rVert_2^2\Big] \overset{\Delta}{=} \frac{C_1}{K}.
$$
On the other hand, let $\bm{\lambda} = \bm{\lambda}^*$, and note that:
\begin{equation}
\begin{aligned}
&\sum_{i=1}^N f_i(\bar{\mathbf{x}}_i^K) - \sum_{i=1}^N f_i(\mathbf{x}_i^*) + (\bm{\lambda}^*)^T \Big(\sum_{i=1}^N A_i \bar{\mathbf{x}}_i^K - \mathbf{b}\Big) \\
\geq &\Big\lvert \sum_{i=1}^N f_i(\bar{\mathbf{x}}_i^K) - \sum_{i=1}^N f_i(\mathbf{x}_i^*) \Big\rvert - \lVert \bm{\lambda}^* \rVert_2 \cdot \lVert \sum_{i=1}^N A_i \bar{\mathbf{x}}_i^K - \mathbf{b}\rVert_2.
\end{aligned}
\end{equation}
Then, we have:
\begin{equation}
\Big\lvert \sum_{i=1}^N f_i(\bar{\mathbf{x}}_i^K) - \sum_{i=1}^N f_i(\mathbf{x}_i^*) \Big\rvert \leq \lVert \bm{\lambda}^* \rVert_2 \cdot \lVert \sum_{i=1}^N A_i \bar{\mathbf{x}}_i^K - \mathbf{b}\rVert_2 + \frac{1}{K} (\frac{1}{2 \rho} \lVert \bm{\lambda}^0 - \bm{\lambda}^* \rVert_2^2) \overset{\Delta}{=} \frac{\delta_{\bm{\lambda}} C_1 + C_2}{K},
\end{equation}
where $\delta_{\bm{\lambda}} = \lVert \bm{\lambda}^* \rVert_2$ and $C_2 = \frac{1}{2 \rho} \lVert \bm{\lambda}^0 - \bm{\lambda}^* \rVert_2^2$.
\end{appendix}
\end{document}